\PassOptionsToPackage{unicode}{hyperref}
\PassOptionsToPackage{hyphens}{url}
\PassOptionsToPackage{dvipsnames,svgnames,x11names}{xcolor}
\documentclass[
  12pt]{article}

\usepackage{amsmath,amssymb,amsthm,amsfonts}
\usepackage{iftex}
\ifPDFTeX
  \usepackage[T1]{fontenc}
  \usepackage[utf8]{inputenc}
  \usepackage{textcomp} 
\else 
  \usepackage{unicode-math}
  \defaultfontfeatures{Scale=MatchLowercase}
  \defaultfontfeatures[\rmfamily]{Ligatures=TeX,Scale=1}
\fi
\usepackage{lmodern}
\ifPDFTeX\else  
\fi
\IfFileExists{upquote.sty}{\usepackage{upquote}}{}
\IfFileExists{microtype.sty}{
  \usepackage[]{microtype}
  \UseMicrotypeSet[protrusion]{basicmath} 
}{}
\makeatletter
\@ifundefined{KOMAClassName}{
  \IfFileExists{parskip.sty}{%
    \usepackage{parskip}
  }{
    \setlength{\parindent}{0pt}
    \setlength{\parskip}{6pt plus 2pt minus 1pt}}
}{
  \KOMAoptions{parskip=half}}
\makeatother
\usepackage{xcolor}
\makeatletter
\ifx\paragraph\undefined\else
  \let\oldparagraph\paragraph
  \renewcommand{\paragraph}{
    \@ifstar
      \xxxParagraphStar
      \xxxParagraphNoStar
  }
  \newcommand{\xxxParagraphStar}[1]{\oldparagraph*{#1}\mbox{}}
  \newcommand{\xxxParagraphNoStar}[1]{\oldparagraph{#1}\mbox{}}
\fi
\ifx\subparagraph\undefined\else
  \let\oldsubparagraph\subparagraph
  \renewcommand{\subparagraph}{
    \@ifstar
      \xxxSubParagraphStar
      \xxxSubParagraphNoStar
  }
  \newcommand{\xxxSubParagraphStar}[1]{\oldsubparagraph*{#1}\mbox{}}
  \newcommand{\xxxSubParagraphNoStar}[1]{\oldsubparagraph{#1}\mbox{}}
\fi
\makeatother

\usepackage{longtable,booktabs,array}
\usepackage{calc} 
\usepackage{etoolbox}
\makeatletter
\patchcmd\longtable{\par}{\if@noskipsec\mbox{}\fi\par}{}{}
\makeatother
\IfFileExists{footnotehyper.sty}{\usepackage{footnotehyper}}{\usepackage{footnote}}
\makesavenoteenv{longtable}
\usepackage{graphicx}
\makeatletter
\def\maxwidth{\ifdim\Gin@nat@width>\linewidth\linewidth\else\Gin@nat@width\fi}
\def\maxheight{\ifdim\Gin@nat@height>\textheight\textheight\else\Gin@nat@height\fi}
\makeatother
\setkeys{Gin}{width=\maxwidth,height=\maxheight,keepaspectratio}
\makeatletter
\def\fps@figure{htbp}
\makeatother

\makeatletter
\@ifpackageloaded{caption}{}{\usepackage{caption}}
\AtBeginDocument{%
\ifdefined\contentsname
  \renewcommand*\contentsname{Table of contents}
\else
  \newcommand\contentsname{Table of contents}
\fi
\ifdefined\listfigurename
  \renewcommand*\listfigurename{List of Figures}
\else
  \newcommand\listfigurename{List of Figures}
\fi
\ifdefined\listtablename
  \renewcommand*\listtablename{List of Tables}
\else
  \newcommand\listtablename{List of Tables}
\fi
\ifdefined\figurename
  \renewcommand*\figurename{Figure}
\else
  \newcommand\figurename{Figure}
\fi
\ifdefined\tablename
  \renewcommand*\tablename{Table}
\else
  \newcommand\tablename{Table}
\fi
}
\@ifpackageloaded{float}{}{\usepackage{float}}
\floatstyle{ruled}
\@ifundefined{c@chapter}{\newfloat{codelisting}{h}{lop}}{\newfloat{codelisting}{h}{lop}[chapter]}
\floatname{codelisting}{Listing}

\makeatother
\makeatletter
\@ifpackageloaded{caption}{}{\usepackage{caption}}
\@ifpackageloaded{subcaption}{}{\usepackage{subcaption}}
\makeatother

\ifLuaTeX
  \usepackage{selnolig}  
\fi
\usepackage[]{natbib}
\usepackage{bookmark}

\IfFileExists{xurl.sty}{\usepackage{xurl}}{} 
\hypersetup{
  pdftitle={Horseshoe Predictive Inference},
  pdfauthor={Percy S. Zhai; Veronika Ro\v{c}kov\'{a}},
  pdfkeywords={Horseshoe prior;
        Kullback-Leibler divergence;
        Minimax optimality;
        Shrinkage prior;
        Sparsity.},
  colorlinks=true,
  linkcolor={blue},
  filecolor={Maroon},
  citecolor={Blue},
  urlcolor={Blue},
  pdfcreator={LaTeX}}

\newcommand{\anon}{1}

\usepackage{booktabs}
\usepackage{bm}
\usepackage{algorithm}
\usepackage{algpseudocode}
\usepackage{subcaption}
\usepackage{setspace}

\numberwithin{equation}{section}
\theoremstyle{plain}

\newtheorem{theorem}{Theorem}[section]
\newtheorem{assumption}{Assumption}[section]
\newtheorem{lemma}{Lemma}[section]

\newtheorem{corollary}{Corollary}[section]
\theoremstyle{remark}
\newtheorem{remark}{Remark}[section]

\newcommand*\diff{\mathop{}\!\mathrm{d}}
\newcommand{\one}[1]{\mathbf{1}\left\{ #1 \right\}}

\def\argmax{\mathop{\arg\max}}

\def\P{\mathbb{P}}
\def\E{\mathbb{E}}
\def\R{\mathbb{R}}

\begin{document}

\def\spacingset#1{\renewcommand{\baselinestretch}%
{#1}\small\normalsize} \spacingset{1}


\if1\anon
{
  \title{\bf Horseshoe Predictive Inference}
  \author{Percy S. Zhai \quad and \quad Veronika Ro\v{c}kov\'{a}\thanks{Veronika Ro\v{c}kov\'{a}'s research is partially supported by Grant NSF/DMS-2515709.}\\
  Booth School of Business, The University of Chicago\\
  Chicago, IL 60601, USA}
  \maketitle
} \fi

\if0\anon
{
  \bigskip
  \bigskip
  \bigskip
  \begin{center}
    {\LARGE\bf Horseshoe Predictive Inference}
  \end{center}
  \medskip
} \fi

\bigskip
\begin{abstract}
Predictive inference in the sparse Gaussian sequence model has received considerably less attention than its non-sparse, finite-sample counterpart.
Existing work has largely been confined to discrete mixture priors.
In this paper, we study predictive inference under a widely used continuous mixture prior, the Horseshoe.
We provide new theoretical results establishing exact asymptotic minimax optimality of the predictive Bayes estimator when the sparsity level is known.
Furthermore, through a Gaussian-mixture representation of the posterior predictive density (which we term Horseshoe spectroscopy), the phase-transition in the local shrinkage scale is inherited by the predictive mechanism, producing behavior similar to that of previous thresholding/switching estimators.
When sparsity is unknown, we adopt a fully Bayesian approach using a hierarchical Horseshoe prior and show that it performs adaptive, as opposed to manual, switching.
Under a theta-min condition, the resulting predictive risk admits an upper bound over a restricted parameter class that is sharper than the minimax rate over the full class.
We demonstrate the practical value of predictive Horseshoe shrinkage on data such as images and time series that can be naturally modeled as sparse Gaussian sequences.
We illustrate this approach on facial recognition across varying facial expressions and study region-wise atypical brain lateralization in autism spectrum disorder.
\end{abstract}

\noindent%
\vfill

\newpage
\spacingset{1.8} 

\section{Introduction}

Predictive inference for \emph{sparse} Gaussian sequence model under the Kullback--Leibler (KL) loss is a foundational problem that remains relatively understudied.
One observes $\bm Y \sim N_n(\bm\theta, I_n)$, and seeks a predictive density $\hat p(\tilde{\bm y}\mid \bm y)$ for the future vector $\tilde{\bm Y} \sim N_n(\bm\theta, r I_n)$ when the unknown mean $\bm\theta$ is sparse, in a sense that
$\bm\theta \in \Theta_n(s_n)=\{\bm\theta\in\R^n:\|\bm\theta\|_0\le s_n\}$.
The quality of the predictive density is quantified by the KL loss,
\[
L_n(\bm\theta, \hat p(\cdot \mid \bm y)) = \int \pi(\tilde{\bm y} \mid \bm\theta) \log \frac{\pi(\tilde{\bm y} \mid \bm\theta)}{\hat p(\tilde{\bm y} \mid \bm y)} \diff \tilde{\bm y},
\]
and its corresponding predictive risk, $\rho_n(\bm\theta, \hat p) = \E_{\bm Y\mid \bm\theta}[L_n(\bm\theta, \hat p(\cdot \mid \bm Y))]$.
When $s_n/n\to 0$, the minimax predictive KL risk over $\Theta_n(s_n)$ is asymptotically of order
\begin{equation}\label{eq:minimax.rate}
    \inf_{\hat p} \sup_{\bm\theta \in \Theta_n(s_n)}\rho_n(\bm\theta, \hat p) \sim \frac{s_n}{1+r}\log\frac{n}{s_n}.
\end{equation}
This minimax risk is first identified in \cite{mukherjee2015exact}.
For brevity, we define $v = (1 + r^{-1})^{-1}$ throughout this paper.
This sparse sequence formulation is a standard asymptotic model for high-dimensional sparsity \citep{johnstone2004needles,mccullagh2018statistical}. At the same time, predictive density estimation is not merely a reformulation of point estimation.
It has been established in the literature that shrinkage can substantially improve predictive performance even in unconstrained Gaussian models \citep{aitchison1975goodness, komaki2001shrinkage, george2006improved, brown2008admissible}, while in sparse models even good point estimators may lead to suboptimal plug-in predictive densities \citep{mukherjee2015exact}.

Outside the fully Bayesian framework, \cite{mukherjee2015exact} constructed a thresholding estimator that uses the uniform-prior predictive density, when the past observation is large, and a predictive density induced by a sparse univariate cluster prior, otherwise.
Estimators with a genuine Bayesian motivation emerged more recently, mainly focusing on discrete  grid priors or discrete mixture priors.
Notably, \cite{mukherjee2022minimax} established asymptotic minimaxity based on a nonadaptive spike-and-slab prior with a finitely supported uniform slab. The spike-and-slab viewpoint was further investigated by \cite{rockova2023adaptive} who obtained rate-minimaxity  results for practically used Laplace spike-and-slab mixtures and, more importantly,
for hierarchical variants to establish the first adaptive rate-minimax guarantees when the sparsity level is unknown. In addition, adaptive rate results in \cite{rockova2023adaptive} cover not only  sparse normal means but also high-dimensional sparse regression.
This paper continues the investigation in the context of another widely used shrinkage prior which is a continuous (as opposed to discrete) mixture.
The Horseshoe prior \citep{carvalho2009handling} has since risen to be one of the more popular alternatives to Bayesian variable selection with spike-and-slab priors.

Continuous shrinkage priors were introduced as continuous approximations to Bayesian model averaging, with one of the earliest  constructions (that predates the Horseshoe) occurring in \cite{denison2012bayesian}.
The Horseshoe prior has acquired many variants and extensions \citep{bhadra2016horseshoe+, piironen2017sparsity, lee2024tail}
that have occurred across a broad range of applications  including graphical models \citep{li2019graphical,li2021joint,sagar2024precision},
function-space and nonparametric models \citep{shin2020functional,agapiou2024heavy,agapiou2025heavy},
and deep Bayesian architectures \citep{ghosh2019model,bhadra2020horseshoe,castillo2025deep}. Theoretical properties have also been studied.
In the sparse normal means problem,  desirable shrinkage behavior has been described in \cite{carvalho2010horseshoe} and then rigorously quantified by  \cite{van2014horseshoe}, and \cite{van2017adaptive}.
\cite{datta2013asymptotic} studies asymptotic Bayes-risk optimality properties under sparse multiple testing.
In regression, \cite{bhadra2019prediction} shows that the Horseshoe shrinkage can also improve plug-in prediction.  
Despite an extensive literature on point estimation and posterior concentration, formal predictive inference theory for the Horseshoe has been missing.
This gap is more than just technical, since good estimation behavior does not automatically imply good predictive performance. Most existing results provide high-probability posterior concentration bounds while our analysis studies the KL risk of the posterior predictive Bayes estimator. We provide new insights into how global-local Horseshoe shrinkage manifests itself in the posterior predictive distribution, and how it  affects the KL loss.
The main focus of this paper is to explore the Horseshoe prior, as well as its hierarchical version, and to develop sharp predictive guarantees.

Let $\tau>0$ be the \emph{global shrinkage scale} that captures the overall sparsity information across all dimensions.
With $\tau$ fixed, the multivariate Horseshoe prior can thus be written as a separable product,
$\pi(\bm{\theta} \mid \tau) = \prod_{i=1}^n \pi(\theta_i \mid \tau)$,
with each univariate prior $\pi(\theta_i \mid \tau)$ being a continuous variance mixture of Gaussian,
\begin{align}\label{eq:HS.prior}
        \theta_i \mid \lambda_i , \tau \sim N(0, \lambda_i^2\tau^2),\qquad
        \lambda_i \sim C^+(0,1).
\end{align}
Here, $C^+(0,1)$ is a half-Cauchy distribution, with pdf $\nu(\lambda_i) = \frac{2}{\pi}\frac{1}{1+\lambda_i^2}\one{\lambda_i>0}$.
For each dimension, $\lambda_i>0$ is the \emph{local shrinkage scale} that characterizes univariate contraction.

Our first finding is a Horseshoe-specific analysis of the posterior predictive density for fixed global shrinkage scale $\tau$.
Each univariate predictive density can be represented as a continuous mixture of conjugate Gaussian predictive densities indexed by the local shrinkage scale $\lambda$.
The corresponding KL loss (and subsequently the predictive KL risk) can be tightly upper bounded by a similar spectral representation.
We refer to this continuous mixture viewpoint as \emph{Horseshoe spectroscopy}.
This view is not only intuitive but also useful in the delicate control of the predictive KL risk.

Our second finding is a refined description of the Horseshoe phase transition at the local-shrinkage level.
Indeed, existing Horseshoe theory emphasizes the automatic separation between strongly-shrunk noise coordinates and signal coordinates that escape shrinkage \citep{carvalho2010horseshoe, van2014horseshoe, van2017adaptive}.
We sharpen this picture, and find \emph{when this separation occurs} by analyzing the posterior of the local shrinkage scale $\lambda$ directly.
We show that this posterior is bimodal.
The switch of dominant mode occurs when $|Y_i|$ reaches the order $\sqrt{\log(1/\tau)}$.
This transition occurs at the same scale as the usual detection threshold, which is also inherited by the predictive behavior of the Horseshoe.

These structural results lead to our first main theorem, which shows that a continuous global-local prior can attain the exact asymptotic minimax predictive KL risk over $\Theta_n(s_n)$.
When the sparsity level $s_n$ is known, by carefully calibrating the global shrinkage scale $\tau$, the Horseshoe predictive risk is not only rate-optimal but also sharp at the minimax constant of one.
At a high level, our argument builds on the risk-decomposition strategy developed in \cite{mukherjee2015exact} and a similar proof framework to \cite{rockova2023adaptive}.
The Horseshoe setting, however, requires new prior-specific ingredients, including the aforementioned spectroscopy representation and phase transition.
We also apply a similar mathematical toolbox as in \cite{van2014horseshoe} for delicate control.

We then turn to the more realistic setting in which $s_n$ is unknown.
We adopt a full-Bayes approach with a hierarchical Horseshoe prior, which places a hyperprior on the global shrinkage scale $\tau$.
Our analysis differs from existing adaptive Horseshoe theory, e.g., as in \cite{van2017adaptive}, since we do not restrict the hyperprior to $[1/n,1]$.
Instead, we work with an exponential prior on $(0,\infty)$ and study the resulting posterior of $\tau$, which is crucial in manipulating the \emph{adaptive} predictive risk.
As a technical input, we obtain an explicit expression for the posterior density $\pi(\tau\mid \bm Y)$, a convenient starting point for the adaptive analysis.

The main adaptive result is obtained under a theta-min condition that requires the minimum signal strength to be large enough.
On this restricted parameter class, we prove two complementary lemmas showing that the posterior of $\tau$ neither \emph{overshoots} the oracle calibration nor \emph{undershoots} it too severely.
Combined with the predictive risk decomposition and spectroscopy arguments, they yield an upper bound for the \emph{adaptive} predictive risk under the hierarchical Horseshoe prior.
The theorem therefore identifies a regime in which full-Bayes Horseshoe prediction adapts successfully to unknown sparsity.

We track the posterior behavior of $\tau$ and examine finite-sample predictive risks with numerical studies.
A consistent pattern is that, when weak nonzero coordinates are present, the posterior of $\tau$ is governed more by the number of strong signals than by the total number of nonzero coordinates, and the predictive risk exhibits the same behavior.

We demonstrate in real data analysis that sparse Gaussian predictive inference can be used not only for forecasting, but also for uncertainty-aware \emph{pairwise verification} and symmetry assessment.
Specifically, we may evaluate the extent to which a second observation is plausibly generated from the same latent source as the first, by comparing the former against the latter's predictive distribution.
Our first example considers image facial recognition on the JAFFE dataset \citep{lyons2020coding}.
A transformation into wavelet coefficients satisfies the Gaussian sequence model setup.
Testing whether the two wavelet coefficients are based on the same underlying sparse mean enables subject recognition despite the elastic deformation caused by facial expressions.
Our second example considers brain lateralization analysis on the ABIDE dataset \citep{di2014autism}.
The functional observations from multiple brain regions are reduced to Gaussian observations using functional principal component analysis.
We use predictive inference to measure discrepancy between contralateral brain regions, and then detect atypical brain lateralization in autism.

The rest of the paper is organized as follows.
Section \ref{sec:horseshoe.fixed.tau} develops the fixed-$\tau$ theory for the separable Horseshoe prior, including Horseshoe spectroscopy, the local phase transition, and the exact asymptotic minimax theorem.
Section \ref{sec:minimax.adaptive} studies the hierarchical Horseshoe prior, derives the posterior expression for $\tau$, and proves the adaptive predictive-risk bound under the theta-min condition.
Section \ref{sec:simulation} provides numerical experiments.
Finally, real data experiments are given in Section \ref{sec:real.data.analysis}.


\section{Separable Horseshoe Prior}\label{sec:horseshoe.fixed.tau}
We first consider the convenient setup where the sparsity level $s_n$ is known.
In this case, a Horseshoe prior with a fixed global shrinkage scale $\tau$ would suffice.
Since this fixed-$\tau$ prior is separable, the predictive density $\hat p(\tilde{\bm y} \mid \bm y)$ can be written as a product of univariate predictive densities $\hat p(\tilde y_i\mid y_i)$, and the KL predictive loss and risk are additive.
In this regard, it is convenient to analyze univariate Horseshoe prior as a starting point.
We show that the univariate predictive risk of a Gaussian mixture prior (including the Horseshoe) also attains a decomposition similar to \cite{mukherjee2015exact} and \cite{rockova2023adaptive}, which is deferred to Appendix \ref{sec:risk.decomp}.

\subsection{Horseshoe Spectroscopy}\label{sec:spectroscopy}

Since the Horseshoe prior is a continuous Gaussian scale mixture, it is convenient to make use of the closed-form predictive density of a Gaussian prior, which is also Gaussian.
The Horseshoe predictive density can thus be decomposed as a continuous mixture of Gaussian.
Consequently, the Horseshoe predictive KL loss attains a tight upper bound, a mixture of KL losses under the Gaussian prior, which has a closed form.
This explains the name {\it spectroscopy}, and could be very useful in controlling the predictive risk.

The univariate prior of $\theta$ is a combination of Gaussian priors.
For brevity, we omit the subscripts and adopt the following notation for the remainder of this section.
We have $\theta \mid \lambda,\tau \sim N(0,\lambda^2\tau^2)$, and $\lambda \sim C^+(0,1)$.
We can show that the predictive density can be expressed as $\hat{p}_\pi(\tilde{y}\mid y) = \int \hat{p}_\lambda(\tilde{y}\mid y)\pi(\lambda\mid y)\diff\lambda$,
where $\hat p_\lambda$ is the predictive density under the Gaussian prior with fixed $\lambda$.
The predictive KL loss $L(\theta, \hat p_\pi (\cdot \mid y))$ can thus be upper bounded by
\begin{equation}\label{eq:loss.spectroscopy}
    L(\theta, \hat p_\pi(\cdot \mid y)) \leq \int L(\theta, \hat p_\lambda(\cdot \mid y)) \pi(\lambda \mid y) \diff \lambda,
\end{equation}
where $L(\theta, \hat p_\lambda(\cdot \mid y))$ is the predictive KL loss under the Gaussian prior with fixed $\lambda$.
The proof is given in Appendix \ref{sec:pf.lm.loss.spectroscopy}.
An important message that these spectroscopy results deliver is that the weight of the spectral combination is exactly the posterior density of the local shrinkage scale $\lambda$.
Therefore, the behavior of the posterior of $\lambda$ will be inherited by the shrinkage behavior of the predictive density.
We shall see in the following subsection that there is indeed a phase transition of $\pi(\lambda\mid y)$ with respect to $|y|$ under the Horseshoe prior.

In fact, given any fixed value of $\tau$ and $\lambda$, the predictive density of $\tilde y$ is Gaussian.
Recall that $v = (1 + r^{-1})^{-1}$. Then we can write
\begin{equation}\label{eq:predictive.density.fixed.lambda}
\hat{p}_\lambda(\tilde y\mid y) \sim N\left(\frac{\lambda^2\tau^2}{1+\lambda^2\tau^2}y, \frac{v+\lambda^2\tau^2}{(1-v)(1+\lambda^2\tau^2)}\right).
\end{equation}
The corresponding KL loss $L(\theta,\hat{p}_\lambda)$ also has a closed-form expression,
\begin{multline*}
    \frac{1-v}{2}\frac{1+\lambda^2\tau^2}{v+\lambda^2\tau^2}\left(\frac{\lambda^2\tau^2}{1+\lambda^2\tau^2}y - \theta\right)^2 +\frac{1}{2}\log\left(1+(1-v)\frac{\lambda^2\tau^2}{v+\lambda^2\tau^2}\right) - \frac{1}{2}(1-v)\frac{\lambda^2\tau^2}{v+\lambda^2\tau^2}\\
    \leq \frac{1-v}{2}\frac{1+\lambda^2\tau^2}{v+\lambda^2\tau^2}\left(\frac{\lambda^2\tau^2}{1+\lambda^2\tau^2}y - \theta\right)^2.
\end{multline*}
The inequality is due to $\log(1+x)\leq x$ for $x\geq 0$.
With a calibration of $\tau$ close to zero, this bound is particularly tight.
Now we arrive at the following lemma that provides an upper bound on the predictive KL risk.
\begin{lemma}[Upper bound of KL risk from spectroscopy]\label{lm:risk.spectroscopy}
    Under the Horseshoe prior with fixed $\tau$, the predictive risk is upper bounded by
    \begin{multline}\label{eq:risk.spectroscopy.nonzero}
            \rho(\theta, \hat p) \leq \frac{1-v}{6}\E_{Y \mid \theta}\left[\frac{\Phi_1(1,1,\frac{5}{2}, -\frac{Y^2}{2}, 1-\tau^2)}{\Phi_1(1,1,\frac{3}{2}, -\frac{Y^2}{2}, 1-\tau^2)}(Y-\theta)^2\right]\\
            + \frac{1-v}{3}\E_{Y \mid \theta}\left[ \frac{\Phi_1(2,1,\frac{5}{2}, -\frac{Y^2}{2}, 1-\tau^2)}{\Phi_1(1,1,\frac{3}{2}, -\frac{Y^2}{2}, 1-\tau^2)} \left(\frac{\theta^2}{v} - 2\theta(Y-\theta)\right)\right].
    \end{multline}
    Specifically, in the case that $\theta=0$, this reduces to
    \begin{equation}\label{eq:risk.spectroscopy}
        \rho(0, \hat p) \leq \frac{1-v}{6}\E\left[\frac{\Phi_1(1,1,\frac{5}{2}, -\frac{Z^2}{2}, 1-\tau^2)}{\Phi_1(1,1,\frac{3}{2}, -\frac{Z^2}{2}, 1-\tau^2)}Z^2\right],
    \end{equation}
    where $Z \sim N(0,1)$.
\end{lemma}
\begin{proof}
    See Appendix \ref{sec:pf.lm.risk.spectroscopy}.
\end{proof}
The bound in \eqref{eq:risk.spectroscopy.nonzero} might not necessarily be helpful in bounding $\sup_{\theta\in\R}\rho(\theta,\hat p)$.
Due to the presence of a $\theta^2$ factor, the second term goes to infinity by taking a supremum over $\theta\in\R$.
However, for the case when $\theta=0$, the bound in \eqref{eq:risk.spectroscopy} may be a good supplement to the risk decomposition in Lemma \ref{lm:risk.decomp}.

\subsection{Phase Transition in Posterior of $\lambda$}\label{sec:lambda.posterior}
By \eqref{eq:predictive.density.fixed.lambda}, define $\kappa = (1+\lambda^2\tau^2)^{-1}$, then the predictive density under fixed $\lambda$ is a Gaussian with mean $(1-\kappa)y$.
Here, $\kappa$ can be interpreted as a posterior shrinkage factor, where $\kappa=0$ yields no shrinkage and describes signals, and $\kappa=1$ yields near-total shrinkage and describes noise \citep{carvalho2010horseshoe}.
The half-Cauchy prior imposed on $\lambda$ is equivalent to a compound confluent hypergeometric prior, $\kappa \sim \text{CCH}\left(\frac{1}{2}, \frac{1}{2}, 1, 0, 1, \frac{1}{\tau^2}\right)$ (see \cite{gordy1998generalization}).
The prior has a spike on both $\kappa = 0$ and $\kappa = 1$.
When $\tau$ decreases, the spike on $\kappa=1$ becomes higher, meaning more full shrinkage.
The posterior of $\kappa$, $\kappa \mid y \sim \text{CCH}\left(1, \frac{1}{2}, 1, \frac{y^2}{2}, 1, \frac{1}{\tau^2}\right)$, has already been well-studied in the Horseshoe literature.
It attains a phase transition with respect to the scale of observation, $|y|$.
However, to answer the question when this phase transition occurs, we shall take a closer look at the posterior of $\lambda$.

The posterior density of $\lambda$ has the following closed form,
\begin{equation}\label{eq:lambda.posterior}
    \pi(\lambda\mid y) = \frac{1}{\tau}\frac{1}{\Phi_1(1,1,3/2,-y^2/2, 1-\tau^2)}\frac{1}{\sqrt{1+\lambda^2\tau^2}}\exp\left(-\frac{y^2}{2(1+\lambda^2\tau^2)}\right)\frac{1}{1+\lambda^2}.
\end{equation}
\begin{figure}[t]
    \centering
    \includegraphics[width=4cm]{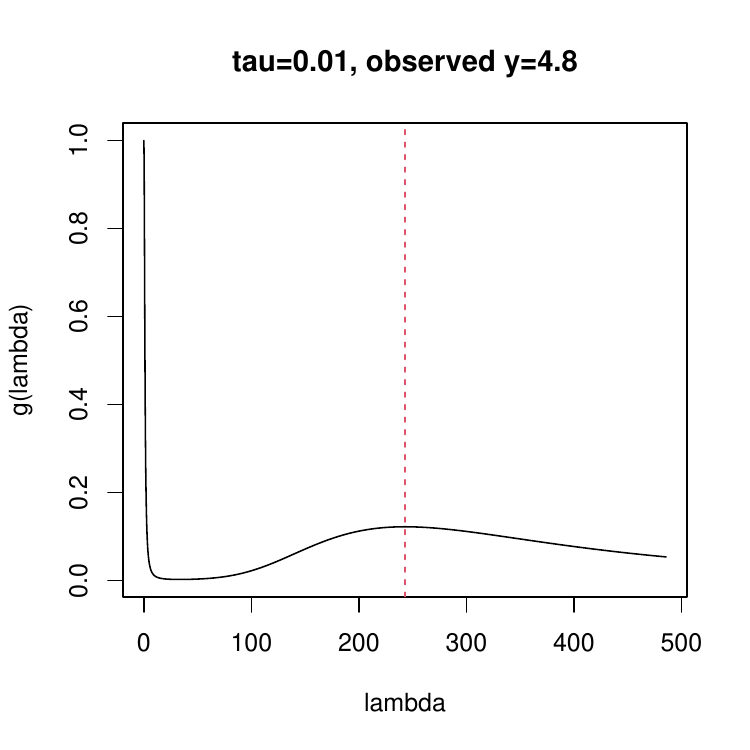}
    \includegraphics[width=4cm]{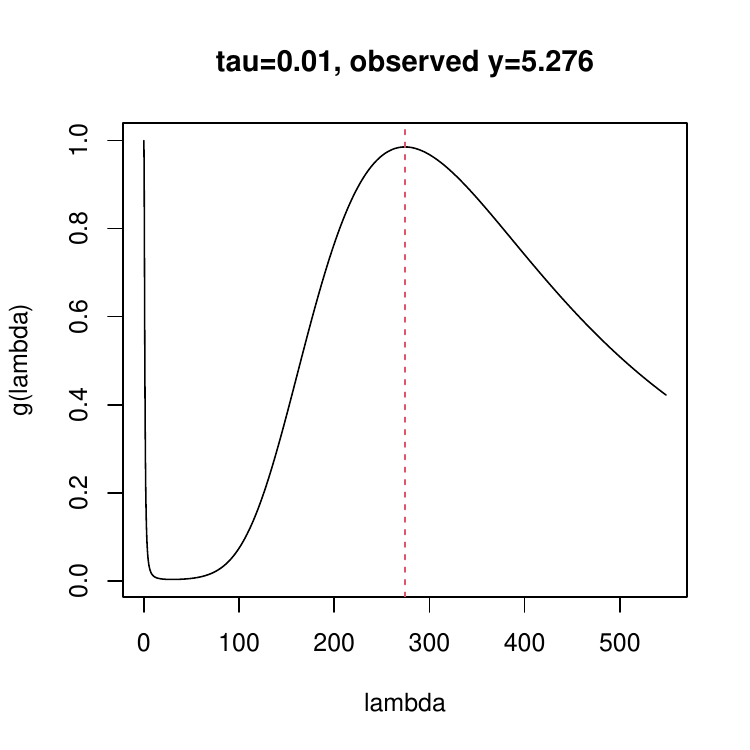}
    \includegraphics[width=4cm]{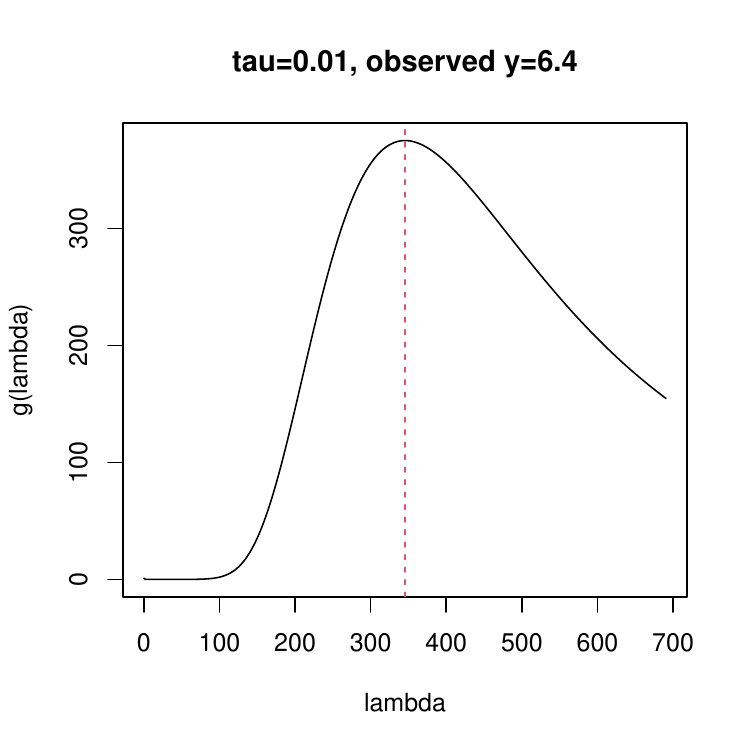}
    \caption{An example of $g(\lambda)$ under different $y$ values. The dashed red line depicts $\lambda_2^\ast$, the theoretical local maximum of $g(\lambda)$. Note that $g(\lambda)$ is normalized posterior density $\pi(\lambda\mid y)$, such that $g(0)=1$.}
    \label{fig:posterior}
\end{figure}
The posterior under different values of $y$ are plotted in Figure \ref{fig:posterior}.
Note that the posterior density $\pi(\lambda\mid y)$ always has a mode on $\lambda = 0$.
As $|y|$ increases from zero, a second mode appears.
As $\tau \rightarrow 0$, for the two modes to be comparable in order, it is required that $|y| = O( \sqrt{\log(1/\tau)})$.
For smaller observations, the first mode on $\lambda=0$ dominates, and the posterior of $\lambda$ is concentrated around zero, which corresponds to full shrinkage toward zero.
For larger observations, the second mode dominates, and the mass around $\lambda=0$ becomes negligible, which corresponds to no shrinkage for prominent signals.

Recall that the posterior behavior of the local shrinkage scale $\lambda$ is inherited by the shrinkage behavior of the predictive density.
The phase transition point, $|y| \asymp \sqrt{\log(1/\tau)}$, prompts an important proof technique in controlling the predictive KL risk.
Indeed, we shall split the cases when the observation is smaller or larger than this order, and discuss their respective contribution to the KL risk.
This will be a key method in the proof of our main results.

\subsection{Asymptotic Exact Minimaxity of KL Risk}\label{sec:minimax.fixed.tau}
We are now ready to demonstrate the first main result of this paper.
When the sparsity level $s_n$ is known, by carefully calibrating the global shrinkage scale $\tau$, the Horseshoe prior achieves an exact-minimax predictive KL risk asymptotically.

\begin{theorem}\label{thm:minimax.fixed.tau}
    Assume the Horseshoe prior \eqref{eq:HS.prior}.
    Denote $v = 1/(1+1/r)$ and let $\tau$ take the value of $\tau_{n,\alpha} = s_n\log^\alpha(n/s_n) / n$, with $\alpha\in[0,1/2]$.
    Then with $s_n / n \rightarrow 0$, for any fixed $r\in(0,\infty)$ (i.e. $v\in(0,1)$), we have
    \begin{equation}
        \sup_{\bm \theta \in \Theta_n(s_n)}\rho_n(\bm \theta, \hat p_\pi) \leq \left(1 + \frac{2}{\sqrt{\pi}\log^{1/2 - \alpha}(n/s_n)}\right) \frac{s_n}{1+r} \log \frac{n}{s_n} + \tilde C_n(v, \tau_{n,\alpha}).
    \end{equation}
    Here, $\tilde C_n(v,\tau)$ is defined in \eqref{eq:C.n.def}, and has a lower order than $s_n\log(n/s_n)$.
\end{theorem}

\begin{proof}
    See Appendix \ref{sec:apdx.pf.minimax.fixed.tau}.
\end{proof}

\begin{remark}
    With $\alpha \in [0,1/2)$, we get
    $\sup_{\bm \theta \in \Theta_n(s_n)}\rho_n(\bm \theta, \hat p_\pi) \leq  (1+r)^{-1} s_n \log (n/s_n)(1+o(1))$,
    i.e., asymptotic exact minimax risk.
    When $s_n$ is known, recall that \cite{mukherjee2022minimax} reaches exact minimaxity with a discrete dual-grid-shaped spike-and-slab prior.
    \cite{rockova2023adaptive} considers Dirac-spike-and-Laplace-slab prior and spike-and-slab Lasso prior, both reaching rate minimaxity up to a universal constant.
    Theorem \ref{thm:minimax.fixed.tau} is the first result that reaches exact minimaxity using a continuous prior that the authors are aware of.
\end{remark}

\begin{remark}
    Our optimal calibration of $\tau$ in the predictive inference problem is consistent with the optimal choice in the estimation problem \citep{van2014horseshoe}.
    With the same calibration of $\tau$, at most at the order $\tau_{n,1/2}$, the Horseshoe estimator (the posterior mean) contracts to the true parameter $\bm\theta$ at the near-minimax rate for quadratic loss over sparse models.
    These two complementary results in two different problems show that $\tau$ is regarded as an effective sparsity rate, instead of merely a technical tuning parameter.
    The fact that this same sparsity calibration leads to exact minimax predictive risk, with the exact variance-dependent constant $(1+r)^{-1}$, is noteworthy.
\end{remark}

Since a separable prior ensures additivity of the predictive risk, it suffices to provide an upper bound for two univariate risks, the maximum risk for the signals and the risk at zero for the noise.
The phase transition in posterior of $\lambda$ indicates that we may split by the order of observation with respect to $\tau$.
Indeed, we split by $|Y_i| = \sqrt{2v\log(1/\tau)}$ for the signal case $\theta_i \neq 0$, and split by $|Y_i| = \sqrt{2\log(1/\tau)}$ for the noise case $\theta_i = 0$.
The aforementioned Horseshoe spectroscopy results are useful in the proof for the noise case, while in the signal case we use risk decomposition and carefully split the integrals.

\section{Adaptiveness to Unknown Sparsity}\label{sec:minimax.adaptive}
Instead of the ideal case scenario with a known sparsity level $s_n$, for the real-world application, practitioners are rather more concerned about the case where $s_n$ is unknown.
In this case, we are no longer able to calibrate $\tau$ as a fixed value with respect to $s_n$ as in Theorem \ref{thm:minimax.fixed.tau}.
When $\tau$ is fixed, following the last step of proof of Theorem \ref{thm:minimax.fixed.tau}, the best we can do is to calibrate $\tau = \log^{\alpha}(n) / n$.
For any $\alpha\in[0,1/2)$, we obtain
\begin{equation}\label{eq:adaptive.risk.fixed.tau}
\sup_{\bm \theta \in \Theta_n(s_n)}\rho_n(\bm \theta, \hat p_\pi) \leq \frac{s_n}{1+r} \log n (1+o(1)).
\end{equation}

To obtain a sharper rate, setting a fixed deterministic value of $\tau$ may not be sufficient.
There are two general directions in the literature, namely a full-Bayes approach by imposing a prior distribution on $\tau$ (e.g., Half-Cauchy as in \cite{carvalho2010horseshoe}), and an empirical Bayes approach by plugging in an estimated $\hat\tau$ (e.g., MMLE as in \cite{van2017adaptive}).
Both methods are capable of achieving rate-minimax quadratic risk in the estimation problem \citep{van2017adaptive}.

For the predictive problem, we consider the full-Bayes approach.
Let $\pi(\tau)$ be the prior distribution of $\tau$.
Under this hierarchical structure, we can still establish risk decomposition and spectroscopy results similar to those in the fixed-$\tau$ scenario.
They are deferred to Appendix \ref{sec:hierarchical.decomp.spec}.
While these results are directly helpful in the proof of an adaptive predictive risk bound, an important yet challenging task is to control the posterior $\pi(\tau\mid \bm Y)$.
Quite interestingly, it has a closed-form expression under the Horseshoe prior.
\begin{lemma}[Posterior of $\tau$]\label{lm:post.tau}
    Under the hierarchical Horseshoe prior with $\pi(\tau)$ being the hyperprior imposed on the global shrinkage parameter $\tau$, we have
    \[
    \pi(\tau \mid \bm Y) \propto \pi(\tau)\cdot \tau^n \prod_{i=1}^n \Phi_1\left(1,1,\frac{3}{2}, -\frac{Y_i^2}{2}, 1-\tau^2 \right),
    \]
    where $\Phi_1$ is the confluent hypergeometric function with two variables.
\end{lemma}

A simulation of the behavior of this posterior of $\tau$ under different setups of $\bm\theta$ is given in Section \ref{sec:posterior.experiment}.
We have discussed that $\tau$ plays a role of effective sparsity rate.
Unlike in the known-$s_n$ case, we expect the posterior of $\tau$ to capture the number of signals automatically.
From its form in Lemma \ref{lm:post.tau}, the only usable evidence is the scale of the observations, $Y_i^2$.
But when the size of the underlying signal $|\theta_i|$ is not large enough, the posterior of $\tau$ might not be able to distinguish between a small signal and pure noise.
For a more accurate control, we introduce the following assumption of minimal signal strength.
\begin{assumption}[Theta-min Condition]\label{assumption:beta.min}
    There exists some $c>0$, such that $\bm\theta \in \Theta_n(s_n,c)$, where
    \begin{equation}\label{eq:min.signal.set}
        \Theta_n(s_n, c) = \Theta_n(s_n) \cap \left\{\min_{i: \theta_i \neq 0} |\theta_i| > c\sqrt{2\log n}\right\}.
    \end{equation}
\end{assumption}
Assumption \ref{assumption:beta.min} requires the true underlying parameter to be either zero or larger than the detection threshold of order $\sqrt{2\log n}$.
This eliminates the possibility of having signals that are too small to be recognized.
For the remainder of this section, we consider the following hyperprior of the global shrinkage scale $\tau$,
\begin{equation}\label{eq:prior.of.tau}
    \pi(\tau) = ne^{-n\tau}\one{\tau>0}.
\end{equation}
We shall see that under the theta-min condition, the full-Bayes method with this hyperprior of $\tau$ captures the underlying sparsity level accurately, and thus still guarantees a rate-minimax predictive risk.
A vital step in the proof is to show that the posterior of $\tau$ is concentrated around the oracle rates $\tau_{n,\alpha}$ as in the fixed-$\tau$ result in Theorem \ref{thm:minimax.fixed.tau}.
The following two lemmas provide an insight on such concentration.

\begin{lemma}\label{lm:tau.not.overshooting}
    Under the theta-min condition in Assumption \ref{assumption:beta.min}, if we impose the exponential hyperprior \eqref{eq:prior.of.tau} on $\tau$, then for $c>\sqrt{6}$,
    \begin{equation}
        \sup_{\bm\theta\in\Theta_n(s_n, c)}\E_{\bm Y \mid \bm\theta}\E[\tau\mid\bm Y] \leq K\frac{s_n}{n} (1+o(1)).
    \end{equation}
    Here, $K$ is a universal constant that satisfies $K - \log K - \log(9 e^{4}/2) > 0$.
\end{lemma}

\begin{lemma}\label{lm:tau.not.undershooting}
    Under the theta-min condition in Assumption \ref{assumption:beta.min}, if we impose the exponential hyperprior \eqref{eq:prior.of.tau} on $\tau$, then for $c>\sqrt{6}$,
    \begin{equation}
        \sup_{\bm\theta\in\Theta_n(s_n, c)}\E_{\bm Y \mid \bm\theta}\E\left[\log \frac 1 \tau\mid\bm Y\right] \leq \log\frac{n}{s_n} + O(1).
    \end{equation}
\end{lemma}

Lemma \ref{lm:tau.not.overshooting} guarantees that the posterior distribution of $\tau$ does not overshoot, in that the posterior mean of $\tau$ does not exceed the order of $\tau_{n,0}$.
On the contrary, Lemma \ref{lm:tau.not.undershooting} ensures that the posterior does not undershoot either, in that the posterior mean of $\log(1/\tau)$ is asymptotically upper bounded by exactly $\log(n/s_n)$.
The proof of Lemma \ref{lm:tau.not.overshooting} is given in Appendix \ref{sec:pf.lm.tau.not.overshooting}, while the proof of Lemma \ref{lm:tau.not.undershooting} is given in Appendix \ref{sec:pf.lm.tau.not.undershooting}.

The following Theorem \ref{thm:adaptive.beta.min} guarantees that under the theta-min condition, the full-Bayes Horseshoe prior achieves a sharper rate of predictive risk than \eqref{eq:adaptive.risk.fixed.tau}.
\begin{theorem}\label{thm:adaptive.beta.min}
    Under Assumption \ref{assumption:beta.min} and the hierarchical Horseshoe prior that imposes a hyperprior \eqref{eq:prior.of.tau} on $\tau$, for any $c>\sqrt{6}$, up to a universal constant $C>0$,
    \begin{equation}\label{eq:adaptive.beta.min}
        \sup_{\bm\theta \in \Theta_n(s_n,c)} \rho_n(\bm\theta, \hat p_\pi)
        \leq C \frac{s_n}{1+r} \sqrt{\log\frac{n}{s_n}} (1+o(1)).
    \end{equation}
\end{theorem}
\begin{proof}
    See Appendix \ref{sec:pf.thm.adaptive.beta.min}.
\end{proof}

\begin{remark}
    The rate in \eqref{eq:adaptive.beta.min} is even lower than the minimax rate \eqref{eq:minimax.rate} over $\Theta_n(s_n)$.
    This is not a contradiction, since Theorem \ref{thm:adaptive.beta.min} holds over a restricted set $\Theta_n(s_n,c) \subseteq \Theta_n$.
    With the theta-min condition of Assumption \ref{assumption:beta.min}, the signal case $\theta_i \neq 0$ becomes much easier to control, and has a negligible contribution on the adaptive predictive risk upper bound.
    The main pillar becomes the noise case $\theta_i = 0$ instead.
\end{remark}

\section{Numerical Study}\label{sec:simulation}

In this section, we present a simulation study demonstrating the accuracy of predictive inference using both the Horseshoe prior and the spike-and-slab family of priors.

\subsection{How Posterior of $\tau$ Adapts to Unknown Sparsity}\label{sec:posterior.experiment}

With a random global shrinkage parameter $\tau$ in place, the Horseshoe predictive density is written as
\begin{equation}\label{eq:pred.density.ada}
    \hat p_\pi(\bm{\tilde y} \mid \bm y) = \int \prod_{i=1}^n \left\{\int \hat p_{\lambda_i}(\tilde y_i \mid y_i,\tau) \pi(\lambda_i \mid y_i) \diff \lambda_i \right\} \pi(\tau\mid \bm y) \diff \tau.
\end{equation}
The posterior of global shrinkage factor $\tau$ is thus worth inspecting.
Recall that $\pi(\tau \mid \bm y)$ has a closed-form solution that depends on the hyperprior of $\tau$; see Lemma \ref{lm:post.tau}.
We hope that the posterior of $\tau$ automatically concentrates around the oracle fixed value $\tau_{n,\alpha}$ specified in Theorem \ref{thm:minimax.fixed.tau}.
In that case, the posterior $\pi(\tau \mid \bm y)$ filters out proper global shrinkage parameters such that the predictive density, as if we knew the true underlying sparsity level $s_n$.
Thus can we obtain an adaptive predictive density comparable to that of the oracle case (i.e. $s_n$ known).

In the following experiment, we inspect the posterior distribution $\pi(\tau \mid \bm y)$ and compare it to the oracle calibration $\tau_{n,\alpha}$.
We compare the following two setups:
\begin{itemize}
    \item {\bf Setup 1}. The case when signals and noise are perfectly separated: $\bm\theta \in \R^n$ contains $s_n$ many signals of size $3\sqrt{2\log n}$, else zero.
    \item {\bf Setup 2}. The case when weak signals less than detection threshold are present: $\bm\theta\in \R^n$ contains $s_n$ many strong signals of size $3\sqrt{2\log n}$, $(n/2-s_n)$ many weak signals of size $0.3\sqrt{2\log n}$, and else zero.
\end{itemize}
We choose an exponential distribution with rate $n$ as the hyperprior of $\tau$, randomly sample a $\bm y$ from the likelihood, and plot the posterior $\pi(\tau \mid \bm y)$ with the range of $n = \{100, 300, 600, 1000\}$.
For this experiment, we set $s_n = n/10$.
The result is shown on Figure \ref{fig:post.tau}.
Along with the posterior, two oracle calibrations $\tau_{n,\alpha}$ with $\alpha = 0$ and $\alpha = 1/2$ have also been marked in these plots.

\begin{figure}[h]
    \centering
    \includegraphics[width=3cm]{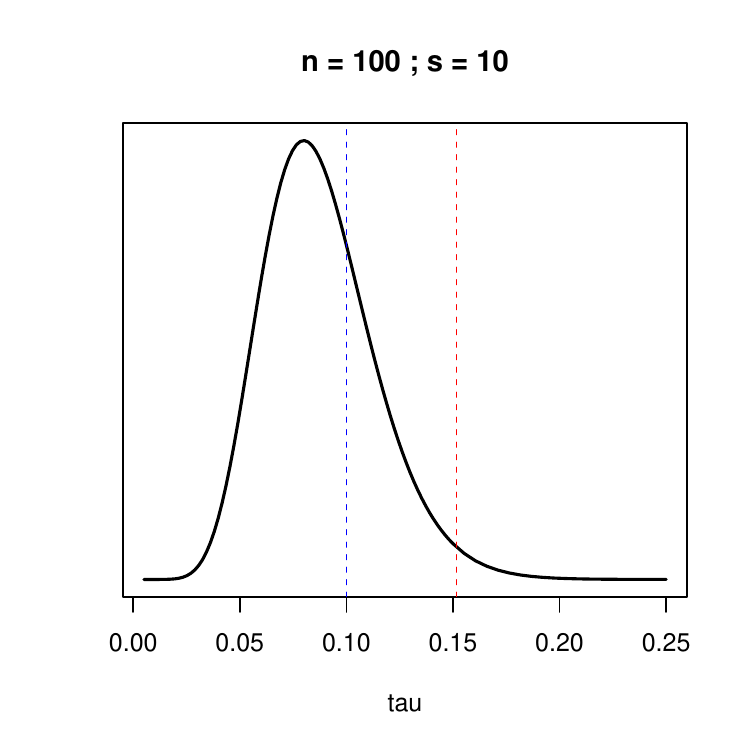}
    \includegraphics[width=3cm]{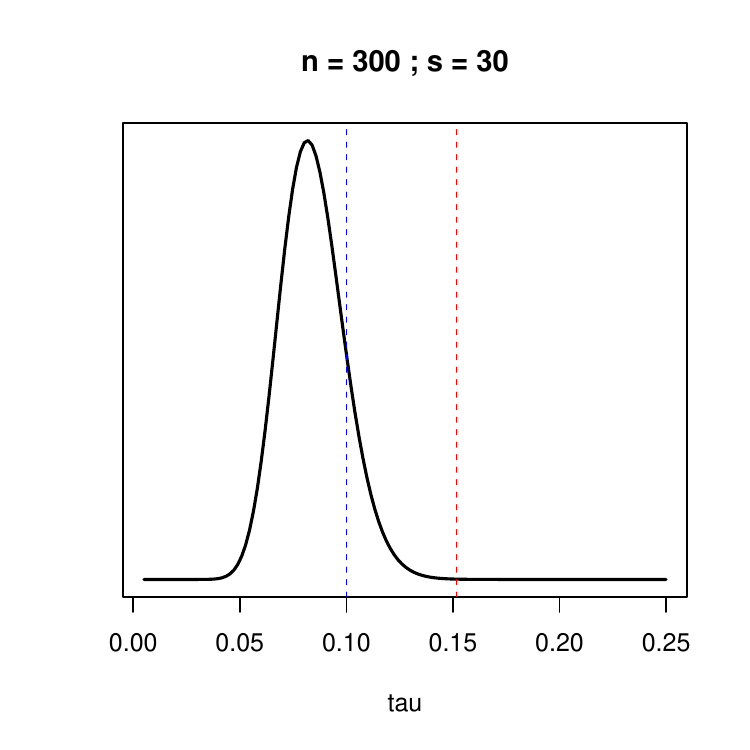}
    \includegraphics[width=3cm]{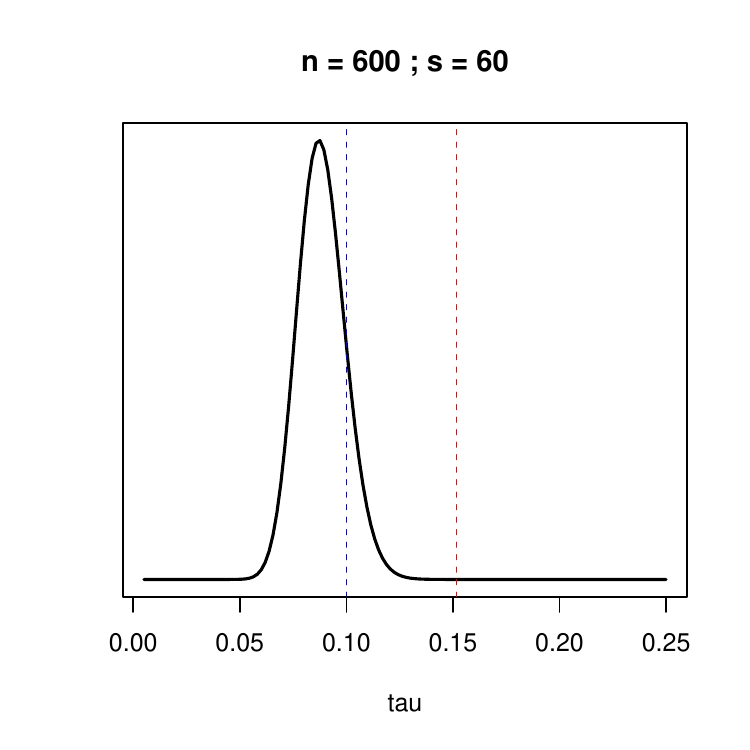}
    \includegraphics[width=3cm]{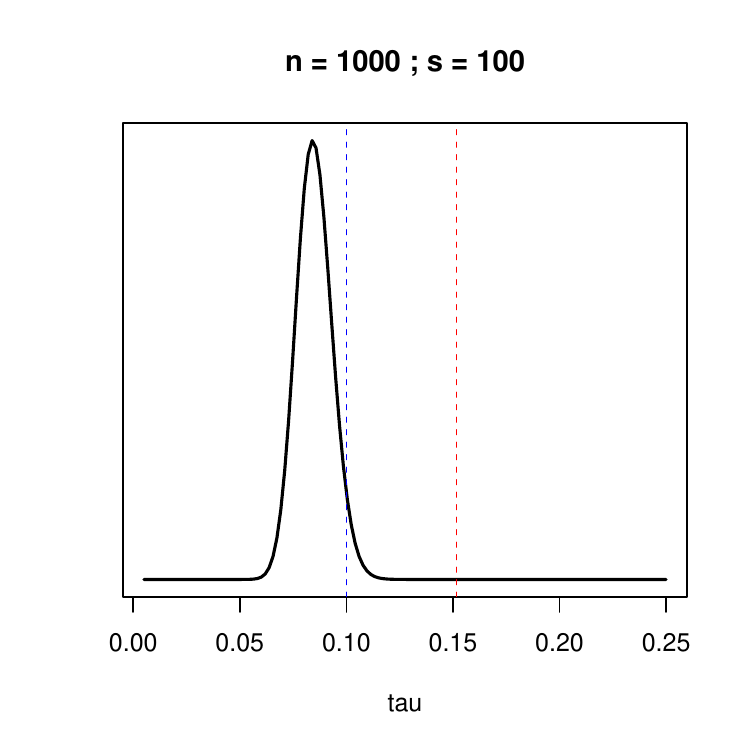}

    \includegraphics[width=3cm]{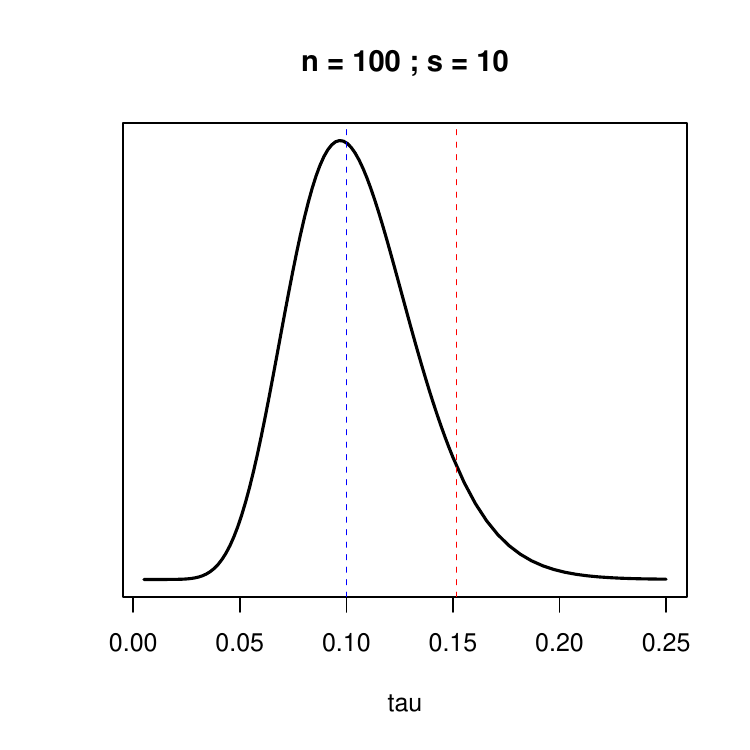}
    \includegraphics[width=3cm]{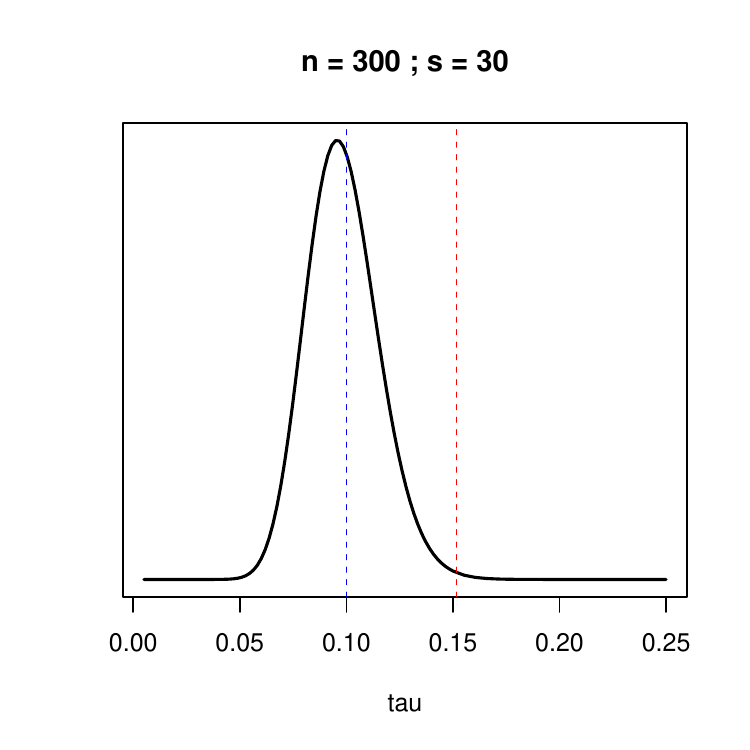}
    \includegraphics[width=3cm]{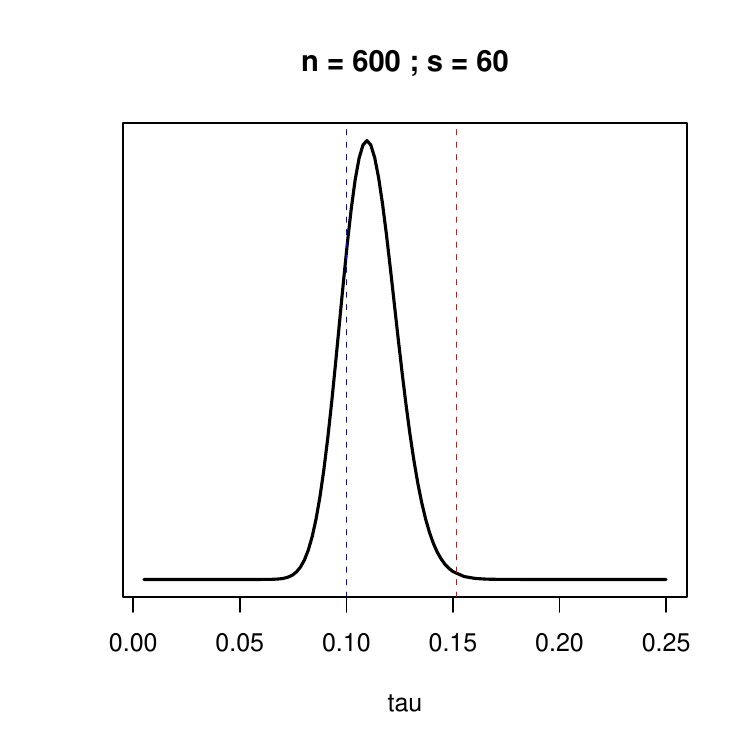}
    \includegraphics[width=3cm]{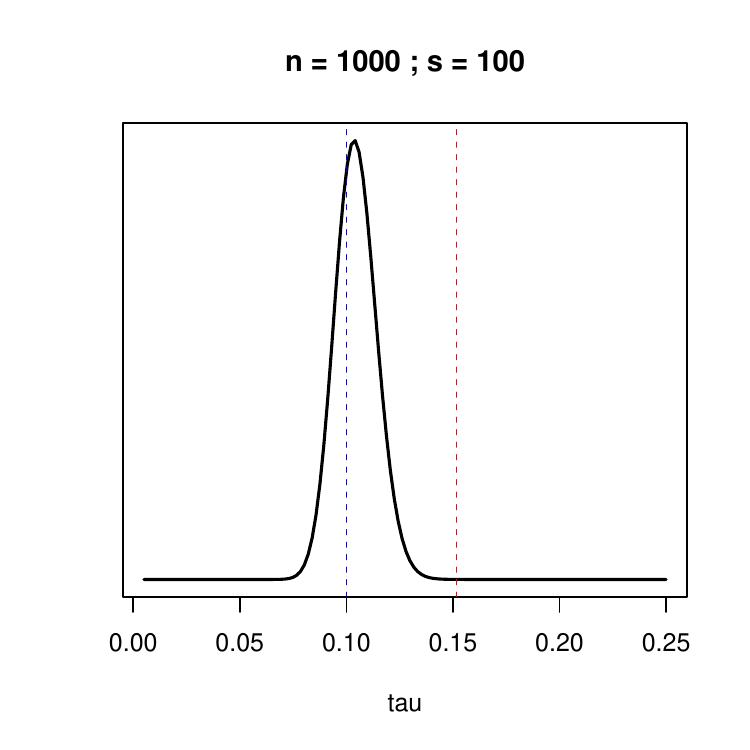}
    \caption{The posterior of $\tau$, $\pi(\tau \mid \bm y)$, under the hierarchical Horseshoe model.
    Here, $\bm y$ is randomly generated from a fixed $\bm\theta$ under two different settings.
    On the first row, $\bm\theta$ follows Setup 1.
    On the second row, $\bm\theta$ follows Setup 2.
    The hyperprior of $\tau$ was selected as an exponential distribution with rate $n$, and we set $s_n = n/10$ all the time.
    The blue dashed vertical line represents the minimum oracle calibration $\tau_{n,0} = s_n/n$, while the red dashed vertical line represents the maximum oracle calibration $\tau_{n,1/2} = s_n\sqrt{\log(n/s_n)}/n$.
    }
    \label{fig:post.tau}
\end{figure}

The first observation from Figure \ref{fig:post.tau}
is that the posterior of both hyperparameters shrink with an increasing $n$.
The posterior appears both upper and lower bounded when significant signal is present.
Another important finding is the ability to detect weak signals.
Note that the only difference between Setups 1 and 2 is the presence of $(n/2 - s_n)$ many weak signals, at the value of $0.3 \sqrt{2\log n}$.
These weak signals indeed pushed the posterior slightly to the right, but it is still concentrated around $s_n / n$.
This phenomenon indicates that the hierarchical Horseshoe with exponential hyperprior on $\tau$
is capturing only the strong signals while almost ignoring the weak ones.
This statement is also evident in the following quantitative experiment.

\subsection{Simulating Predictive Risks}\label{sec:KL.risk.experiment}

This experiment computes the predictive KL risk for various priors under a certain underlying model.
For the maximum KL risk for these separable priors over $\Theta_n(s_n)$, see Appendix \ref{sec:max.KL}.

An objective of this numerical study is to explore how the full-Bayes methods adapt to unknown sparsity levels.
To do this, we apply the following setup where there exist both strong and weak signals.
For a fixed $n$, we assign multiple sparsity levels $s_n$.
For each sparsity level, we let $s_n^\ast$ many entries of $\bm\theta$ be $c\sqrt{2\log n}$, $(300-s_n^\ast)$ entries be $\sqrt{2\log n}$, and the rest zero.
Here, the factor $c>0$ represents the signal level, which is set to be 2, 3, or 4.
Throughout this experiment, we let $n = 500$, and $s_n^\ast \in \{25, 50, 100\}$.
This setup contains weak signals not exceeding the detection threshold, which is representative of many real-world applications where a sloppy separation between signal and noise is present.
We have also conducted a similar simulation with only strong signals and noise; see Appendix \ref{sec:risk.sim.setup.A}.
The detailed technical setup and the algorithm are deferred to Appendix \ref{sec:risk.sim.detail}.
In this experiment, we compare the predictive risk in the following three scenarios.
\begin{itemize}
    \item {\bf As if $s_n$ were known}: we calibrate the hyperparameters of all methods to their respective theoretically-proven optimal fixed values.
    For our setup, specifically, there are two sub-cases: either {\bf $s_n$ known as $s_n^\ast$}, the number of strong signals, or {\bf $s_n$ known as $300$}, the number of all signals.
    \item {\bf Does not adapt to $s_n$}: fixed-value priors are still used, but we no longer know the sparsity level $s_n$. Hyperparameters are calibrated to their respective optimal fixed values, with $s_n$ replaced by 1, which is the worst-case scenario.
    \item {\bf Adapts to $s_n$}: we use the hierarchical methods by imposing a hyperprior on the hyperparameter.
\end{itemize}

When $s_n$ is known, we consider the following calibrations.
First, we compute the KL risk for the bi-grid prior with the choice of $\eta = s_n/n$ \citep{mukherjee2022minimax}.
Next, we compute the KL risk for Dirac spike-and-slab (DSnS) prior with the choice of $\eta$ satisfying $\eta / (1-\eta) = s_n/n$ \citep{rockova2023adaptive}, with the rate of the Laplace slab $\lambda$ chosen over a grid (see Appendix \ref{sec:risk.sim.detail} for details).
Finally, we consider the Horseshoe (HS) prior with calibrations $\tau_{n,0}$ and $\tau_{n,1/2}$.
These same calibrations are also considered in the fixed-hyperparameter case that does not adapt to unknown sparsity level.
For the bi-grid method, we choose $\eta = 1/n$.
For Dirac spike-and-slab, we choose $\eta / (1-\eta) = 1/n$.
For the Horseshoe prior, our calibration is $\tau = 1/n$ or $\tau = \sqrt{\log n}/n$.

In the adaptive experiment with unknown $s_n$, we consider the Dirac spike-and-slab prior with the hyperprior $\eta \sim \text{Beta}(2, n+1)$ (denoted by DSnS-Beta) \citep{rockova2023adaptive}, where the Laplace slab rate $\lambda$ is chosen out-of-sample over a grid (see Appendix \ref{sec:risk.sim.detail} for details).
For the Horseshoe prior, we consider the exponential hyperprior with rate $n$ on $\tau$, denoted by HS-Exp.
The main point of interest is to inspect whether the hierarchical models imitate the KL risks of fixed-hyperprior methods as if $s_n = s_n^\ast$ or as if $s_n = 300$ --- a proxy of whether the full-Bayes approach is capable of detecting weak signals.

\begin{table}[t]
\centering
\caption{Kullback-Leibler risks for parameter vector $\bm\theta$.
When adapting to $s_n$, the risks are computed on 1000 data vectors $\bm Y$ and $\tilde{\bm Y}$ generated from $\bm\theta$ for full-Bayes approach.
}
\footnotesize
\label{tab:KL.risk.setup.B}
\renewcommand{\arraystretch}{1.15}
\begin{tabular}{l *{9}{r}}
\toprule
 &
 \multicolumn{3}{c}{$s_n^\ast=25$} &
 \multicolumn{3}{c}{$s_n^\ast=50$} &
 \multicolumn{3}{c}{$s_n^\ast=100$} \\
\cmidrule(lr){2-4}\cmidrule(lr){5-7}\cmidrule(lr){8-10}
Signal Level $c$ & 2 & 3 & 4 & 2 & 3 & 4 & 2 & 3 & 4 \\
\midrule
\multicolumn{10}{c}{\textit{As if $s_n$ were known as $s_n^\ast$, the number of significant signals}}\\
Bi-Grid
& 135.34 & 135.32 & 135.31
& 128.75 & 128.74 & 128.71
& 124.81 & 124.80 & 124.80\\
DSnS
& 143.80 & 143.79 & 143.79
& 135.32 & 135.31 & 135.31
& 127.78 & 127.77 & 127.77\\
HS, $\alpha=1/2$  
& 130.23 & 129.47 & 129.25
& 120.69 & 119.19 & 118.75
& 115.49 & 112.52 & 111.65\\
HS, $\alpha=0$  
& 138.75 & 137.98 & 137.76
& 127.15 & 125.64 & 125.20
& 117.85 & 114.87 & 113.99\\

\addlinespace[2pt]
\multicolumn{10}{c}{\textit{As if $s_n$ were known as 300, the number of all signals}}\\

Bi-Grid
& 118.19 & 118.19 & 118.19
& 119.86 & 119.86 & 119.86
& 123.20 & 123.20 & 123.20\\

DSnS
& 115.47 & 115.47 & 115.47
& 118.43 & 118.43 & 118.43
& 121.29 & 121.29 & 121.29 \\

HS, $\alpha=1/2$
& 106.10 & 105.38 & 105.17
& 108.07 & 106.63 & 106.19
& 112.00 & 109.12 & 108.25 \\

HS, $\alpha=0$
& 104.30 & 103.61 & 103.40
& 106.70 & 105.31 & 104.88
& 111.48 & 108.71 & 107.85 \\

\addlinespace[2pt]
\multicolumn{10}{c}{\textit{$s_n$ unknown; calibration fixed as if $s_n=1$}}\\

Bi-Grid
& 168.39 & 164.21 & 168.25
& 176.49 & 168.13 & 176.21
& 192.69 & 175.96 & 192.13 \\

DSnS
& 160.30 & 160.19 & 160.19
& 155.59 & 155.35 & 155.35
& 145.99 & 145.42 & 145.42 \\
HS, $\alpha=1/2$  
& 158.32 & 157.45 & 157.23
& 154.85 & 153.12 & 152.68
& 147.92 & 144.46 & 143.58 \\
HS, $\alpha=0$  
& 160.44 & 159.59 & 159.37
& 156.76 & 155.06 & 154.62
& 149.40 & 146.01 & 145.12 \\

\addlinespace[2pt]
\multicolumn{10}{c}{\textit{Adapts to $s_n$ --- full-Bayes approach}}\\
DSnS-Beta
& 139.88 & 140.92 & 139.98
& 134.11 & 133.70 & 133.21
& 128.05 & 126.82 & 127.75 \\
HS-Exp
& 132.00 & 131.25 & 130.19
& 124.45 & 122.45 & 122.08
& 117.74 & 114.20 & 113.11 \\

\bottomrule
\end{tabular}
\end{table}

The results are given in Table \ref{tab:KL.risk.setup.B}.
We would like to specifically point out that it is not intended to identify a uniformly best prior.
In the case where the signal is well-separated from the noise (like in Table \ref{tab:KL.risk.setup.A} of Appendix \ref{sec:risk.sim.setup.A}), the spike-and-slab family with a clear Dirac spike on zero attains lower KL risk in most fixed-hyperparameter setups.
With weak signals present, the Horseshoe with a soft continuous spike near zero achieves lower KL risk in most fixed-hyperparameter setups.
It may be wise to choose a more suitable prior with respect to the assumptions in the data.

The most valuable insight reveals itself if we treat fixed-hyperparameter results as benchmarks for the notion of sparsity learned by each hierarchical prior.
DSnS-Beta and HS-Exp are the only two hierarchical full-Bayes methods for predictive inference that, to our knowledge, currently have theoretical guarantees.
Their performance closely imitates the oracle settings with sparsity level known as the number of significant signals (as if $s_n = s_n^\ast$), as opposed to the number of all signals (as if $s_n=300$).
This indicates that, while able to adapt to an unknown sparsity, these full-Bayes hierarchical models only capture the strong signals, and usually cannot distinguish between the weak signals and the noise.
Such a finding agrees with what we learned from the posterior of the hyperparameters in Section \ref{sec:posterior.experiment}.
Overall, this phenomenon may explain the reason why we imposed a theta-min condition (Assumption \ref{assumption:beta.min}) in theory.

\section{Horseshoe Predictive Inference in Practice}\label{sec:real.data.analysis}

In this section, we demonstrate the practical efficacy of our predictive inference method on two real-world datasets, one in image processing, the other in functional analysis of time series.
The Gaussian sequence model that we have considered in this paper can be seamlessly connected to these setups, despite the apparently strict conditions.
We may transform the data into frequency scale using wavelet decomposition for images and functional principal component analysis for time series.
Because of the underlying smoothness, the coefficients obtained by these procedures can be regarded as realizations from a Gaussian distribution with a sparse mean vector and unknown sparsity level.
Our proposed adaptive predictive inference method based on Gaussian sequence model is then a suitable tool.

There is a notable distinction between estimation and predictive inference.
For the tasks of identifying the true underlying mean, e.g. signal denoising, Gaussian mean estimation generally suffices.
Inferential tasks, on the contrary, requires uncertainty quantification, which falls into the scope of this paper.
Such tasks encompass a wide range of real-world applications.
The most immediate is forecasting, where the goal is to quantify the uncertainty of a future observation.
We would like to emphasize, however, a powerful yet usually overlooked application of {\it anomaly detection}.
Based on the observation $\bm Y$, we may construct a predictive set that can be interpreted as the likely range for the next observation to occur, given that the baseline distribution is consistent.
We may thus test whether the observed $\tilde{\bm Y}$ constitutes a significant deviation in distribution.
In this context, predictive inference serves as a rigorous quality control mechanism; it allows us to determine whether $\tilde{\bm Y}$ falls within the same probabilistic envelope as $\bm Y$, effectively testing for discrepancies.

\subsection{Facial Recognition Study on JAFFE Dataset}\label{sec:JAFFE}

The first example is the Japanese Female Facial Expression (JAFFE) Dataset \citep{lyons2020coding, lyons2021excavating} which contains 213 greyscale images (size $256\times256$) of 10 female posers with different facial expressions.
Our goal is to test whether any two images belong to the same person, despite the potentially different facial expressions.
Repeating this task throughout all pairs of images recovers a graph, where each image is a node, and each edge represents identification.
It is then possible to cluster images into groups.
The images within the same group is then recognized as the same person.

\begin{figure}[t]
    \spacingset{1} 
    \centering
    
    \includegraphics[width=0.15\textwidth, height=0.15\textwidth, keepaspectratio=false]{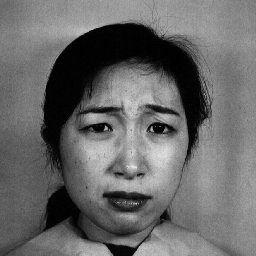}
    \includegraphics[width=0.15\textwidth, height=0.15\textwidth, keepaspectratio=false]{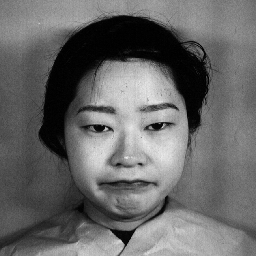}
    \includegraphics[width=0.15\textwidth, height=0.15\textwidth, keepaspectratio=false]{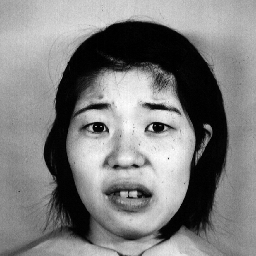}
    \includegraphics[width=0.15\textwidth, height=0.15\textwidth, keepaspectratio=false]{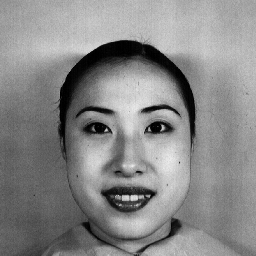}
    \includegraphics[width=0.15\textwidth, height=0.15\textwidth, keepaspectratio=false]{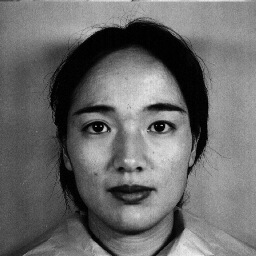}
    
    \vspace{0.05in} 
    
    \includegraphics[width=0.15\textwidth, height=0.15\textwidth, keepaspectratio=false]{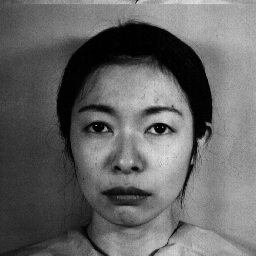}
    \includegraphics[width=0.15\textwidth, height=0.15\textwidth, keepaspectratio=false]{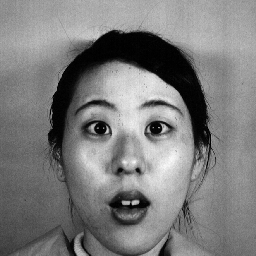}
    \includegraphics[width=0.15\textwidth, height=0.15\textwidth, keepaspectratio=false]{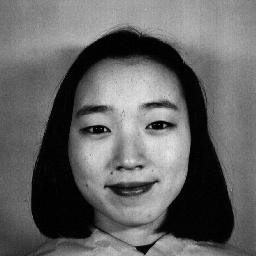}
    \includegraphics[width=0.15\textwidth, height=0.15\textwidth, keepaspectratio=false]{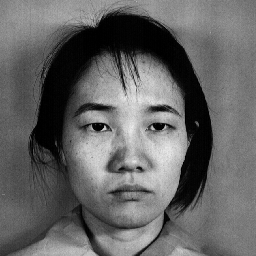}
    \includegraphics[width=0.15\textwidth, height=0.15\textwidth, keepaspectratio=false]{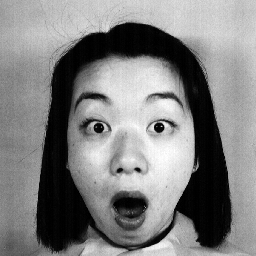}
    
    \caption{Samples from the JAFFE dataset.}
    \label{fig:jaffe.samples}
\end{figure}

We process each image with Daubechies-4 wavelet decomposition (see Appendix \ref{sec:daubechies} for details).
The coefficients satisfy the Gaussian sequence model setup, and can be used for Horseshoe predictive inference.
Consider that we observe a vectorized wavelet coefficient $\bm y_{i_1} \in \R^{n}$ from image $i_1$, standardized to unit variance.
The proposed Horseshoe predictive inference method is capable of recovering noisy observations $\{\hat{\bm y}_{i_1}^{(l)}\}_{1\leq l\leq N}$ from the same mean vector.
Here, the predictive sample size is taken as $N=10,000$ throughout our experiment.
We may then compare these samples from the predictive distribution to the Daubechies-4 wavelet coefficient $\bm y_{i_2}$ from another image $i_2$.
Its proximity to the distribution is measured by the \emph{energy score}:
\begin{equation}\label{eq:JAFFE.energy.score}
E_{i_1, i_2} = \frac{1}{N} \sum_{l=1}^N \| \hat{\bm y}_{i_1}^{(l)} - \bm y_{i_2} \|_2 - \frac{1}{2N^2} \sum_{k=1}^N \sum_{l=1}^N \| \hat{\bm y}_{i_1}^{(k)} - \hat{\bm y}_{i_1}^{(l)} \|_2.
\end{equation}
An energy score closer to zero means that the other image is close to the Bayesian predictive distribution, which suggests that the other observed image may be based on the same sparse mean vector.
On the contrary, a large energy score represents a large discrepancy detected.
For all 213 images, we can therefore run pairwise evaluations and obtain a $213\times 213$ square matrix $\bm E$.
The diagonal entries of $\bm E$ are set to be zero.
Note that $\bm E$ is not necessarily symmetric.
For any $1\leq i_1\ne i_2 \leq 213$, generally $E_{i_1,i_2} \neq E_{i_2,i_1}$.
Figure \ref{fig:JAFFE.E.matrix} shows a visualization of matrix $\bm E$.

\begin{figure}[t]
    \centering
    \begin{subfigure}[b]{0.3\textwidth}
        \centering
        \includegraphics[width=\textwidth]{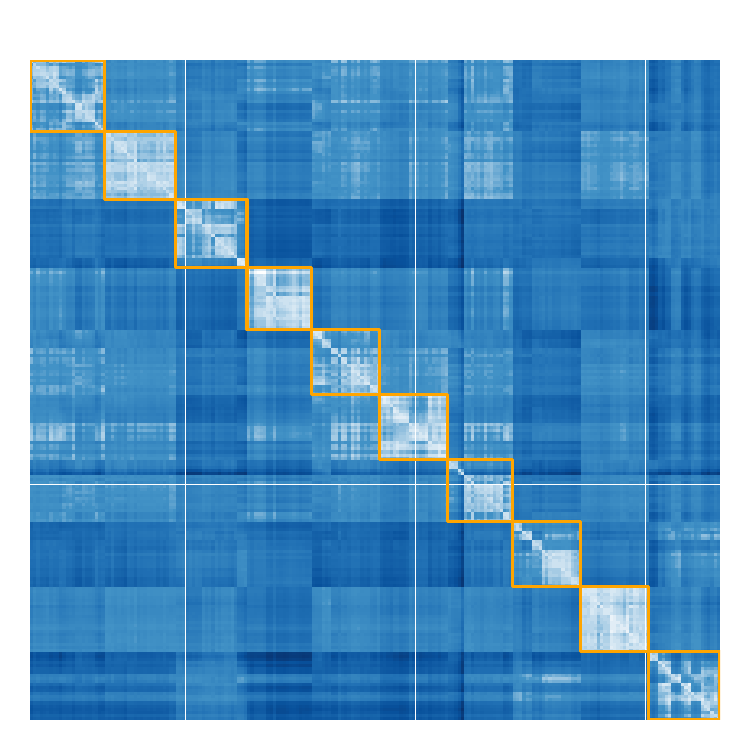}
        \caption{Energy score matrix $\bm E$}
        \label{fig:JAFFE.E.matrix}
    \end{subfigure}
    \hfill 
    \begin{subfigure}[b]{0.3\textwidth}
        \centering
        \includegraphics[width=\textwidth]{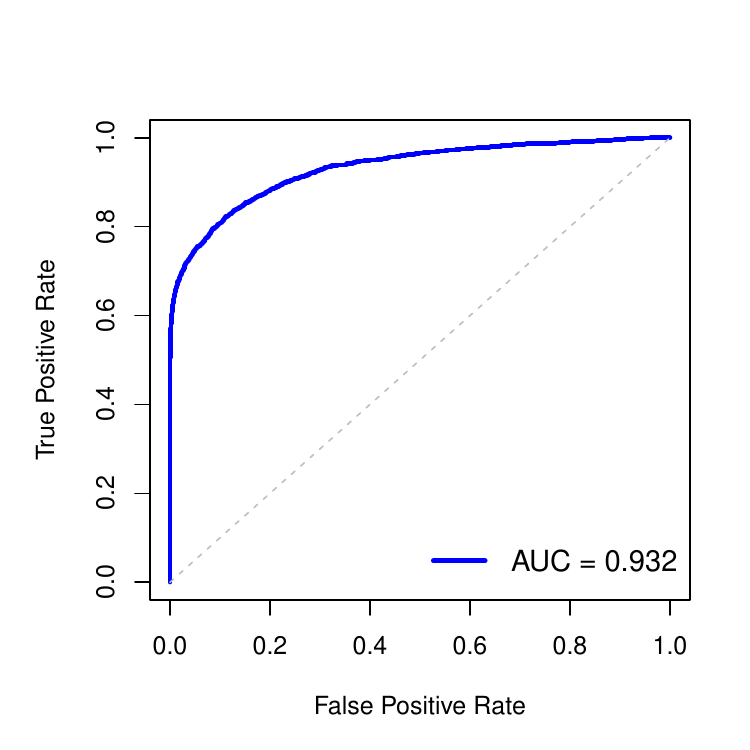}
        \caption{ROC Curve}
        \label{fig:JAFFE.ROC.E}
    \end{subfigure}
    \hfill 
    \begin{subfigure}[b]{0.3\textwidth}
        \centering
        \includegraphics[width=\textwidth]{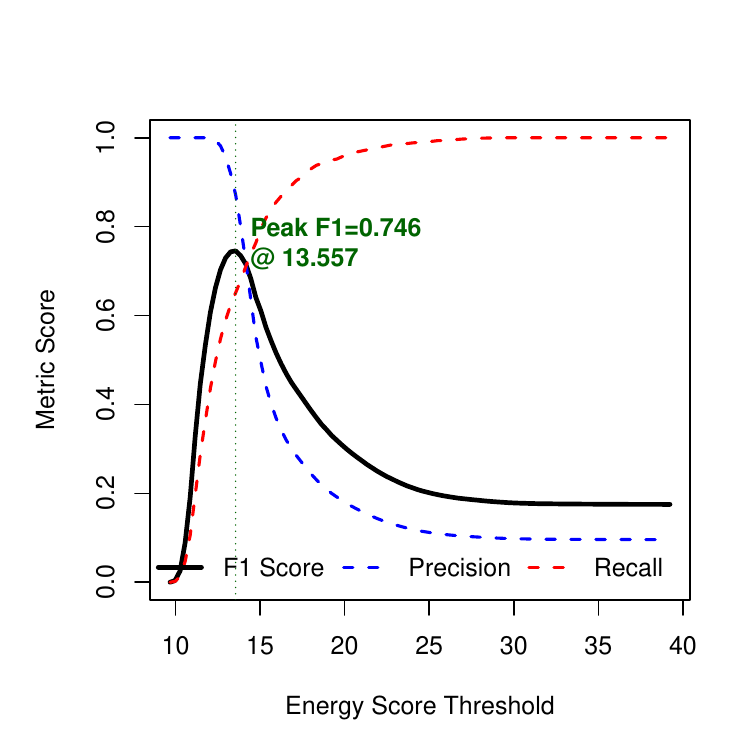}
        \caption{Accuracy of varying $\bar E$}
        \label{fig:JAFFE.threshold.E}
    \end{subfigure}
    
    \caption{Results of the Horseshoe predictive inference method on the JAFFE dataset, using energy score as the metric.
    (a) Heatmap of energy score matrix $\bm E$. 
    Each row and column represents an image.
    Larger energy scores are represented by darker blue.
    The yellow boxes represent images of one subject.
    (b) The ROC curve obtained by varying the pairwise testing cutoff value $\bar E$.
    (c) In-sample selection of optimal threshold $\bar E$, with precision, recall, and F1 score plotted against varying $\bar E$.}
    \label{fig:JAFFE.E.varying}
\end{figure}

To test any pair of images $i_1$ and $i_2$, we build a symmetric matrix $\tilde{\bm E}$, such that
$\tilde E_{i_1, i_2} = \tilde E_{i_2, i_1} = (E_{i_1,i_2} + E_{i_2,i_1}) / 2$,
and accept that the pair of images $i_1$ and $i_2$ are from the same subject if $\tilde E_{i_1, i_2}$ is less than a threshold $\bar E$.
This threshold can be treated as a tuning parameter.
A valid pairwise test should maximize the true positive rate (i.e., accepting when the images are indeed of the same subject), while minimizing the false positive rate (i.e., accepting when the images are of different subjects).
By varying the cutoff value, we obtain a ROC curve, as shown in Figure \ref{fig:JAFFE.ROC.E}.
Our Horseshoe predictive inference method achieves an area under curve (AUC) of 0.932.
We can also plot global precision, recall, and F1 score versus the varying threshold $\bar E$; see Figure \ref{fig:JAFFE.threshold.E}.
When $\bar E = 13.557$, we reach a maximum global F1 score of 0.746.

One caveat, however, is that this choice of $\bar E$ is based on knowledge of true labels.
It is not a valid method to choose thresholds.
Alternatively, we can consider the following three ways for a data-driven choice of $\bar E$:
\begin{itemize}
    \item In the applications where we have a certain amount of labeled images beforehand, we can use those images as a held-out set, on which we run an in-sample selection of optimal threshold $\bar E$.
    In our example, we randomly select one third of the images as held-out set.
    This led to the choice of $\bar E = 13.27$.

    \item In the application cases where we know approximately the number of subjects, we may focus on the number of resulting clusters.
    In our example, we target the number of clusters to be exactly 10, which corresponds to an interval of $\bar E$.
    Its center is selected as the cutoff value.
    This led to $\bar E = 13.62$.

    \item In the fully unsupervised case where none of such information is available, we may consider drawing a histogram of all energy scores.
    The distribution is dual-modal, separating matches and non-matches.
    We can therefore find the valley between the two modes.
    This led to $\bar E = 12.78$.
\end{itemize}
The accuracy of the Horseshoe predictive inference method under each choice of $\bar E$ is demonstrated in Figure \ref{fig:JAFFE.E.results}, in terms of subject-wise precision, recall, and F1 score.

\begin{figure}[t]
    \centering
    \begin{subfigure}[b]{0.24\textwidth}
        \centering
        \includegraphics[width=\textwidth]{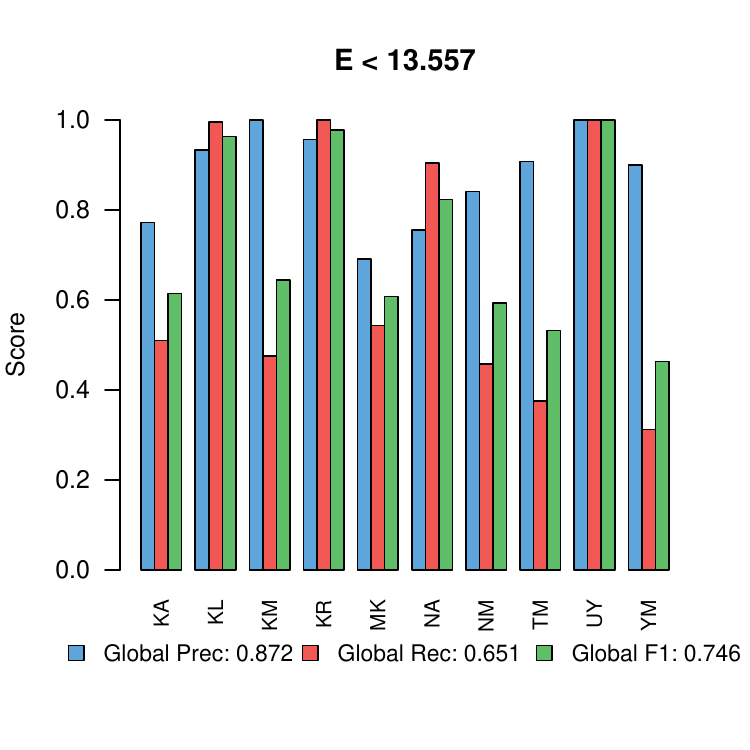}
        \caption{Oracle choice $\bar E$}
        \label{fig:E.MEAN.oracle}
    \end{subfigure}
    \begin{subfigure}[b]{0.24\textwidth}
        \centering
        \includegraphics[width=\textwidth]{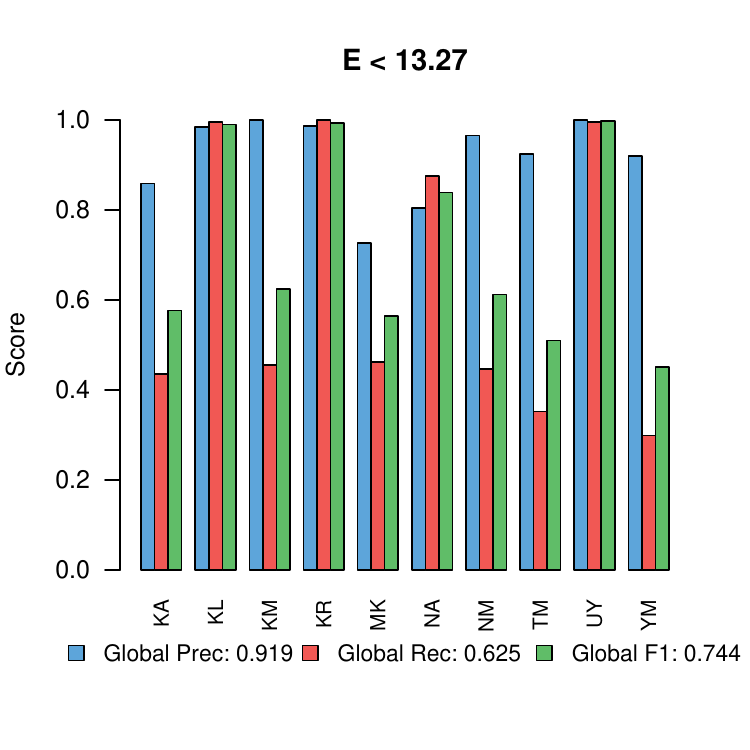}
        \caption{$\bar E$ from held-out set}
        \label{fig:E.MEAN.heldout}
    \end{subfigure}
    \begin{subfigure}[b]{0.24\textwidth}
        \centering
        \includegraphics[width=\textwidth]{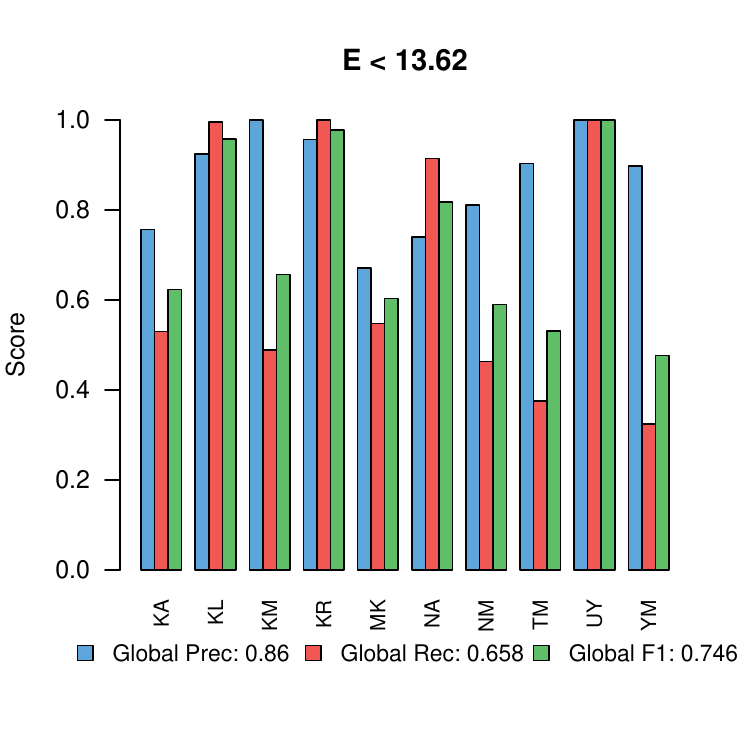}
        \caption{$\bar E$ for 10 clusters}
        \label{fig:E.MEAN.cluster}
    \end{subfigure}
    \begin{subfigure}[b]{0.24\textwidth}
        \centering
        \includegraphics[width=\textwidth]{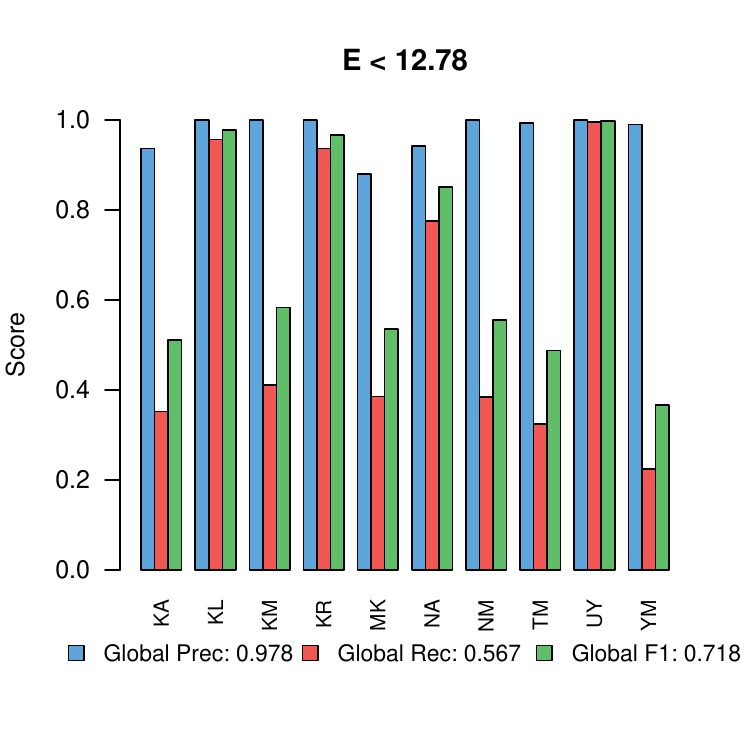}
        \caption{Valley $\bar E$}
        \label{fig:E.MEAN.valley}
    \end{subfigure}
    
    \caption{Accuracy of JAFFE facial recognition under different choices of cutoff $\bar E$. The predictive inference is done using the Horseshoe prior.}
    \label{fig:JAFFE.E.results}
\end{figure}

As a comparison, we repeat the experiment by calculating the predictive distribution using standard Gaussian prior $\theta_i \sim N(0,1)$ instead of the Horseshoe prior.
Without a shrinkage prior, the performance on the JAFFE dataset becomes less satisfactory.
The AUC-ROC of the Gaussian prior is 0.856 compared to the Horseshoe prior's 0.932.
The maximum global F1 score achieved by the Gaussian prior is 0.627, as opposed to the Horseshoe prior's 0.746.
A noteworthy issue caused by the lack of shrinkage is the misclassification of the entire subject KM; see, e.g., Figure \ref{fig:E.MEAN.oracle.Gaussian}.
The performance also deteriorates when choosing the data-adaptive cutoff $\bar E$ using certain methods like targeting 10 clusters (see Figure \ref{fig:E.MEAN.cluster.Gaussian}) and finding valley between two modes (see Figure \ref{fig:E.MEAN.valley.Gaussian}).

\begin{figure}[t]
    \centering
    \begin{subfigure}[b]{0.3\textwidth}
        \centering
        \includegraphics[width=\textwidth]{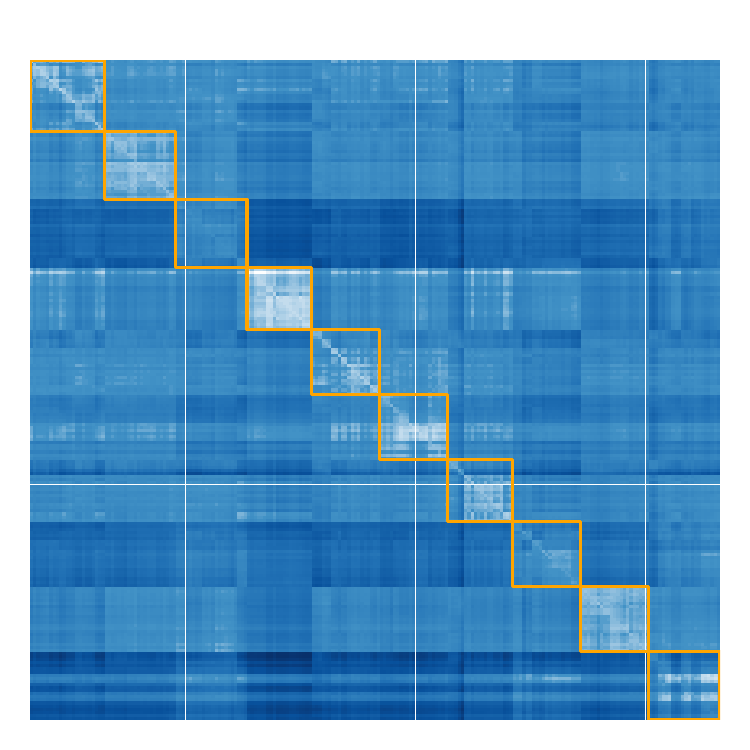}
        \caption{Energy score matrix $\bm E$}
        \label{fig:JAFFE.E.matrix.Gaussian}
    \end{subfigure}
    \hfill 
    \begin{subfigure}[b]{0.3\textwidth}
        \centering
        \includegraphics[width=\textwidth]{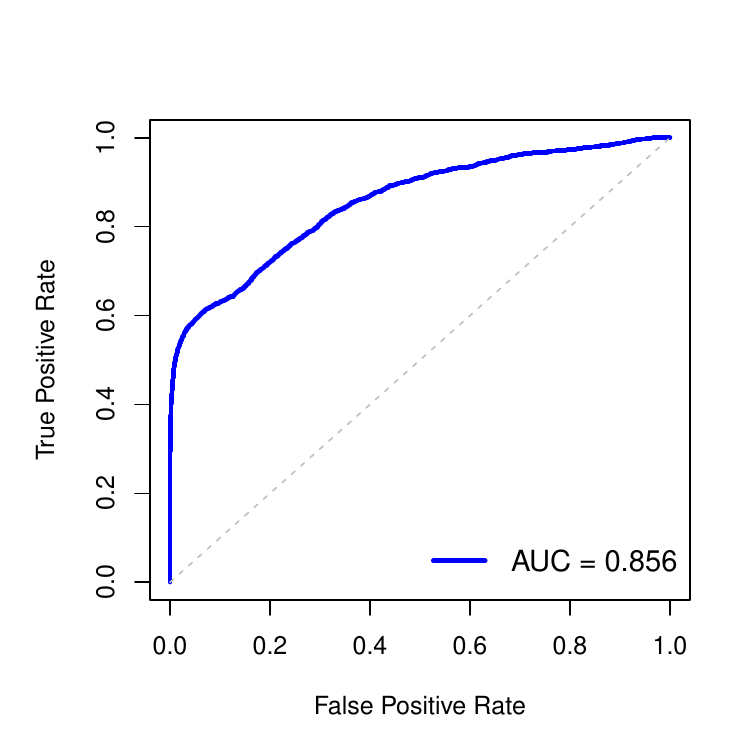}
        \caption{ROC Curve}
        \label{fig:JAFFE.ROC.E.Gaussian}
    \end{subfigure}
    \hfill 
    \begin{subfigure}[b]{0.3\textwidth}
        \centering
        \includegraphics[width=\textwidth]{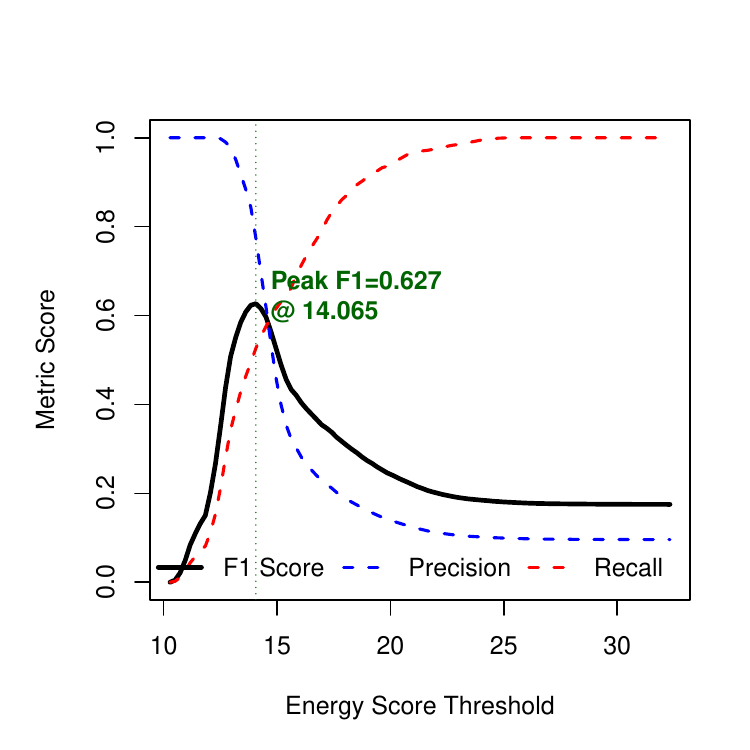}
        \caption{Accuracy of varying $\bar E$}
        \label{fig:JAFFE.threshold.E.Gaussian}
    \end{subfigure}
    
    \caption{Results of the Gaussian predictive inference method on the JAFFE dataset, using energy score as the metric.
    (a) Heatmap of energy score matrix $\bm E$. 
    Each row and column represents an image.
    Larger energy scores are represented by darker blue.
    The yellow boxes represent images of one subject.
    (b) The ROC curve obtained by varying the pairwise testing cutoff value $\bar E$.
    (c) In-sample selection of optimal threshold $\bar E$, with precision, recall, and F1 score plotted against varying $\bar E$.}
    \label{fig:JAFFE.E.varying.Gaussian}
\end{figure}

\begin{figure}[t]
    \centering
    \begin{subfigure}[b]{0.24\textwidth}
        \centering
        \includegraphics[width=\textwidth]{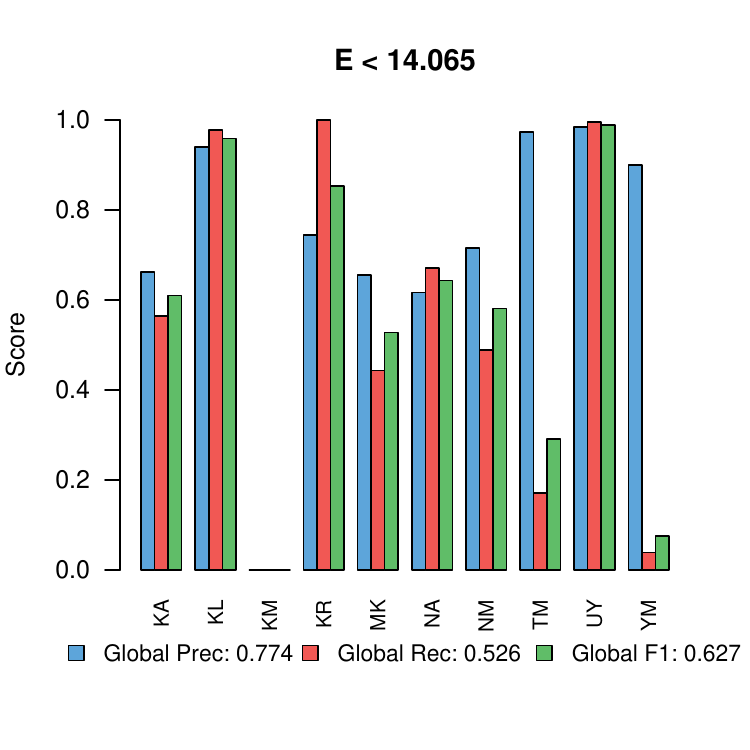}
        \caption{Oracle $\bar E$}
        \label{fig:E.MEAN.oracle.Gaussian}
    \end{subfigure}
    \begin{subfigure}[b]{0.24\textwidth}
        \centering
        \includegraphics[width=\textwidth]{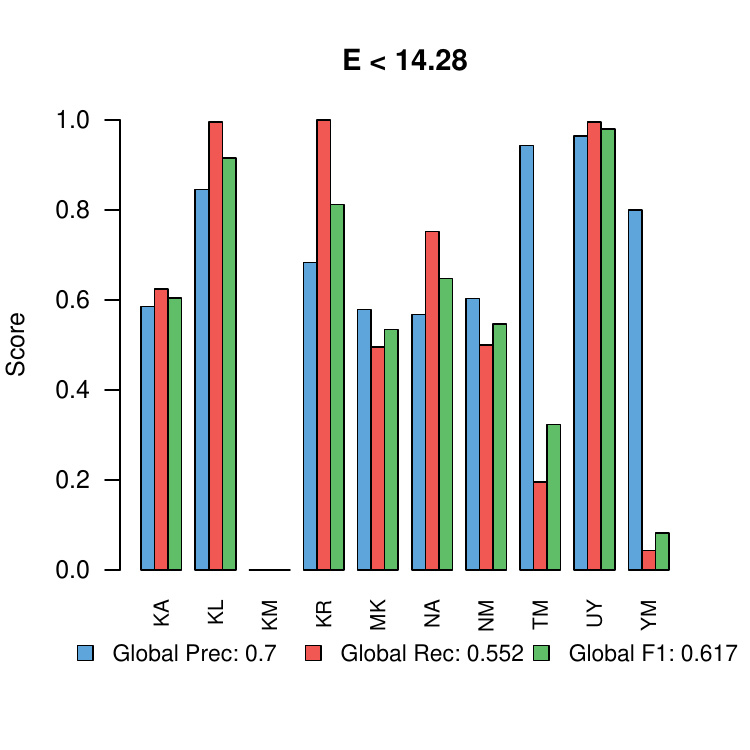}
        \caption{$\bar E$ from held-out set}
        \label{fig:E.MEAN.heldout.Gaussian}
    \end{subfigure}
    \begin{subfigure}[b]{0.24\textwidth}
        \centering
        \includegraphics[width=\textwidth]{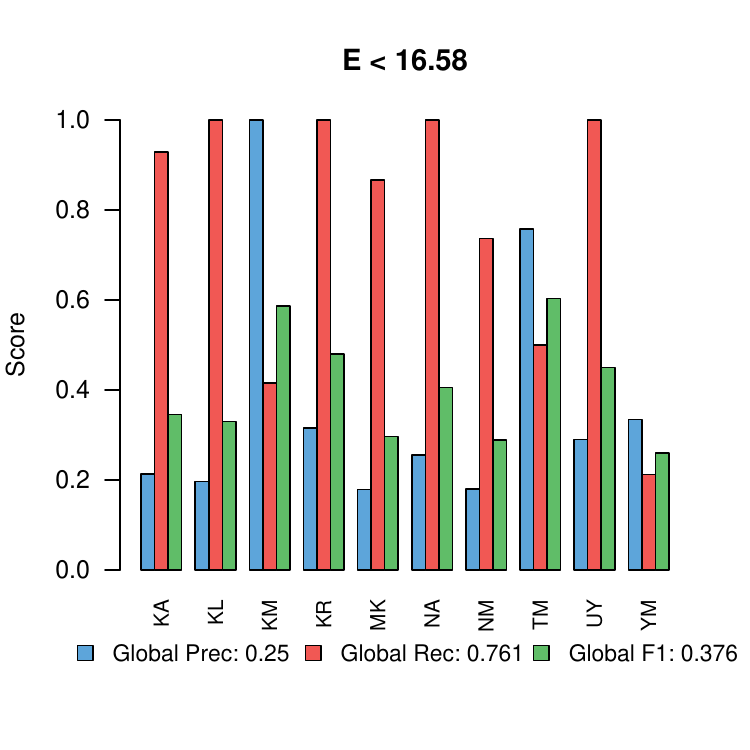}
        \caption{$\bar E$ for 10 clusters}
        \label{fig:E.MEAN.cluster.Gaussian}
    \end{subfigure}
    \begin{subfigure}[b]{0.24\textwidth}
        \centering
        \includegraphics[width=\textwidth]{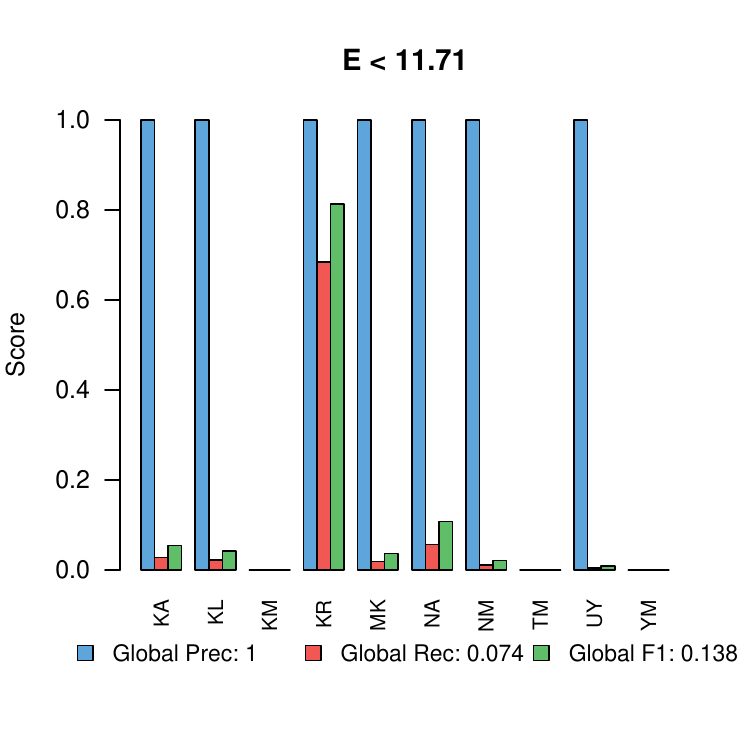}
        \caption{Valley $\bar E$}
        \label{fig:E.MEAN.valley.Gaussian}
    \end{subfigure}
    
    \caption{Accuracy of JAFFE facial recognition under different choices of cutoff $\bar E$. The predictive inference is done using Gaussian prior.}
    \label{fig:JAFFE.E.results.Gaussian}
\end{figure}

\subsection{Brain Lateralization Study on ABIDE Dataset}\label{sec:ABIDE}

The second experiment is focused on functional magnetic resonance imaging (fMRI) brain signal analysis on the Autism Brain Imaging Data Exchange (ABIDE) dataset \citep{di2014autism}.
The time series observations are regarded as noisy realizations of true underlying signals as a continuous function in time.
We analyze fMRI scans from 73 individuals with Autism Spectrum Disorder (ASD), a chronic neurodevelopmental disorder \citep{christensen2016prevalence}, and 98 healthy controls.
Raw brain magnetic resonance images are segmented into temporal signals for 116 regions of interest (ROIs) using the automatic anatomic labeling (AAL) parcellation approach.
We use the preprocessed data by \cite{craddock2013neuro} that contains $p=116$ time series for each patient, each representing a signal in the respective ROI.

For patient $i$, the time series of $j$-th ROI, denoted by $g_{ij}(t)$, can be treated as a noisy observation of a time-dependent functional.
We use functional principal component analysis (FPCA), a data-driven method that converts a functional observation to a set of basis and coefficients \citep{hsing2015theoretical}, the latter also named FPC scores.
In practice, we may estimate the FPC scores of the first $M$ components, $\{\hat a_{ijm}\}_{m=1}^M$ for each patient $i$ and ROI $j$, standardized by the estimated eigenvalues $\hat\lambda_{jm}$ of covariance kernel.
We use the standardized FPC scores $y_{ijm} = \hat{a}_{ijm} / \hat\lambda_{jm}^{1/2}$.
The resulting vector $\bm y_{ij} = (y_{ij1}, \cdots, y_{ijM})$ can thus be regarded as an independent random vector that is approximately Gaussian with unit variance due to the Karhunen-Lo\`eve Theorem \citep{bosq2000linear}.
The reasoning of FPCA method is given in Appendix \ref{sec:apdx.FPCA}.
A more detailed discussion can be found in, e.g. \cite{zhao2024high}.

Excluding vermis, a midline structure in the cerebellum, 108 of the 116 ROIs in the AAL atlas form 54 symmetric left-right pairs.
We can thus investigate atypical inter-hemispheric coordination, an important topic in ASD studies.
Neurotypical brains exhibit some degree of functional lateralization.
Specific tasks such as language processing are usually localized to one hemisphere.
Our predictive inference method provides a framework for detecting ASD-related \emph{localized} deviations from baseline.
By testing whether the signal in one hemisphere accurately predicts its contralateral counterpart in all 54 pairs of ROIs, our method pinpoints which regions in ASD brains suffer from abnormal hyper- or hypo-symmetry.

For any pair of ROIs, $j_1$ and $j_2$, we may sample $N$ instances from the predictive set of $\bm y_{ij_1}$ with our proposed Horseshoe predictive inference method.
These instances are denoted by $\{\hat{\bm y}_{ij_1}^{(l)}\}_{1\leq l\leq N}$.
We can then compare their distribution to $\bm y_{ij_2}$ observed in the other ROI.
Specifically, we consider the $k$-th pair of ROIs, and let $j_1 = 2k-1$, and $j_2 = 2k$.
Their proximity can be measured by the energy score.
Similar to \eqref{eq:JAFFE.energy.score}, we define the energy score from ROI $j_1$ to ROI $j_2$ for subject $i$ as
\begin{equation}\label{eq:ABIDE.energy.score}
    E^i_{j_1, j_2} = \frac{1}{N}\sum_{l=1}^N \| \hat{\bm y}_{ij_1}^{(l)} - \bm y_{ij_2}\|_2 - \frac{1}{2N^2} \sum_{k=1}^N\sum_{l=1}^N \| \hat{\bm y}_{ij_1}^{(k)} - \hat{\bm y}_{ij_1}^{(l)}\|_2.
\end{equation}
This energy score itself measures the extent of symmetry between any given corresponding left-right pair of ROIs, which is interesting on its own.
However, the practitioners may find it more useful to study the increase or decrease of the extent of symmetry.
For such a test for ROI pair $k$ that contains regions $j_1$ and $j_2$, we consider the aggregated energy score $\tilde E^i_{j_1,j_2} = (E^i_{j_1, j_2} + E^i_{j_2, j_1})/2$.
We compare the distribution of $\{\tilde E^i_{j_1,j_2}\}_{i\in \mathcal I_{\text{ASD}}}$ versus that of $\{\tilde E^i_{j_1,j_2}\}_{i\in \mathcal I_{\text{Ctrl}}}$ with Wilcoxon rank-sum test, where $\mathcal I_{\text{ASD}}$ and $\mathcal I_{\text{Ctrl}}$ are the index sets of ASD group and Control group patients, respectively.
A significant change in distribution of energy scores represents a change in symmetry pattern of ASD group compared to controls.
Hyper- or hypo-symmetry can be treated as a proxy of change in activity level of these brain regions, and may indicate a different pattern in ASD brains compared to neurotypical ones.

In this experiment, we choose specifically that $M=5$. There are two reasons for this choice.
First, higher-order FPCA basis functions usually capture the high-frequency patterns.
However, our observed time series data is based on mid- to low-frequency signals.
With only 172 time points, higher-order FPC scores will not be able to capture high-frequency signals as expected.
Second, the higher-order eigenvalues are usually very small.
The estimation of $\hat\lambda_{jm}$ for large $m$ may not be reliable, leading to an unstable distribution of $\bm y_{ij}$, whose variance may be far from one.
Overall, if we choose a large $M$, the signals in low-order FPC scores will be dominated by the unstable noise in high-order FPC scores.

Figure \ref{fig:ABIDE.E.boxplot} provides boxplots that provide the distribution of energy scores of both the ASD group and the Control group across all ROI pairs.
The pairs demonstrating significant hypo-symmetry and hyper-symmetry according to the two-sided Wilcoxon rank-sum tests are highlighted in blue or red, respectively.
The corresponding p-values of these tests and their respective Benjamini--Yekutieli corrections are given in Table \ref{tab:ABIDE.wilcoxon} of Appendix \ref{sec:apdx.BY}.
For a more convenient visualization, we also provide a 3-dimensional brain anatomical topography in Figure \ref{fig:ABIDE.E.brain.image}, where the ROIs with significant hypo- or hyper-symmetry are highlighted in blue or red.

\begin{figure}[p]
    \centering
    
    \includegraphics[width=0.95\linewidth]{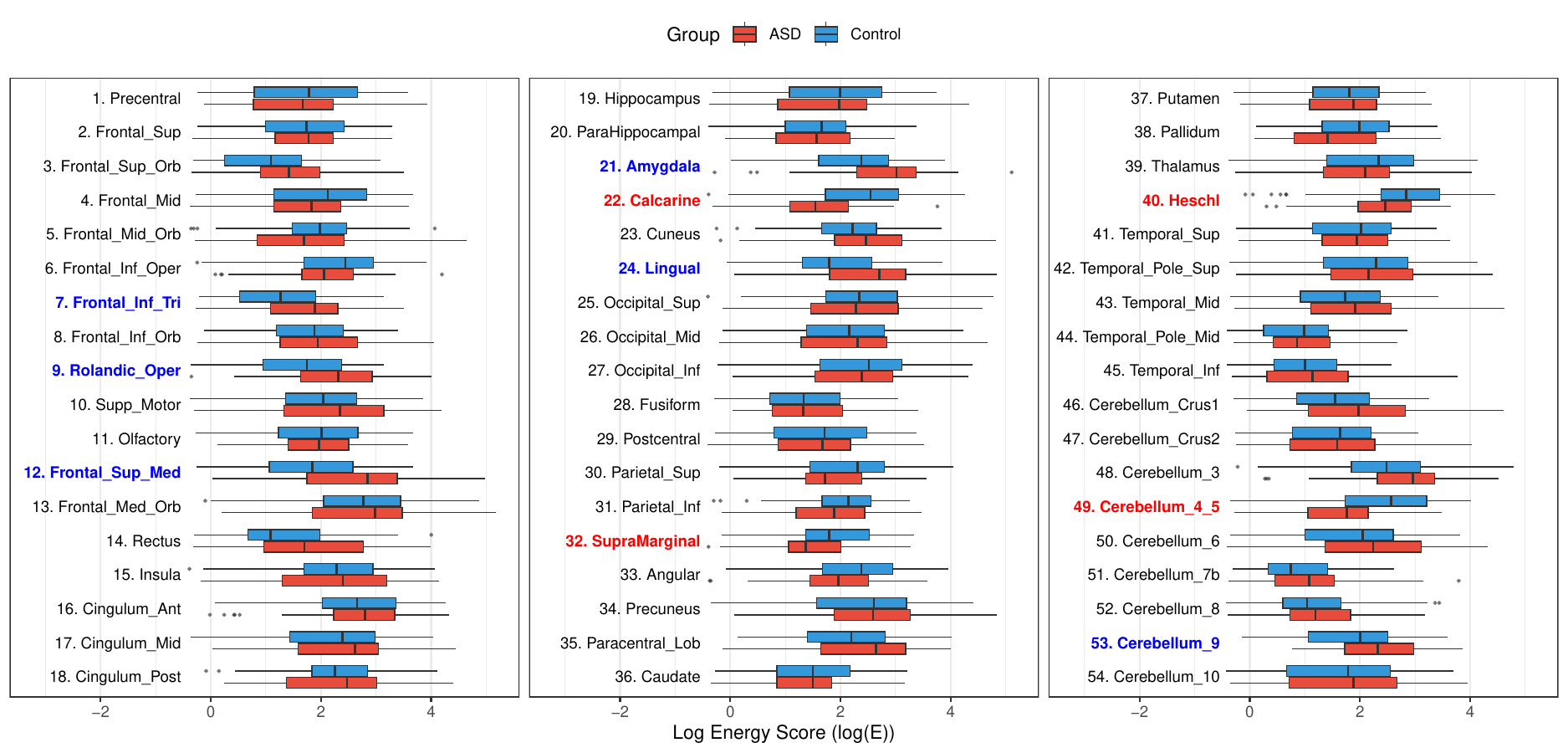}
    \caption{Distribution of energy scores across 54 brain regions.
    Boxplots display the distribution of log energy scores for the ASD group (red) and control group (blue).
    Region labels are color-coded based on significant group differences in Wilcoxon rank-sum test, subject to Benjamini--Yekutieli correction.
    Hypo-symmetric (i.e., ASD less symmetric than control) regions are marked in blue, while hyper-symmetric (i.e., ASD more symmetric than control) regions are in red.}
    \label{fig:ABIDE.E.boxplot}
    
    \vspace{1cm} 
    
    \includegraphics[width=0.45\linewidth]{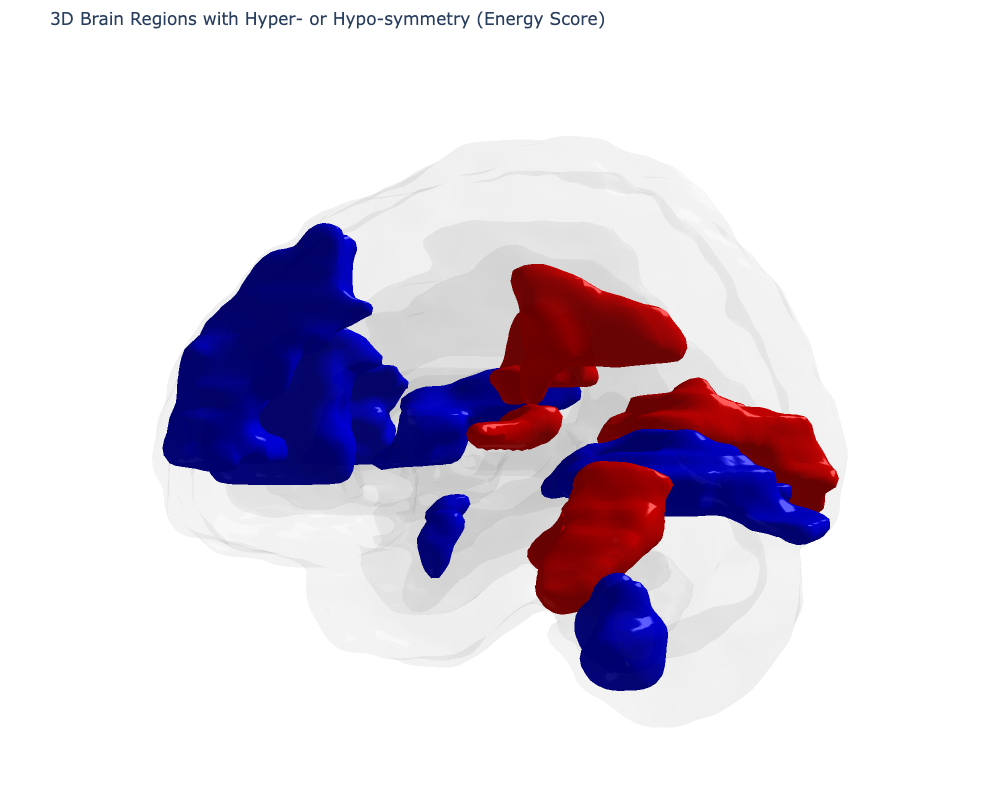}
    \includegraphics[width=0.45\linewidth]{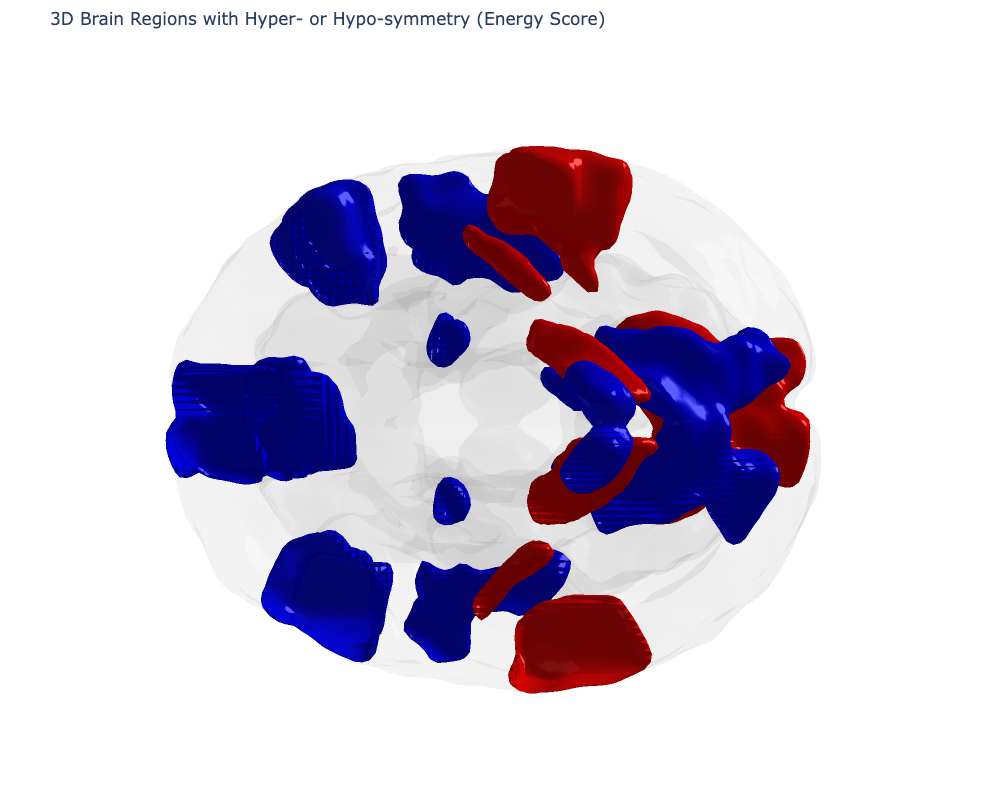}
    \caption{Anatomical topography of brain regions with significant discrepancy between ASD and control groups.
    Hypo-symmetric regions are marked in blue, while hyper-symmetric regions are in red.
    The first image shows the view from the left, while the second image shows the view from above.
    In both images, the front of the brain is on the left, and the rear is on the right.}
    \label{fig:ABIDE.E.brain.image}
    
\end{figure}

An immediate first impression from Figure \ref{fig:ABIDE.E.brain.image} is that, compared to neurotypical brains, the frontal lobe (located in the front) of the ASD brains generally exhibits more lateralization, while the parietal lobe (located in the top-back) exhibits more synchronization.
This broad pattern is consistent with the frontal underconnectivity and posterior overconnectivity discussed by \cite{just2004cortical}.
These findings suggest that the symmetry patterns detected by our predictive approach may align with previously reported differences between higher-order frontal systems and posterior sensory-association systems in ASD.

Among the regions with the strongest hypo-symmetry, the {\it amygdala} is involved in processing emotional salience and facial expressions.
The increased lateralization observed there is therefore consistent with atypical socio-emotional processing often reported in ASD.
The {\it frontal inferior triangularis} and {\it rolandic operculum} are parts of language-related {\it Broca's Area} in the frontal cortex, whose hypo-symmetry may be relevant to communication-related differences that are typical in ASD.
The {\it frontal superior medial}, located in the medial prefrontal cortex, a node in default mode network (DMN), is crucial in higher-order social cognition like interpretation of others' thoughts.
Taken together, these regions overlap with systems frequently implicated in social, emotional, and communication-related differences in ASD.

By contrast, the most hyper-symmetric regions include {\it Heschl's gyrus}, which lies in the auditory cortex and is involved in early auditory processing, including pitch and intensity.
Greater interhemispheric synchronization in this region may be related to atypical auditory sensitivity in ASD, and may also be compatible with reports of enhanced pitch discrimination in some individuals.
Meanwhile, the {\it supramarginal} region, together with other parietal areas, has been associated with attention and sensory integration.
The increased symmetry observed there may be consistent with altered attentional allocation or sensory integration in ASD.

Another interesting pattern is the contrast between the {\it lingual gyrus} and the {\it calcarine sulcus}, two adjacent regions in the medial occipital lobe.
The calcarine sulcus contains primary visual cortex that receives raw information from the retina, whereas the lingual gyrus is involved in subsequent visual processing.
In our results, the calcarine region appears more hyper-symmetric, while the lingual region appears more hypo-symmetric in the ASD group.
Interpreted cautiously, this contrast is consistent with the view in \cite{bertone2005enhanced} and \cite{happe2006weak} that ASD may involve relatively enhanced low-level visual processing together with weaker higher-order visual integration.
It is also broadly compatible with reported differences in face processing \citep{kleinhans2008abnormal} and color perception \citep{franklin2008color}.

\section{Concluding Remarks}

This paper studies the Horseshoe prior from a predictive inference viewpoint in the sparse Gaussian sequence models under the KL loss.
This setup has a wide range of real-world applications.
For instance, wavelet decompositions of images and functional principal component analysis for time series often lead to sparse Gaussian sequence models.
Predictive inference enables pairwise tests that associate images or time series with structural similarities.

Empirically, we find  that the Horseshoe prior is a good fit for such problems.
Theoretically, we find that the posterior of the local shrinkage scale $\lambda$ serves as a key spectral weight.
Its phase transition controls the predictive risk, leading to a smooth thresholding behavior similar to the one seen in discrete mixtures.
With known sparsity level, the Horseshoe prior achieves the exact-minimax predictive KL risk asymptotically.
Moreover, the full-Bayes approach attains a sharp rate on the predictive KL risk under the theta-min condition.

\vspace{-1cm}
\section*{Data availability statement and Acknowledgement}

The data that support the findings of this study are available from public repositories.
The JAFFE dataset is available at \url{https://doi.org/10.5281/zenodo.14974867}.
The ABIDE dataset is available at \url{https://datacatalog.med.nyu.edu/dataset/10452}.
The authors used ChatGPT (OpenAI, GPT-5.4) to assist with language editing and improving clarity of expression during manuscript preparation.
All scientific content, analysis, interpretations, and conclusions were developed and verified by the authors.

\vspace{-1cm}


\begingroup
\renewcommand{\newline}{} 
\renewcommand{\url}[1]{} 
\renewcommand{\harvardurl}[1]{} 
\bibliography{BP_literature}       
\endgroup

\newpage
\begin{appendix}

\section{Extra Discussions on Separable Horseshoe Prior}

We begin with the following useful lemma for the bound of univariate Horseshoe prior, which is an extension of Theorem 1 of \cite{carvalho2010horseshoe}.
We shall see that the univariate Horseshoe prior has a natural, continuous spike near zero, and a Cauchy-like heavy tail.
\begin{lemma}\label{lm:HS.pdf.UB.LB}
    For a fixed $\tau$, the pdf of the univariate Horseshoe prior satisfies
    \[
    \frac{1}{2\sqrt{2\pi^3}\tau}\log\left(1+\frac{4\tau^2}{\theta^2}\right) < \pi(\theta\mid\tau) < \frac{1}{\sqrt{2\pi^3}\tau}\log\left(1+\frac{2\tau^2}{\theta^2}\right).
    \]
\end{lemma}
\begin{proof}
    Since $\theta \sim N(0, \lambda^2\tau^2)$, $\lambda \sim C^+(0,1)$, we have
    \[
    \pi(\theta\mid\tau) = \int_0^\infty \frac{1}{\sqrt{2\pi}\lambda\tau} e^{-\frac{\theta^2}{2\lambda^2\tau^2}} \frac{2}{\pi}\frac{1}{1+\lambda^2} \diff \lambda.
    \]
    This can be further expressed using the exponential integral function
    $E_1(z) = \int_z^\infty e^{-t}/t \diff t$.
    With a change of variable $u = 1/\lambda^2$,
    \[
    \pi(\theta\mid\tau) = \frac{1}{\sqrt{2\pi^3}} \int_0^\infty e^{-\frac{\theta^2}{2\tau^2}u} \frac{1}{1+u} \diff u.
    \]
    Another change of variable $t = (1+u)\theta^2/2\tau^2$ leads to
    \[
    \pi(\theta\mid\tau) = \frac{1}{\sqrt{2\pi^3}\tau} e^{\frac{\theta^2}{2\tau^2}} \int_{\theta^2/2\tau^2}^\infty \frac{e^{-t}}{t}\diff t = \frac{1}{\sqrt{2\pi^3}\tau} e^{\frac{\theta^2}{2\tau^2}} E_1\left(\frac{\theta^2}{2\tau^2}\right).
    \]
    The proof of the assertion is finished by noticing that the exponential integral function satisfies the following tight upper and lower bounds,
    \[
    \frac{1}{2}e^{-t} \log\left(1+\frac{2}{t}\right) < E_1(t) < e^{-t} \log\left(1+\frac{1}{t}\right).
    \]
\end{proof}

\subsection{Risk Decomposition}\label{sec:risk.decomp}
The following lemma characterizes the exact behavior of univariate Bayesian predictive risk under the continuous Gaussian mixture prior.
\begin{lemma}\label{lm:risk.decomp}
    The univariate Bayesian predictive risk $\rho(\theta, \hat{p})$ has the following decompositions. Let $Z \sim N(0,1)$, and $v^{-1} = 1+r^{-1}$.
    \begin{enumerate}

        \item Assume that the prior of $\theta$ is a scale-mixture of Gaussians, $\theta \mid \lambda \sim N(0, \lambda^2)$, where the prior of $\lambda$ is $\nu(\lambda)$, then
        \begin{equation}\label{eq:risk.decomp.b}
            \rho(\theta,\hat{p}) = \frac{\theta^2}{2r} - \E_\theta \log N^{GM}_{\theta,v}(Z) + \E_\theta \log N^{GM}_{\theta,1}(Z),
        \end{equation}
        where
        \[
        N^{GM}_{\theta,v}(Z) = \int_0^\infty \sqrt{\frac{v}{\lambda^2 + v}} \exp\left[\frac{\lambda^2}{\lambda^2 + v}\frac{(\sqrt{v}Z+\theta)^2}{2v}\right] \nu(\lambda) \diff \lambda.
        \]

        \item Specifically, if $\theta$ follows Horseshoe prior with a fixed $\tau>0$, then
        \begin{equation}\label{eq:risk.decomp.c}
            \rho(\theta,\hat{p}) = \frac{\theta^2}{2r} - \E_\theta \log N^{HS}_{\theta,v}(Z) + \E_\theta \log N^{HS}_{\theta,1}(Z),
        \end{equation}
        where $N^{HS}_{\theta,v}(Z)$ takes any of the following equivalent forms:
        \begin{align}
            N^{HS}_{\theta,v}(Z) &= \frac{\tau}{\pi\sqrt{v}} \int_0^1 u^{-1/2}\frac{1}{\tau^2/v+(1-\tau^2/v)u} \exp\left[\frac{(\sqrt{v}Z + \theta)^2}{2v}u\right] \diff u \label{eq:N.HS.form.A}\\
            &= \frac{\tau}{\pi\sqrt{v}} e^{\frac{(\sqrt{v}Z+\theta)^2}{2v}} \int_0^1 (1-u)^{-1/2}\frac{1}{1-(1-\tau^2/v)u} \exp\left[-\frac{(\sqrt{v}Z + \theta)^2}{2v}u\right] \diff u\label{eq:N.HS.form.B}\\
            &= \frac{2\tau}{\pi\sqrt{v}} e^{\frac{(\sqrt{v}Z+\theta)^2}{2v}} \Phi_1\left(1,1,\frac{3}{2},-\frac{(\sqrt{v}Z+\theta)^2}{2v},1-\frac{\tau^2}{v}\right).\label{eq:N.HS.form.C}
        \end{align}
        Here, $\Phi_1$ is confluent hypergeometric function for two variables; see \cite{gradshteyn2014table}, 9.261.
    \end{enumerate}
\end{lemma}
\begin{proof}
    By Lemma 2.1 of \cite{rockova2023adaptive}, For any prior $\pi(\theta)$,
\begin{equation}\label{eq:risk.decomp.a}
    \rho(\theta,\hat{p}) = \frac{\theta^2}{2r} - \E_\theta \log N_{\theta,v}(Z) + \E_\theta \log N_{\theta,1}(Z),
\end{equation}
where $N_{\theta,v}(Z) = \int_\R \exp\left\{\frac{\mu Z}{\sqrt{v}} + \frac{\mu\theta}{v} - \frac{\mu^2}{2v}\right\} \pi(\mu) \diff\mu$, and $v^{-1} = 1+r^{-1}$.
For the first assertion, note that
\begin{align*}
    N_{\theta,v}^{GM}(Z) &= \int_\R \int_0^\infty \frac{1}{\sqrt{2\pi}\lambda\tau} \exp\left\{\frac{\mu Z}{\sqrt{v}} + \frac{\mu\theta}{v} - \frac{\mu^2}{2v} - \frac{\mu^2}{2\sigma^2}\right\}\nu(\sigma) \diff \mu \diff \sigma\\
    &= \int_0^\infty \sqrt{\frac{v}{\sigma^2 + v}} \int_\R \frac{\sqrt{\sigma^2+v}}{\sqrt{2\pi v}\sigma} \exp\left\{-\frac{\sigma^2+v}{2v\sigma^2}\left(\mu - \frac{\sigma^2}{\sigma^2+v}(\sqrt{v}Z+\theta)\right)^2\right\} \diff \mu\\
    &\hspace{183pt} \cdot\exp\left(\frac{\sigma^2}{\sigma^2+v}\frac{(\sqrt{v}Z+\theta)^2}{2v}\right) \nu(\sigma) \diff \sigma\\
    &= \int_0^\infty \sqrt{\frac{v}{\sigma^2 + v}} \exp\left\{\frac{\sigma^2}{\sigma^2 + v}\frac{(\sqrt{v}Z+\theta)^2}{2v}\right\} \nu(\sigma) \diff \sigma.
\end{align*}
For the second assertion, note that a Horseshoe prior is a Gaussian mixture with $\sigma = \lambda\tau$ and a Half Cauchy prior $\lambda\sim C^+(0,1)$.
By changing variable $u = \frac{\lambda^2\tau^2}{v+\lambda^2\tau^2}$,
\begin{align*}
    N_{\theta,v}^{HS}(Z) &= \frac{2}{\pi} \int_0^\infty \sqrt{\frac{v}{\lambda^2\tau^2 + v}} \exp\left[\frac{\lambda^2\tau^2}{\lambda^2\tau^2 + v}\frac{(\sqrt{v}Z+\theta)^2}{2v}\right] \frac{1}{1+\lambda^2} \diff \lambda\\
    &= \frac{\sqrt{v}}{\pi\tau}\int_0^1 (1-u)^{1/2}\exp\left[\frac{(\sqrt{v}Z+\theta)^2}{2v}u\right]\frac{1-u}{1+(\frac{v}{\tau^2}-1)u} u^{-1/2}(1-u)^{-3/2}\diff u\\
    &= \frac{\tau}{\pi\sqrt{v}}\int_0^1 u^{-1/2}\frac{1}{\tau^2/v+(1-\tau^2/v)u} \exp\left[\frac{(\sqrt{v}Z + \theta)^2}{2v}u\right] \diff u.
\end{align*}
This proves expression \eqref{eq:N.HS.form.A}.
With a change of variable from $u$ to $1-u$, we get expression \eqref{eq:N.HS.form.B}.
In fact, $N_{\theta,v}^{HS}(Z)$ can also be written as
\[
\frac{2\sqrt{v}}{\pi\tau} \Phi_1\left(\frac{1}{2}, 1, \frac{3}{2}, \frac{(\sqrt{v}Z+\theta)^2}{2v}, 1-\frac{v}{\tau^2}\right),
\]
or
\[
\frac{2\tau}{\pi\sqrt{v}} e^{\frac{(\sqrt{v}Z+\theta)^2}{2v}} \Phi_1\left(1,1,\frac{3}{2},-\frac{(\sqrt{v}Z+\theta)^2}{2v},1-\frac{\tau^2}{v}\right),
\]
where $\Phi_1$ is confluent hypergeometric function for two variables; see \cite{gradshteyn2014table}, 9.261.
\end{proof}

Lemma \ref{lm:risk.decomp} is a generalization of Theorem 2.1 in \cite{mukherjee2015exact} for point-mass spike-and-slab priors and Lemma 2.6 in \cite{rockova2023adaptive} for any spike-and-slab prior.
Compared to these previous results, our first assertion presents a decomposition of predictive risk under any Gaussian mixture prior.

\begin{figure}
    \centering
    \includegraphics[width=0.4\linewidth]{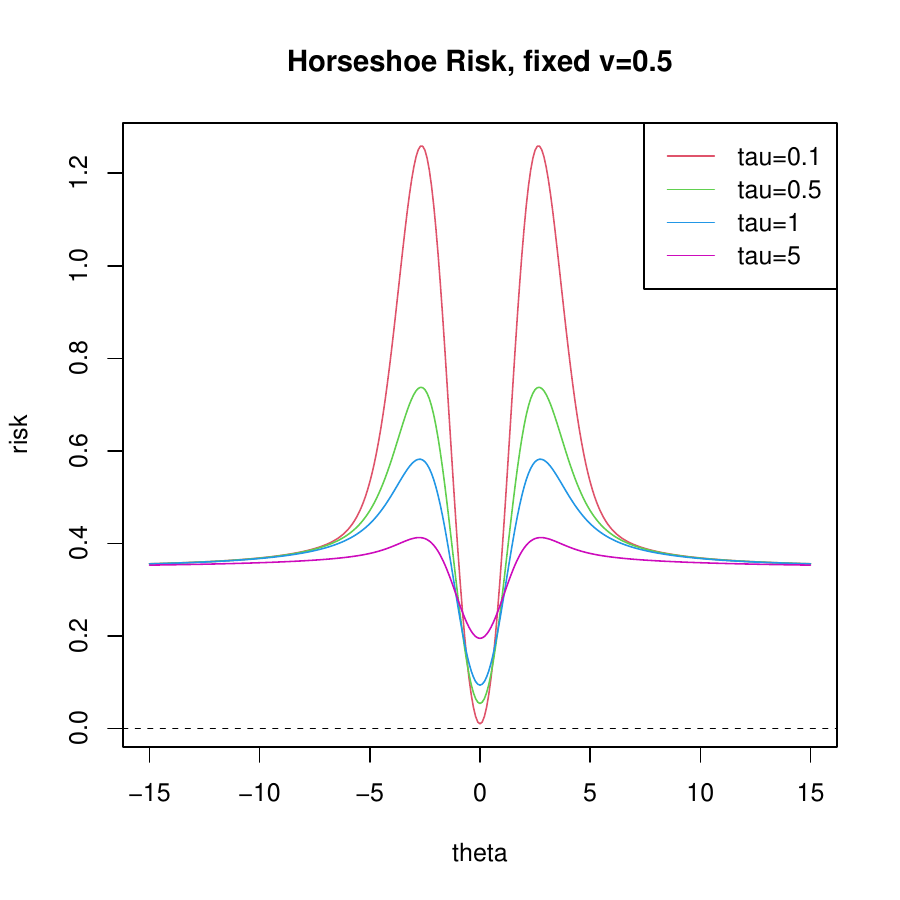}
    \includegraphics[width=0.4\linewidth]{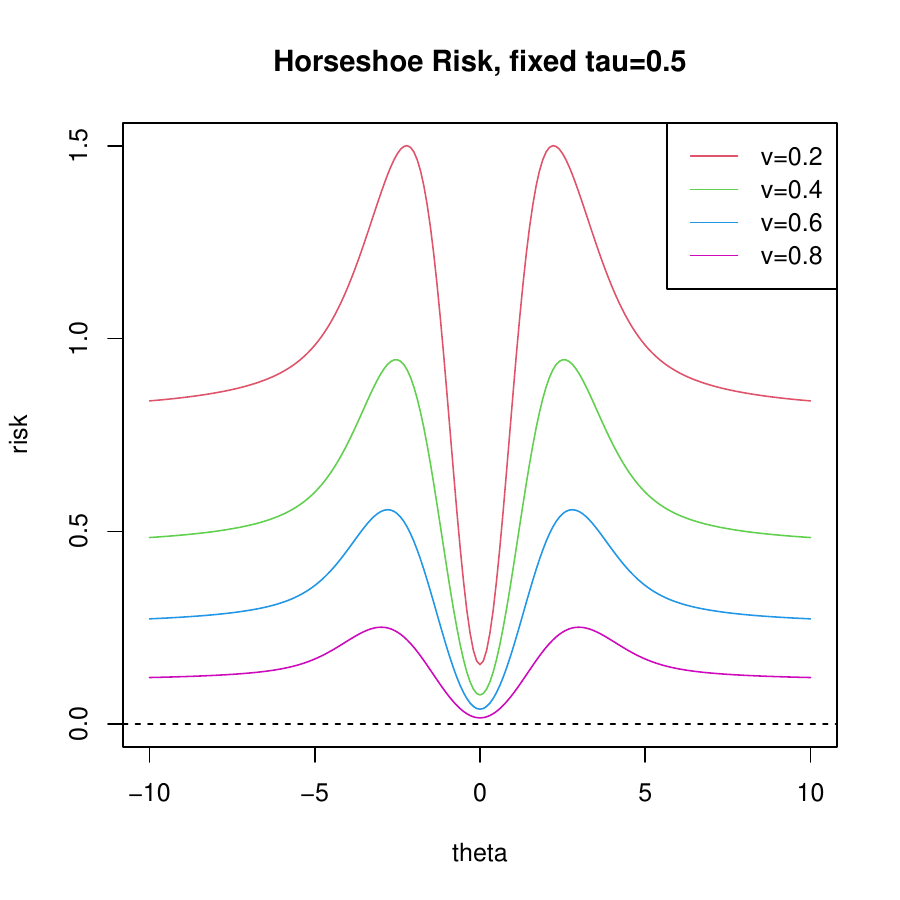}
    \caption{The univariate Horseshoe predictive Kullback-Leibler risk $\rho(\theta, \hat p)$}
    \label{fig:risk}
\end{figure}

From Lemma \ref{lm:risk.decomp}, we can draw the plots of the univariate predictive risk $\rho(\theta,\hat p)$ versus $\theta$ under different values of $\tau$ and $v$.
These plots are shown in Figure \ref{fig:risk}.
When $\tau$ is fixed, a larger value of $v$ generally corresponds to lower risk over all $\theta \in \R$, which agrees with the $(1-v)$ factor in the minimax risk.
On the contrary, with a fixed value of $v$, we see a tradeoff between the risk at zero $\rho(0,\hat p)$ and the maximum risk $\sup_{\theta\in \R}\rho(\theta, \hat p)$.
As $\tau$ decreases toward zero, $\rho(0,\hat p)$ approaches zero, while the maximum risk increases rapidly.
Recall that the overall multivariate predictive risk $\rho_n(\bm\theta, \hat p_\pi)$ is upper bounded by a weighted combination of the two.
This is intuitively indicative to the role that the global shrinkage parameter $\tau$ plays in the Horseshoe prior.
For a lower sparsity level $s_n$, we would like to set a smaller $\tau$ to lower the risk at zero, at the cost of a higher maximum risk.
To achieve this, the global shrinkage parameter $\tau$ should be positively correlated to the sparsity level $s_n$.

Note that this tradeoff is also present in spike-and-slab priors; see Figures 1 and 2 of \cite{rockova2023adaptive} for the risk curves under Dirac spike-and-slab and spike-and-slab Lasso priors.
These priors also feature global shrinkage parameters.
For Dirac spike-and-slab prior, it is the proportion of the Dirac spike.
For spike-and-slab Lasso, it is the rate of the steeper Laplace component.
The parameter $\tau$ in the Horseshoe prior plays a similar role.

\subsection{Proof of Spectroscopy Results}\label{sec:pf.lm.loss.spectroscopy}
Note that the predictive density can be expressed by
\[
\hat{p}_\pi(\tilde{y}\mid y) = \frac{\int \pi(\tilde y\mid \mu) \pi(\mu \mid y) \diff \mu}{\int \pi(\mu \mid y) \diff \mu} = \frac{\int \pi(\tilde{y}\mid \mu) \pi(y\mid \mu) \pi(\mu\mid \tau) \diff \mu}{\int \pi(y\mid \mu) \pi(\mu\mid \tau) \diff \mu}.
\]
For a fixed $\tau$, the denominator is exactly the marginal density of $y$ under the mixture prior, denoted by $m_\pi(y) = \int \pi(y\mid \mu) \pi(\mu\mid \tau) \diff \mu$.
Meanwhile, the numerator can be written as
\begin{align*}
    \int \pi(\tilde{y}\mid \mu) \pi(\mu\mid y) \pi(\mu\mid \tau) \diff \mu &= \int \pi(\tilde{y}\mid \mu) \pi(\mu\mid y) \int \pi(\mu \mid \lambda,\tau) \nu(\lambda)\diff \lambda \diff \mu\\
    &= \int \left\{\int \pi(\tilde y\mid \mu) \pi(y \mid \mu) \pi(\mu \mid \lambda, \tau) \diff \mu\right\} \nu(\lambda) \diff \lambda.
\end{align*}
Under a Gaussian prior with fixed $\lambda$, the marginal distribution of $y$ is $m_\lambda(y) = \int \pi(y\mid \mu) \pi(\mu \mid \lambda,\tau) \diff \mu$, and thus the predictive density is
\[
\hat p_\lambda (\tilde y\mid y) = \frac{\int \pi(\tilde y\mid \mu) \pi(y \mid \mu) \pi(\mu \mid \lambda,\tau) \diff \mu}{m_\lambda(y)}.
\]
Therefore
\[
\hat p_\pi (\tilde y\mid y) = \int \hat p_\lambda(\tilde y \mid y) \frac{m_\lambda(y)}{m_\pi(y)} \nu(\lambda) \diff \lambda.
\]
Note that $\pi(\mu\mid \tau) = \int \pi(\mu\mid \lambda,\tau) \nu(\lambda) \diff \lambda$, then we have $m_\pi(y) = \int m_\lambda(y)\nu(\lambda)\diff \lambda$.
Therefore
\[
\frac{m_\lambda(y)}{m_\pi(y)} \nu(\lambda) = \frac{m_\lambda (y) \nu(\lambda)}{\int m_\lambda(y) \nu(\lambda) \diff \lambda} = \pi(\lambda \mid y),
\]
which leads to
\[
\hat{p}_\pi(\tilde{y}\mid y) = \int \hat{p}_\lambda(\tilde{y}\mid y)\pi(\lambda\mid y)\diff\lambda.
\]
By Jensen's Inequality $\E \log X \leq \log \E X$, the predictive KL loss under the mixture prior can be written as
\begin{align*}
    L(\theta, \hat p_\pi(\cdot \mid y)) &= \int \pi(\tilde y\mid \theta) \log\frac{\pi(\tilde y\mid \theta)}{\int \hat p_\lambda(\tilde y \mid y) \pi(\lambda \mid y)\diff \lambda} \diff \tilde y\\
    &\leq \int \pi(\tilde y\mid \theta) \int \log\frac{\pi(\tilde y\mid \theta)}{\hat p_\lambda(\tilde y \mid y)} \pi(\lambda\mid y)\diff \lambda \diff \tilde y.
\end{align*}
Recall that for any fixed $\lambda$, the KL predictive loss is
\[
L(\theta, \hat p_\lambda(\cdot \mid y)) = \int \pi(\tilde y \mid \theta) \log \frac{\pi(\tilde y \mid \theta)}{\hat p_\lambda (\tilde y \mid y)} \diff \tilde y.
\]
By Fubini's Theorem, we arrive at \eqref{eq:loss.spectroscopy}.

\subsection{Proof of Lemma \ref{lm:risk.spectroscopy}}\label{sec:pf.lm.risk.spectroscopy}
Recall that for a fixed $\lambda$,
\[
L(\theta, \hat p_\lambda(\cdot \mid y)) \leq \frac{1-v}{2} \frac{1+\lambda^2\tau^2}{v+\lambda^2\tau^2} \left(\frac{\lambda^2\tau^2}{1+\lambda^2\tau^2} y - \theta\right).
\]
By \eqref{eq:loss.spectroscopy}, under the posterior of $\lambda$ in \eqref{eq:lambda.posterior},
\begin{align*}
    L(\theta, \hat p(\cdot \mid y)) &\leq \frac{1-v}{2\tau}\frac{1+\lambda^2\tau^2}{v+\lambda^2\tau^2}\frac{\int_0^\infty (\frac{\lambda^2\tau^2}{1+\lambda^2\tau^2} y -\theta)^2 \frac{1}{\sqrt{1+\lambda^2\tau^2}} \frac{1}{1+\lambda^2} e^{-\frac{y^2}{2(1+\lambda^2\tau^2)}} \diff \lambda}{\Phi_1(1,1,\frac{3}{2}, -\frac{y^2}{2}, 1-\tau^2)}.
\end{align*}
With a change of variable $u = \frac{\lambda^2\tau^2}{1+\lambda^2\tau^2}$, the right hand side becomes
\[
\frac{1-v}{4}\frac{1}{u+v(1-u)} e^{-\frac{y^2}{2}}\frac{\int_0^1 (u y -\theta)^2 u^{-1/2}\frac{1}{\tau^2+(1-\tau^2)u} e^{\frac{y^2}{2}u} \diff u}{\Phi_1(1,1,\frac{3}{2}, -\frac{y^2}{2}, 1-\tau^2)}.
\]
The integral on the numerator can be split into three using $(uy-\theta)^2 = u^2(y-\theta)^2 - 2u(1-u)\theta(y-\theta) + (1-u)^2\theta^2$.
For the first two terms, use the bound $\frac{1}{u+v(1-u)} \leq \frac{1}{u}$.
For the third term, use $\frac{1}{u+v(1-u)} \leq \frac{1}{v(1-u)}$.
Thus $L(\theta, \hat p(\cdot \mid y))$ is upper bounded by
\begin{align*}
    \frac{(1-v)e^{-\frac{y^2}{2}}}{4\Phi_1(1,1,\frac{3}{2}, -\frac{y^2}{2}, 1-\tau^2)} \bigg[&(y-\theta)^2 \int_0^1 u^{1/2} \frac{1}{\tau^2+(1-\tau^2)u} e^{\frac{y^2}{2}u} \diff u\\
    &-2\theta(y-\theta) \int_0^1 u^{-1/2}(1-u) \frac{1}{\tau^2+(1-\tau^2)u} e^{\frac{y^2}{2}u} \diff u\\
    &+ \frac{\theta^2}{v} \int_0^1 u^{-1/2}(1-u) \frac{1}{\tau^2+(1-\tau^2)u} e^{\frac{y^2}{2}u} \diff u\bigg].
\end{align*}
Converting the integrals into confluent hypergeometric functions of two variables \citep{gordy1998generalization}, this can be equivalently written as
\begin{align*}
    \frac{1-v}{\Phi_1\left(1,1,\frac{3}{2}, -\frac{y^2}{2}, 1-\tau^2\right)} \bigg[&\frac{1}{6}(y-\theta)^2 \Phi_1\left(1,1,\frac{5}{2}, -\frac{y^2}{2}, 1-\tau^2\right)\\
    &+ \frac{1}{3} \left[\frac{\theta^2}{v} - 2\theta(y-\theta)\right] \Phi_1\left(2,1,\frac{5}{2}, -\frac{y^2}{2}, 1-\tau^2\right)\bigg].
\end{align*}
Note that the predictive risk is the expectation of the loss over $Y\sim N(\theta,1)$, i.e. $\rho(\theta, \hat p) = \E_\theta L(\theta, \hat p(\cdot \mid Y))$, we arrive at
\begin{align*}
    \rho(\theta, \hat p) &\leq \frac{1-v}{6}\E_\theta\left[\frac{\Phi_1(1,1,\frac{5}{2}, -\frac{y^2}{2}, 1-\tau^2)}{\Phi_1(1,1,\frac{3}{2}, -\frac{y^2}{2}, 1-\tau^2)}(Y-\theta)^2\right]\\
    &+ \frac{1-v}{3}\E_\theta\left[ \frac{\Phi_1(2,1,\frac{5}{2}, -\frac{y^2}{2}, 1-\tau^2)}{\Phi_1(1,1,\frac{3}{2}, -\frac{y^2}{2}, 1-\tau^2)} \left(\frac{\theta^2}{v} - 2\theta(Y-\theta)\right)\right].
\end{align*}
This finishes the proof of Lemma \ref{lm:risk.spectroscopy}.

\section{Proof of Theorem \ref{thm:minimax.fixed.tau}}\label{sec:apdx.pf.minimax.fixed.tau}

\subsection{Signal Case: $\theta \neq 0$}\label{sec:known.tau.signal}
For simplicity, denote $\zeta_\tau = \sqrt{2v\log(1/\tau)}$ in this section.

\subsubsection{Large observation: $|z+\theta| > \sqrt{2v\log(1/\tau)}$}\label{sec:known.tau.signal.large.obs}

Using expression \eqref{eq:N.HS.form.B} of $N^{HS}$, $\theta^2/2r$ is cancelled out, thus
\[
g(z,\theta,v) = \log(\sqrt{v}) + \left(1-\frac{1}{\sqrt{v}}\right)\theta z + \log \frac{\int_0^1 (1-u)^{-1/2}\frac{1}{1-(1-\tau^2) u}e^{-\frac{(z+\theta)^2}{2}u}\diff u}{\int_0^1 (1-u)^{-1/2}\frac{1}{1-(1-\frac{\tau^2}{v}) u}e^{-\frac{(\sqrt{v}z+\theta)^2}{2v}u}\diff u}.
\]
Since $0<v<1$, the first term is negative.
For the second term, use
\begin{multline}\label{eq:cross.term.exp.neg}
\E [\theta Z \one{|Z+\theta| > \zeta_\tau}] = \int_{-\infty}^{-\theta-\zeta_\tau} \theta z \phi(z) + \int_{-\theta+\zeta_\tau}^\infty \theta z \phi(z)\\
= \theta[\phi(-\theta+\zeta_\tau) - \phi(-\theta-\zeta_\tau)] > 0.
\end{multline}
Therefore $\E [(1-1/\sqrt{v})\theta Z\one{|Z+\theta| > \zeta_\tau}] < 0$.

For the third term, define $g_v(u) = (1-u)^{-1/2}\frac{1}{1-(1-\frac{\tau^2}{v}) u}$.
Similar to the technique as in \cite{van2014horseshoe}, by the mean value theorem, $g_v(u) = g_v(0) + ug'_v(\Bar{u})$ for some $\Bar{u}\in[0,u]$.
Then for any $s\in[0,1]$, we can split the integral
\begin{multline*}
\int_0^1 g_v(u) e^{-\frac{(\sqrt{v}z+\theta)^2}{2v}u} \diff u \\
= g_v(0)\int_0^s e^{-\frac{(\sqrt{v}z+\theta)^2}{2v}u}\diff u + g'_v(\Bar{u})\int_0^s u e^{-\frac{(\sqrt{v}z+\theta)^2}{2v}u} \diff u + \int_s^1 g_v(u) e^{-\frac{(\sqrt{v}z+\theta)^2}{2v}u} \diff u.
\end{multline*}
Note that $g_v(0)=1$. The first term is equal to
\[
\frac{v}{(\sqrt{v}z+\theta)^2}\left(2-2e^{-\frac{(\sqrt{v}z+\theta)^2}{2v}s}\right) := \frac{v}{(\sqrt{v}z+\theta)^2} h_1^v(z,s,\theta).
\]
The second term is equal to
\[
\frac{v^2 g'_v(\Bar{u})}{(\sqrt{v}z+\theta)^4}\left[4-4\left(1+\frac{(\sqrt{v}z+\theta)^2}{2v}s\right)e^{-\frac{(\sqrt{v}z+\theta)^2}{2v}s}\right]:= \frac{v^2 g'_v(\Bar{u})}{(\sqrt{v}z+\theta)^4} h_2^v(z,s,\theta).
\]
The third term is trivially lower bounded by 0, and upper bounded by
\begin{multline*}
e^{-\frac{(\sqrt{v}z+\theta)^2}{2v}s} \int_s^1 g_v(u)\diff u \\
\leq 2ve^{-\frac{(\sqrt{v}z+\theta)^2}{2v}s} \tau^{-1}(v-\tau^2)^{-1/2}\tan^{-1}\left(\frac{\sqrt{v-\tau^2}}{\tau}\right) := 2ve^{-\frac{(\sqrt{v}z+\theta)^2}{2v}s} h_3^v(\tau).
\end{multline*}
Therefore, we have the following upper bound
\begin{multline*}
\int_0^1 (1-u)^{-1/2}\frac{1}{1-(1-\tau^2) u}e^{-\frac{(z+\theta)^2}{2}u}\diff u\\
\leq \frac{1}{(z+\theta)^2} h_1^1(z,s,\theta) + \frac{g'_1(\Bar{u})}{(z+\theta)^4}h_2^1(z,s,\theta) + 2e^{-\frac{(z+\theta)^2}{2}s} h_3^1(\tau).
\end{multline*}
Note that the third term is the only one related to $\tau$.
Similarly, we have the lower bound
\begin{multline*}
\int_0^1 (1-u)^{-1/2}\frac{1}{1-(1-\frac{\tau^2}{v}) u}e^{-\frac{(\sqrt{v}z+\theta)^2}{2v}u}\diff u\\
\geq \frac{1}{(z+\theta/\sqrt{v})^2} h_1^v(z,s,\theta) + \frac{g'_v(\Bar{u})}{(z+\theta/\sqrt{v})^4} h_2^v(z,s,\theta).
\end{multline*}
We should have
\begin{multline}\label{eq:risk.frac.ub}
\frac{\int_0^1 (1-u)^{-1/2}\frac{1}{1-(1-\tau^2) u}e^{-\frac{(z+\theta)^2}{2}u}\diff u}{\int_0^1 (1-u)^{-1/2}\frac{1}{1-(1-\frac{\tau^2}{v}) u}e^{-\frac{(\sqrt{v}z+\theta)^2}{2v}u}\diff u}\\
\leq \frac{(z + \frac{\theta}{\sqrt{v}})^2}{(z+\theta)^2} \frac{h_1^1(z,s,\theta) + \frac{g'_1(\Bar{u}_1)}{(z+\theta)^2}h_2^1(z,s,\theta) + 2(z+\theta)^2 e^{-\frac{(z+\theta)^2}{2}s} h_3^1(\tau)}{h_1^v(z,s,\theta) + \frac{g'_v(\Bar{u}_v)}{(z+\theta/\sqrt{v})^2} h_2^v(z,s,\theta)}.
\end{multline}
Note that $\left|\frac{z+\theta/\sqrt{v}}{z+\theta}\right| = 1+\frac{(v^{-1/2}-1)|\theta|}{|z+\theta|} \leq 1+(v^{-1/2}-1)|\theta|\zeta_\tau^{-1} \rightarrow 1$ for any fixed value of $\theta$ as $\tau\downarrow 0$.
Also note that the first two terms in the numerator and the denominator do not contain $\tau$, and $h_1^v(z,s,\theta) \leq 2$, $h_2^v(z,s,\theta) \leq 4$.
Since $|z+\theta| > \zeta_\tau$, we also have
\[
h_1^v(z,s,\theta) > 2\left[1-\exp\left(-[1+(v^{-1/2}-1)\theta\zeta_\tau^{-1}]^2\frac{\zeta_\tau^2}{2}s\right)\right] > 2\left(1-e^{-\frac{\zeta_\tau^2}{2}s}\right),
\]
and similarly
\[
h_2^v(z,s,\theta) > 4\left(1-e^{-\frac{\zeta_\tau^2}{2}s}\right)
\]
For simplicity, write $y = z+\theta$. The right hand side of \eqref{eq:risk.frac.ub} is upper bounded by
\begin{equation}\label{eq:risk.frac.ub.bound}
\left(1+ \frac{(v^{-1/2}-1)|\theta|}{|z+\theta|}\right)^2 \frac{2+\frac{4g_1'(\Bar{u}_1)}{\zeta_\tau^2} + 2\zeta_\tau^2 e^{-\frac{\zeta_\tau^2}{2}s} h_3^1(\tau)}{2(1-e^{-\frac{\zeta_\tau^2}{2}s}) + \frac{4g'_v(\Bar{u}_v)}{\zeta_\tau^2}(1-e^{-\frac{\zeta_\tau^2}{2}s})}
\end{equation}
Note that the third term in the numerator dominates as long as
\[
\zeta_\tau^2 e^{-\frac{\zeta_\tau^2}{2}s} h_3^1(\tau) = 2v\tau^{vs-1}(1-\tau^2)^{-1/2}\log\frac{1}{\tau}\tan^{-1}\left(\frac{\sqrt{v-\tau^2}}{\tau}\right) \gtrsim 1.
\]
If we calibrate $\tau \asymp (s_n/n)^\alpha$ for some constant $\alpha$, up to a logarithm factor, we have $\tau\downarrow 0$ as $n\rightarrow\infty$, and this condition is satisfied.
Moreover, we take $s$ close to one.
For any $|y| > \zeta_\tau$,
\begin{multline*}
\log \frac{\int_0^1 (1-u)^{-1/2}\frac{1}{1-(1-\tau^2) u}e^{-\frac{(z+\theta)^2}{2}u}\diff u}{\int_0^1 (1-u)^{-1/2}\frac{1}{1-(1-\frac{\tau^2}{v}) u}e^{-\frac{(\sqrt{v}z+\theta)^2}{2v}u}\diff u}\\
\leq (1-v)\log\frac{1}{\tau} + \log(\pi v(1-\tau^2)^{-1/2}\log(1/\tau)) + 2(v^{-1/2}-1)\frac{|\theta|}{|z+\theta|}.
\end{multline*}
As $\tau\downarrow 0$, the second term is of order $\log \log (1/\tau)$, which is dominated by the first term.
We shall see that the third term is upper bounded by a constant after taking expectation over $Z$ and supremum over $\theta\in \R$.
When $|\theta| \leq \zeta_\tau$, $\frac{|\theta|}{|z+\theta|} \leq 1$.
When $|\theta| > \zeta_\tau$, we further split the value of $|Z|$ by $|\theta|/2$.
By Chernoff bound,
\[
\P(|Z| > \epsilon |\theta|) \leq 2e^{-\frac{1}{2}\epsilon^2\theta^2}.
\]
Meanwhile, as $|Z/\theta| < 1/2$, we have
\[
\frac{|\theta|}{|Z+\theta|} = \frac{1}{|1+Z/\theta|} \leq \sum_{k=0}^\infty \left|\frac{Z}{\theta}\right|^k \leq 2.
\]
We conclude that
\begin{equation}\label{eq:theta.Z.ratio.bound}
\sup_{\theta\in \R}\E\left[\frac{|\theta|}{|Z+\theta|}\one{|Z+\theta|>\zeta_\tau}\right] \leq 1 + 2 + \sup_{|\theta| \geq \zeta_\tau}\frac{|\theta|}{\zeta_\tau}2e^{-\frac{1}{8}\theta^2}.
\end{equation}
As $\zeta_\tau > 2$, which is the case for all $\tau < e^{-2/v}$, the third term in the upper bound can be further upper bounded by $2\tau^{v/4}$.
The bound then becomes $3+2\tau^{v/4}$.
We conclude that
\begin{multline}\label{eq:Eg.large.obs}
\sup_{\theta\in\R}\E\left[g(Z,\theta,v)\one{|Z+\theta| > \sqrt{2v\log(1/\tau)}}\right]\\
\leq (1-v)\log\frac{1}{\tau} + \log(\pi v(1-\tau^2)^{-1/2}\log(1/\tau)) + (v^{-1/2}-1)(6+4\tau^{v/4}).
\end{multline}
As $\tau\downarrow 0$, the first term dominates.

\subsubsection{Small observation: $|z+\theta| \leq \sqrt{2v\log(1/\tau)}$}\label{sec:known.tau.signal.small.obs}

In the case when observed value satisfies $|y| \leq \zeta_\tau$, we use the trivial bound $\log N^{HS}_{\theta,v}(z) \leq 0$. Thus
\[
g(z,\theta,v) \leq \frac{\theta^2}{2r} + \log N^{HS}_{\theta,1}(z).
\]
The first term can be bounded by using the inequality $(a+b)^2 \leq 2a^2 + 2b^2$ and $v/r = 1-v$:
\[
\frac{\theta^2}{2r} \leq \frac{Z^2 + (Z+\theta)^2}{r} \leq \frac{Z^2}{r} + 2(1-v)\log\frac{1}{\tau}.
\]
We shall bound the second term using the expression \eqref{eq:N.HS.form.A}.
The proof is split into two cases according to the value of the variance ratio $r$.
For both cases, we use expression \eqref{eq:N.HS.form.A},
\[
N_{\theta,1}^{HS}(z) = \frac{\tau}{\pi}\int_0^1 u^{-1/2}\frac{1}{\tau^2 + (1-\tau^2)u} \exp\left[\frac{(z+\theta)^2}{2}u\right] \diff u.
\]

\noindent{\bf The case when $0<r\leq 1$}

When $0<r\leq 1$, we have $0<v\leq 1/2$, and $v\leq 1-v$.
For any $0<\tau<1$,
\begin{align*}
N_{\theta,1}^{HS}(z) &\leq \frac{\tau}{\pi}e^{\frac{(z+\theta)^2}{2}} \int_0^1 u^{-1/2}\frac{1}{\tau^2 + (1-\tau^2)u} \diff u\\
& = \frac{2}{\pi}e^{\frac{(z+\theta)^2}{2}} (1-\tau^2)^{-1/2} \tan^{-1}\left(\frac{\sqrt{1-\tau^2}}{\tau}\right)\\
&\leq e^{\frac{(z+\theta)^2}{2}}.
\end{align*}
Since $|z+\theta| \leq \zeta_\tau$, we have $e^{\frac{(z+\theta)^2}{2}} \leq \tau^{-v} \leq \tau^{v-1}$.
We conclude that
\[
\log N_{\theta,1}^{HS}(z) \leq (1-v)\log\frac{1}{\tau} \quad \text{when }r\in(0,1].
\]

\noindent{\bf The case when $r>1$}

When $r>1$, we have $v>1/2$.
It is therefore possible to split the integral by $\frac{1-v}{v}<1$.
\begin{multline*}
    N_{\theta,1}^{HS}(z) = \frac{\tau}{\pi} \int_0^{\frac{1-v}{v}} u^{-1/2}\frac{1}{\tau^2 + (1-\tau^2)u} e^{\frac{(z+\theta)^2}{2}u} \diff u\\
    + \frac{\tau}{\pi} \int_{\frac{1-v}{v}}^1 u^{-1/2}\frac{1}{\tau^2 + (1-\tau^2)u} e^{\frac{(z+\theta)^2}{2}u} \diff u.
\end{multline*}
For the first term, note that
\begin{multline*}
\frac{\tau}{\pi} \int_0^{\frac{1-v}{v}} u^{-1/2}\frac{1}{\tau^2 + (1-\tau^2)u} e^{\frac{(z+\theta)^2}{2}u} \diff u\\
\leq e^{\frac{(z+\theta)^2}{2}\frac{1-v}{v}} \frac{\tau}{\pi} \int_0^{\frac{1-v}{v}} u^{-1/2}\frac{1}{\tau^2 + (1-\tau^2)u} \diff u \leq \frac{1}{\tau^{1-v}}.
\end{multline*}
For the second term, use $\tau^2 + (1-\tau^2)u \geq u$.
Therefore
\begin{multline*}
    \frac{\tau}{\pi} \int_{\frac{1-v}{v}}^1 u^{-1/2}\frac{1}{\tau^2 + (1-\tau^2)u} e^{\frac{(z+\theta)^2}{2}u} \diff u\\
    \leq \frac{\tau}{\pi} \int_{\frac{1-v}{v}}^1 u^{-3/2} e^{\frac{(z+\theta)^2}{2}u} \diff u \leq \frac{2\tau}{\pi} e^{\frac{(z+\theta)^2}{2}} \left(\sqrt{\frac{v}{1-v}}-1\right),
\end{multline*}
which is upper bounded by $\frac{2}{\pi} \left(\sqrt{\frac{v}{1-v}}-1\right) \tau^{1-v}$,
a quantity that converges to zero as $\tau\downarrow 0$.
Therefore, using $\log(1+x) \leq x$ for $x\geq 0$, we arrive at
\[
\log N_{\theta,1}^{HS}(z) \leq (1-v)\log\frac{1}{\tau} + \frac{2}{\pi}\left(\sqrt{\frac{v}{1-v}} -1\right) \tau^{2(1-v)} \quad \text{when }r\in(1,\infty).
\]

\noindent{\bf Putting both cases together}

For the case when the observation is small in scale, we conclude that
\begin{multline}\label{eq:Eg.small.obs}
    \sup_{\theta\in\R}\E\left[g(Z,\theta,v)\one{|Z+\theta| \leq \sqrt{2v\log(1/\tau)}}\right]\\
    \leq (1-v)\log\frac{1}{\tau} + \frac{2}{\pi}\left(\sqrt{\frac{v}{1-v}} -1\right)_+ \tau^{2(1-v)}.
\end{multline}
In a slight abuse of notation, $\left(\sqrt{\frac{v}{1-v}} -1\right)_+$ takes value $\sqrt{\frac{v}{1-v}} -1$ as $r>1$ (or $v>1/2$), and zero otherwise.
Combining the bounds \eqref{eq:Eg.large.obs} and \eqref{eq:Eg.small.obs}, we arrive at the following upper bound on the univariate predictive risk:
\begin{equation}\label{eq:risk.bound.signal}
    \sup_{\theta\in\R} \rho(\theta,\hat{p}) \leq (1-v)\log\frac{1}{\tau} + \tilde C_1(v,\tau)
\end{equation}
for any $\tau\in(0,1)$, where
\begin{multline}\label{eq:C.1.def}
    \tilde C_1(v,\tau) = \log((1-\tau^2)^{-1/2}\log(1/\tau)) + \log(\pi v) \\
    +(v^{-1/2}-1)(6+4\tau^{v/4}) + \frac{2}{\pi}\left(\sqrt{\frac{v}{1-v}} -1\right)_+ \tau^{2(1-v)}.
\end{multline}
Note that as $\tau\downarrow 0$, $C_1(v,\tau)$ becomes
\[
\log\log\frac{1}{\tau} + \log(\pi v) + 3(v^{-1/2}-1).
\]
Despite its divergence around $\tau=0$, $C_1(r,\tau)$ is dominated by the $(1-v)\log(1/\tau)$ term.

\subsection{Noise Case: $\theta = 0$}\label{sec:known.tau.noise}

By Lemma \ref{lm:risk.decomp}, when $\theta = 0$, the KL risk $\rho(0, \hat p)$ can be written as the expectation of $g(Z,0,v)$.
Using the form \eqref{eq:N.HS.form.A},
\[
g(z,0,v) = \frac{1}{2}\log(v) + \log \frac{\int_0^1 u^{-1/2} [\tau^2 + (1-\tau^2)u]^{-1} e^{\frac{z^2}{2}u} \diff u}{\int_0^1 u^{-1/2} \left[\frac{\tau^2}{v} + \left(1-\frac{\tau^2}{v}\right)u\right]^{-1} e^{\frac{z^2}{2}u} \diff u}.
\]
Despite not being directly used for the proof, the following analysis provides an intuitive view on the different behavior of $g(z,0,v)$ under different levels of observation $y=z$.
When $|z|$ is small, the integrals on both the numerator and the denominator will be dominated by small $u$, making $g(z,0,v) \approx \frac{1}{2}\log (v) + \log \frac{1}{v} > 0$.
Meanwhile, when $|z|$ is large enough, the integrals will be dominated by large $u$, leading to an almost identical numerator and denominator, thus a negative $g(z,0,v)$ thanks to the $\log(v)$ term.
This finding is heuristic to our proof strategy; we should expect the upper bound to be very small when $|z|$ is large (rate minimax with a very small constant, or even exceeding the minimax rate).
We shall see that it is indeed the case in the following analysis in Section \ref{sec:known.tau.noise.large.obs}.
In the following proof, we split the observation $|z|$ by $\sqrt{2\log(1/\tau)}$.

\subsubsection{Large observation: $|z| > \sqrt{2\log(1/\tau)}$}\label{sec:known.tau.noise.large.obs}
Recall Section \ref{sec:known.tau.signal.large.obs} that $g(z,0,v)$ can be written as follows using the form \eqref{eq:N.HS.form.B} and $g_v(u) = (1-u)^{-1/2} \frac{1}{1-(1-\frac{\tau^2}{v})u}$,
\[
g(z,0,v) = \log(\sqrt{v}) + \log \frac{\int_0^1 g_1(u) e^{-\frac{z^2}{2}u}\diff u}{\int_0^1 g_v(u) e^{-\frac{z^2}{2}u}\diff u}.
\]
Using the same manipulation that led to \eqref{eq:risk.frac.ub}, we have for any $s\in(0,1)$, there exists $\bar u_1 \in (0,s)$ and $\bar u_v \in (0,s)$ that
\[
\frac{\int_0^1 g_1(u) e^{-\frac{z^2}{2}u}\diff u}{\int_0^1 g_v(u) e^{-\frac{z^2}{2}u}\diff u} \leq \frac{h_1^1(z,s,0) + \frac{g_1'(\Bar{u}_1)}{z^2} h_2^1(z,s,0) + 2z^2 e^{-\frac{z^2}{2}s} h_3^1(\tau)}{h_1^v(z,s,0) + \frac{g_v'(\Bar{u}_v)}{z^2} h_2^v(z,s,0)}
\]
As $z = \sqrt{2\log(1/\tau)}$, the third term in the numerator dominates when $\tau\downarrow 0$ as long as
\[
\lim_{\tau\downarrow 0} z^2 e^{-\frac{z^2}{2}s} h_3^1(\tau) = \frac{\pi}{2\sqrt{v}} \tau^{s-1}\log\frac{1}{\tau} \gtrsim 1.
\]
This condition is satisfied with any choice of $s\in(0,1)$.
Thus
\[
g(z,0,v) \leq \log(\sqrt{v}) + (1-s) \log\frac{1}{\tau} + \log\left(\frac{\pi}{2\sqrt{v}} \log\frac{1}{\tau}\right).
\]
For Mills ratio $R(x) = \frac{1-\Phi(x)}{\phi(x)}$, we have the upper bound $R(x) < 1/x$ \citep{wainwright2019high}, thus
\begin{align*}
    &\E g(Z,0,v)\one{|Z|>\sqrt{2\log(1/\tau)}}\\
    &\leq \left[ \log(\sqrt{v}) + (1-s)\log\frac{1}{\tau} + \log\left(\frac{\pi}{2\sqrt{v}} \log\frac{1}{\tau}\right)\right] 2\int_{\sqrt{2\log(1/\tau)}}^{\infty}\phi(z)\diff z\\
    &\leq \left[(1-s)\log\frac{1}{\tau} + \log\left(\frac{\pi}{2\sqrt{v}} \log\frac{1}{\tau}\right)\right] \frac{2\phi(\sqrt{2\log(1/\tau)})}{\sqrt{2\log(1/\tau)}}\\
    &\leq \frac{1-s}{\sqrt{\pi}}\tau\sqrt{\log\frac{1}{\tau}} + \frac{\tau \log\left(\frac{\pi}{2\sqrt{v}} \log\frac{1}{\tau}\right)}{\sqrt{\pi\log(1/\tau)}}.
\end{align*}
Both terms converge to zero as $\tau\rightarrow 0$, with the first term dominating.
Since the choice of $s$ is arbitrary, under our calibration of $\tau = \frac{s_n}{n}(\log\frac{n}{s_n})^\alpha$, any choice of $s \geq v$ achieves rate-minimaxity.
Specifically, if we choose an $s$ that is close to 1, the constant can be arbitrarily small.

\subsubsection{Small observation: $|z| \leq \sqrt{2\log(1/\tau)}$} \label{sec:known.tau.noise.small.obs}
The proof of this case relies on the upper bound \eqref{eq:risk.spectroscopy} given by the risk spectroscopy result in Lemma \ref{lm:risk.spectroscopy}, which can be rewritten as
\[
g(z,0,v) \leq \frac{1-v}{2} \frac{\int_0^1 u^{1/2} \frac{1}{\tau^2 + (1-\tau^2)u} e^{\frac{z^2}{2}} \diff u}{\int_0^1 u^{-1/2} \frac{1}{\tau^2 + (1-\tau^2)u} e^{\frac{z^2}{2}} \diff u} z^2.
\]
We shall upper bound the denominator and lower bound the numerator by splitting the integrals.
When $u\in (0,\tau^2)$, we use the upper and lower bounds $\tau^2 + (1-\tau^2) u \in (\tau^2, 2\tau^2)$,
whereas the bound for $u \in (\tau^2,1)$ is $\tau^2 + (1-\tau^2) u \in (u, 2u)$.

The integral on the denominator can be split by $(0,\tau^2)$ and $(\tau^2,1)$, such that
\begin{align*}
    \int_0^1 u^{-1/2} \frac{1}{\tau^2 + (1-\tau^2)u} e^{\frac{z^2}{2}} \diff u &\geq \frac{1}{2\tau^2}\int_0^{\tau^2} u^{-1/2}\diff u + \int_{\tau^2}^1 \frac{1}{2}u^{-3/2} e^{\frac{z^2}{2}u} \diff u\\
    &\geq \frac{1}{\tau} + \frac{1}{z^2}\left(e^{\frac{z^2}{2}} - e^{\frac{z^2}{2}\tau^2}\right).
\end{align*}
As $|z| \leq \sqrt{2\log(1/\tau)}$, the second term has a smaller order to $\tau^{-1}$.
Therefore the lower bound is dominated by $1/\tau$, and we arrive at $\int_0^1 u^{-1/2} \frac{1}{\tau^2 + (1-\tau^2)u} e^{\frac{z^2}{2}} \diff u \geq \tau^{-1}$. 

For the numerator, split the integral by $(0,\tau^2)$, $(\tau^2, 1/a)$, and $(1/a,1)$.
Here, $a$ is an arbitrarily chosen constant between $1$ and $1/\tau^2$.
Then
\begin{align*}
    &\int_0^1 u^{1/2} \frac{1}{\tau^2 + (1-\tau^2)u} e^{\frac{z^2}{2}} \diff u\\
    &\leq  \frac{1}{\tau^2} e^{\frac{z^2}{2}\tau^2}\int_0^{\tau^2} u^{1/2}\diff u + \int_{\tau^2}^{1/a} u^{-1/2} e^{\frac{z^2}{2}u}\diff u + \int_{1/a}^1 u^{-1/2} e^{\frac{z^2}{2}u}\diff u\\
    &\leq \frac{2}{3}\tau e^{\frac{z^2}{2}\tau^2} + 2e^{\frac{z^2}{2}\frac{1}{a}}\left(\frac{1}{\sqrt{a}} - \tau\right) + \frac{2\sqrt{a}}{z^2}\left(e^{\frac{z^2}{2}} - e^{\frac{z^2}{2}\frac{1}{a}}\right).
\end{align*}
Therefore we arrive at the upper bound
\begin{equation}\label{eq:g.UB.noise.small.obs}
    g(z,0,v) \leq \frac{1-v}{2} \tau \left[\frac{2}{3}\tau z^2 e^{\frac{z^2}{2}\tau^2} + 2 z^2 e^{\frac{z^2}{2}\frac{1}{a}}\left(\frac{1}{\sqrt{a}} - \tau\right) + 2\sqrt{a} \left(e^{\frac{z^2}{2}} - e^{\frac{z^2}{2}\frac{1}{a}}\right)\right].
\end{equation}
To get an upper bound for the risk, we need to take an expectation over $Z\sim N(0,1)$. For $|Z| < \sqrt{2\log(1/\tau)}$.
Using integration by part,
\begin{align*}
    &2\int_0^{\sqrt{2\log(1/\tau)}} z^2 e^{\frac{z^2}{2}\tau^2} \frac{1}{\sqrt{2\pi}}e^{-\frac{z^2}{2}}\diff z\\
    &= \frac{1}{\sqrt{2\pi}(1-\tau^2)} \left(\int_0^{\sqrt{2\log(1/\tau)}} e^{-\frac{z^2}{2}(1-\tau^2)} \diff z - \tau^{(1-\tau^2) \sqrt{2\log(1/\tau)}}\right)\\
    &\leq (1-\tau^2)^{-3/2}.
\end{align*}
The inequality is due to expanding the integral limit to $(0,\infty)$ and noticing the second term beings  negative.
Similarly,
\begin{align*}
    2\int_0^{\sqrt{2\log(1/\tau)}} z^2 e^{\frac{z^2}{2}\frac{1}{a}} \frac{1}{\sqrt{2\pi}}e^{-\frac{z^2}{2}}\diff z \leq (1-a^{-1})^{-3/2}.
\end{align*}
Meanwhile,
\begin{align*}
    2\int_0^{\sqrt{2\log(1/\tau)}} e^{\frac{z^2}{2}}\frac{1}{\sqrt{2\pi}}e^{-\frac{z^2}{2}}\diff z = \frac{2}{\sqrt{\pi}} \sqrt{\log(1/\tau)}.
\end{align*}
Combining these results, we arrive at
\begin{multline}\label{eq:Eg0.small.z}
        \E g(Z,0,v) \one{|Z|\leq \sqrt{2\log(1/\tau)}}\\
        \leq (1-v)\tau\sqrt{\log\frac{1}{\tau}}\left[\frac{2\sqrt{a}}{\sqrt{\pi}} + \frac{(1-a^{-1})^{-3/2}}{\sqrt{\log(1/\tau)}} + \frac{\tau(1-\tau^2)^{-3/2}}{3 \sqrt{\log(1/\tau)}}\right].
\end{multline}
When $\tau \rightarrow 0$, the last two terms in the bracket converge to zero, and the first term is a constant not related to $\tau$.
Therefore, the upper bound in \eqref{eq:Eg0.small.z} is of order $(1-v)\tau \sqrt{\log(1/\tau)}$.
With the choice of $a$ being close to 1, combining the two cases in Sections \ref{sec:known.tau.noise.large.obs} and \ref{sec:known.tau.noise.small.obs}, we conclude that
\begin{equation}\label{risk.bound.noise}
\rho(0,\hat p) \leq \frac{2}{\sqrt{\pi}} (1-v)\tau\sqrt{\log\frac{1}{\tau}} + \tilde C_2(v,\tau),
\end{equation}
where
\begin{equation}\label{eq:C.2.def}
    \tilde C_2(v,\tau) = \frac{1-v}{3}\tau^2 (1-\tau^2)^{-3/2} + \frac{\tau \log\left(\frac{\pi}{2\sqrt{v}} \log\frac{1}{\tau}\right)}{2\sqrt{\pi\log(1/\tau)}}.
\end{equation}
As $\tau\rightarrow 0$, $\tilde C_2(v,\tau)$ converges to zero at a rate faster than $\tau\sqrt{\log(1/\tau)}$.

\subsection{Conclusion}

Combining the results in Sections \ref{sec:known.tau.signal} and \ref{sec:known.tau.noise}, we conclude that
\[
\sup_{\bm \theta \in \Theta_n(s_n)}\rho_n(\bm \theta, \hat p_\pi) \leq s_n (1-v)\log\frac{1}{\tau} + \frac{2(n-s_n)}{\sqrt{\pi}} (1-v) \tau\sqrt{\log\frac{1}{\tau}} + \tilde C_n(v,\tau),
\]
where
\[
\tilde C_n(v,\tau) = s_n \tilde C_1(v,\tau) + (n-s_n)\tilde C_2(v,\tau).
\]
For the fixed-$\tau$ case, we calibrate $\tau$ such that it takes value $\tau_{n,\alpha} = \frac{s_n}{n}[\log(\frac{n}{s_n})]^\alpha$ where $\alpha\in[0,1/2]$.
Note that $1-v = \frac{1}{1+r}$, therefore
\[
\sup_{\bm \theta \in \Theta_n(s_n)}\rho_n(\bm \theta, \hat p_\pi) \leq \left(1 + \frac{2}{\sqrt{\pi}\log^{1/2 - \alpha}(n/s_n)}\right) \frac{s_n}{1+r} \log \frac{n}{s_n} + \tilde C_n(v, \tau_{n,\alpha}).
\]
By definition of $\tilde C_1(v,\tau)$ in \eqref{eq:C.1.def} and $\tilde C_2(v,\tau)$ in \eqref{eq:C.2.def}, we get
\begin{multline}\label{eq:C.n.def}
    \tilde C_n(v,\tau_{n,\alpha}) = s_n\log\log\frac{n}{s_n} + s_n\frac{\log(\pi/2\sqrt v) + \log \log\frac{n}{s_n}}{2\sqrt{\pi}\log^{1/2-\alpha}(n/s_n)} + \\
    s_n\left\{\log(\pi v) + 
    \left(\frac{1}{\sqrt{v}}-1\right)\left[3+2\left(\frac{s_n}{n}\right)^{v/4}\right] + \frac{2}{\pi}\left(\sqrt{\frac{v}{1-v}}-1\right)_+ \left(\frac{s_n}{n}\right)^{2(1-v)} + \frac{s_n}{n}\right\}.
\end{multline}
Since $s_n/n\rightarrow 0$, we have
\[
\tilde C_n(v,\tau_{n,\alpha}) \asymp s_n\log\log\frac{n}{s_n} \lesssim s_n \log \frac{n}{s_n}.
\]
If we further choose any $\alpha \in [0,1/2)$, we have
\[
\sup_{\bm \theta \in \Theta_n(s_n)}\rho_n(\bm \theta, \hat p_\pi) \leq \frac{s_n}{1+r} \log \frac{n}{s_n}(1+o(1)).
\]
This finishes the proof of Theorem \ref{thm:minimax.fixed.tau}.

\section{Proof of Lemmata in Section \ref{sec:minimax.adaptive}}

\subsection{Risk Decomposition and Spectroscopy for Hierarchical Model}\label{sec:hierarchical.decomp.spec}
The risk decomposition results under fixed $\tau$ given in Lemma \ref{lm:risk.decomp} does not directly apply to the hierarchical Bayes model.
Instead, it is generalized as follows. Let
\[
\tilde g(Y_i, \tilde Y_i, \theta_i, v) = \log \frac{\pi(\tilde Y_i \mid \theta_i)}{\hat p(\tilde Y_i\mid Y_i, \tau)}.
\]
Then we can write $L(\theta_i, \hat p (\cdot \mid Y_i, \tau)) = \E_{\tilde{\bm Y} \mid \bm \theta}[\tilde g(Y_i, \tilde Y_i, \theta_i, v) ]$.
Despite its slightly more complicated form, $\tilde g(Y_i, \tilde Y_i, \theta_i, v)$ can also be decomposed similarly as in Lemma \ref{lm:risk.decomp}, which is formalized by the following lemma.
\begin{lemma}\label{lm:g.tilde.decomp}
    Denote
    \begin{align}
        \begin{split}\label{eq:Nv.D.form.1}
            N_{v}(Y_i, \tilde Y_i) &= \frac{\tau}{\pi\sqrt{v}} \int_0^1 u^{-1/2}\frac{1}{\frac{\tau^2}{v} + (1-\frac{\tau^2}{v})u} \exp \left\{ \frac{v}{2} (Y_i + \tilde Y_i / r)^2 u \right\}\diff u,\\
            D(Y_i) &= \frac{\tau}{\pi} \int_0^1 u^{-1/2}\frac{1}{\tau^2 + (1-\tau^2)u} \exp \left\{ \frac{Y_i^2}{2} u \right\}\diff u,
        \end{split}
    \end{align}
    or equivalently,
    \begin{align}
        \begin{split}\label{eq:Nv.D.form.2}
            N_{v}(Y_i, \tilde Y_i) &= \frac{\tau e^{\frac{v}{2}(Y_i + \tilde Y_i / r)^2}}{\pi\sqrt{v}}  \int_0^1 (1-u)^{-1/2}\frac{1}{1- (1-\frac{\tau^2}{v})u} \exp \left\{ -\frac{v}{2} (Y_i + \tilde Y_i / r)^2 u \right\}\diff u,\\
            D(Y_i) &= \frac{\tau e^{\frac{Y_i^2}{2}}}{\pi} \int_0^1 (1-u)^{-1/2}\frac{1}{1- (1-\tau^2)u} \exp \left\{ -\frac{Y_i^2}{2} u \right\}\diff u.
        \end{split}
    \end{align}
    Then the following decomposition holds for $\tilde g(Y_i, \tilde Y_i, \theta_i, v)$,
    \[
    \tilde g(Y_i, \tilde Y_i, \theta_i, v) = \frac{\tilde Y_i \theta_i}{r} - \frac{\theta_i^2}{2r} - \log N_{v}(Y_i, \tilde Y_i) + \log D (Y_i).
    \]
\end{lemma}
\begin{proof}
    See Appendix \ref{sec:proof.lm.g.tilde.decomp}.
\end{proof}
Once the hyperprior $\pi(\tau)$ degenerates to a point mass, i.e. $\tau$ is fixed to a single value, Lemma \ref{lm:g.tilde.decomp} becomes identical to Lemma \ref{lm:risk.decomp} by taking expectation over $Y_i$ and $\tilde Y_i$.
Likewise, the risk spectroscopy discussed in Section \ref{sec:spectroscopy} also applies to the hierarchical Bayes model, serving as an alternative way to bound each summand in \eqref{eq:risk.decomp.adaptive}.
The following lemma serves as a generalization to Lemma \ref{lm:risk.spectroscopy}.
\begin{lemma}\label{lm:risk.spectroscopy.adaptive}
    Under the hierarchical Horseshoe prior, for each $1\leq i\leq n$,
    \begin{multline}
        \E_{\bm Y \mid \bm \theta} \E_{\tau \mid \bm Y} [L(\theta_i, \hat p(\cdot \mid Y_i, \tau))] \leq \frac{1-v}{6}\E_{\bm Y \mid \bm \theta} \E_{\tau \mid \bm Y}\left[\frac{\Phi_1(1,1,\frac{5}{2}, -\frac{Y_i^2}{2}, 1-\tau^2)}{\Phi_1(1,1,\frac{3}{2}, -\frac{Y_i^2}{2}, 1-\tau^2)}(Y_i-\theta_i)^2\right]\\
        + \frac{1-v}{3}\E_{\bm Y \mid \bm \theta} \E_{\tau \mid \bm Y}\left[ \frac{\Phi_1(2,1,\frac{5}{2}, -\frac{Y_i^2}{2}, 1-\tau^2)}{\Phi_1(1,1,\frac{3}{2}, -\frac{Y_i^2}{2}, 1-\tau^2)} \left(\frac{\theta_i^2}{v} - 2\theta_i(Y_i-\theta_i)\right)\right].
    \end{multline}
    Specifically, in the case that $\theta_i=0$, this reduces to
    \begin{equation}\label{eq:risk.spectroscopy.adaptive}
        \E_{\bm Y \mid \bm \theta} \E_{\tau \mid \bm Y} [L(0, \hat p(\cdot \mid Y_i, \tau))] \leq \frac{1-v}{6}\E_{\bm Y \mid \bm \theta} \E_{\tau \mid \bm Y}\left[\frac{\Phi_1(1,1,\frac{5}{2}, -\frac{Y_i^2}{2}, 1-\tau^2)}{\Phi_1(1,1,\frac{3}{2}, -\frac{Y_i^2}{2}, 1-\tau^2)}Y_i^2\right].
    \end{equation}
\end{lemma}
The rates are similar to those in Lemma \ref{lm:risk.spectroscopy}, but with an extra layer of expectation over $\tau$.

\subsection{A Useful Event under Theta-min Condition}

Recall that $Y_i = \theta_i + Z_i$ for all $1\leq i\leq n$, where $Z_i$ are i.i.d.~standard Gaussian.
Then the following concentration bound holds.
By Lemma 4 in \cite{castillo2015bayesian},
\[
\P\left(\max_{1\leq i\leq n}|Z_i| \leq 2\sqrt{\log n}\right) \geq 1 - \frac{2}{n}.
\]
When we assume the theta-min condition $\bm\theta\in\Theta_n(s_n,c)$, any nonzero $\theta_i$ satisfies $|\theta_i| > c\sqrt{2\log n}$.
When $\max_{1\leq i\leq n}|Z_i| \leq 2\sqrt{\log n}$, by triangle inequality, for all $i\in\mathcal S$, i.e. all $i$ such that $\theta_i \neq 0$, we have $|Y_i| \geq |\theta_i| - |Z_i| > (c-\sqrt{2})\sqrt{2\log n}$.
Denote the event
\[
\mathcal B_n = \{\min_{i\in\mathcal S} |Y_i| > (c-\sqrt{2})\sqrt{2\log n}\}.
\]
We have $\P(\mathcal B_n^c) \leq 2/n$.
This event will be used multiple times in the subsequent proof process.
To exclude the trivial case where $\mathcal B_n^c$ is empty, we require the choice of $c>\sqrt{2}$.

\subsection{Proof of Lemma \ref{lm:g.tilde.decomp}}\label{sec:proof.lm.g.tilde.decomp}
Under a fixed $\tau$, the predictive density is additive.
For each $1\leq i\leq n$, it can be written as
\[
\hat p(\tilde Y_i \mid Y_i, \tau) = \int \pi(\tilde Y_i \mid \mu_i) \pi(\mu_i \mid Y_i, \tau) \diff \mu_i = \frac{\int \pi(\tilde Y_i \mid \mu_i) \pi(Y_i\mid \mu_i) \pi(\mu_i \mid \tau) \diff \mu_i}{\int \pi(Y_i\mid \mu_i) \pi(\mu_i \mid \tau) \diff \mu_i}.
\]
Recall that $Y_i \mid \mu_i \sim N(\mu_i, 1)$, $\tilde Y_i \mid \mu_i \sim N(\mu_i, r)$.
We can therefore write $\tilde g(Y_i, \tilde Y_i, \theta_i, \tau)$ as
\[
\tilde g(Y_i, \tilde Y_i, \theta_i, \tau) = \frac{\tilde Y_i \theta_i}{r} - \frac{\theta_i^2}{2r} + \log \frac{\int \exp\{-\frac{\mu_i^2}{2v} + (Y_i + \tilde Y_i / r)\mu_i\} \pi(\mu_i\mid \tau) \diff \mu_i}{\int \exp\{-\frac{\mu_i^2}{2} + Y_i \mu_i\} \pi(\mu_i\mid \tau) \diff \mu_i}.
\]
The remaining proof is same as that of Lemma \ref{lm:risk.decomp}.
Note that for the Horseshoe prior, $\theta_i \mid \lambda_i, \tau \sim N(0, \lambda_i^2\tau^2)$, and $\lambda_i \sim C^+(0,1)$. 
Thus
\[
\pi(\mu_i\mid \tau) = \int \frac{1}{\sqrt{2\pi}\lambda_i \tau} \exp\left\{\frac{-\mu^2}{2\lambda_i^2\tau^2}\right\} \frac{2}{\pi}\frac{1}{1+\lambda_i^2} \diff \lambda_i.
\]
After applying Fubini's theorem, integrating out $\mu_i$, and changing variable $u = \frac{\lambda_i^2 \tau^2}{v+\lambda_i^2 \tau^2}$, we arrive at
\begin{multline*}
    \int \exp\left\{-\frac{\mu_i^2}{2v} + (Y_i + \tilde Y_i / r)\mu_i\right\} \pi(\mu_i\mid \tau) \diff \mu_i\\
    = \frac{\tau}{\pi\sqrt{v}} \int_0^1 u^{-1/2}\frac{1}{\frac{\tau^2}{v} + (1-\frac{\tau^2}{v})u} \exp \left\{ \frac{v}{2} (Y_i + \tilde Y_i / r)^2 u \right\}\diff u,
\end{multline*}
and
\[
    \int \exp\left\{-\frac{\mu_i^2}{2} + Y_i \mu_i\right\} \pi(\mu_i\mid \tau) \diff \mu_i
    = \frac{\tau}{\pi} \int_0^1 u^{-1/2}\frac{1}{\tau^2 + (1-\tau^2)u} \exp \left\{ \frac{Y_i^2}{2} u \right\}\diff u.
\]
This finishes the proof of Lemma \ref{lm:g.tilde.decomp} with the form \eqref{eq:Nv.D.form.1}.
The alternative form \eqref{eq:Nv.D.form.2} is attained by another change of variable from $u$ to $1-u$.

\subsection{Proof of Lemma \ref{lm:tau.not.overshooting}}\label{sec:pf.lm.tau.not.overshooting}

We aim to bound the posterior expectation $\sup_{\bm\theta\in\Theta_n(s_n,c)}\E_{\bm Y \mid \bm\theta} \E[\tau \mid \bm Y]$ where the prior is $\pi(\tau) = n e^{-n\tau} \one{\tau > 0}$.
Let $\mathcal S = \{i: \theta_i \neq 0\}$ be the active set of signals.
To bound $\sup_{\bm\theta \in \Theta_n(s_n,c)}\E_{\bm Y\mid\bm\theta}\E[\tau \mid \bm Y]$, we split the integral in the numerator by $K\tau_{n,0} = Ks_n/n$.
Specifically, we write
\[
\E[\tau \mid \bm Y] = \frac{\int_0^\infty \tau \pi(\bm Y \mid \tau)\pi(\tau)\diff\tau}{\int_0^\infty\pi(\bm Y\mid\tau) \pi(\tau)\diff\tau} = \frac{\mathcal N_1 + \mathcal N_2}{\mathcal D}.
\]
Here, $\mathcal N_1 = \int_0^{K\tau_{n,0}} \tau \pi(\bm Y \mid \tau)\pi(\tau)\diff\tau$, and $\mathcal N_2 = \int_{K\tau_{n,0}}^\infty \tau \pi(\bm Y \mid \tau)\pi(\tau)\diff\tau$.
We can write $\pi(\bm Y \mid \tau) = \prod_{i\in\mathcal S}\pi(Y_i \mid \tau) \prod_{i\notin \mathcal S}\pi(Y_i \mid \tau)$.
The denominator $\mathcal D$ is lower bounded by $\int_{s_n/n}^{2s_n/n} \pi(\bm Y\mid\tau) \pi(\tau)\diff\tau$.
When $i\in\mathcal S$ and $\tau \in (s_n/n, 2s_n/n)$, by Lemma \ref{lm:HS.pdf.UB.LB},
\begin{align*}
    \pi(Y_i \mid \tau) &\geq \int_{Y_i - |Y_i|/2}^{Y_i + |Y_i|/2} \frac{1}{\sqrt{2\pi}} e^{-\frac{(Y_i - \mu_i)^2}{2}} \frac{1}{2\sqrt{2\pi^3}\tau} \log\left(1 + \frac{4\tau^2}{\mu_i^2}\right) \diff \mu_i\\
    &\geq \frac{1}{2\sqrt{2\pi^3}\tau} \frac{16\tau^2}{9Y_i^2 + 16\tau^2} \left[1-2\Phi\left(-\frac{|Y_i|}{2}\right)\right]\\
    &\geq \frac{8}{\sqrt{2\pi^3}} \frac{Y_i^2[1-2\Phi(-|Y_i|/2)]}{9Y_i^2 + 64s_n^2/n^2} \cdot \frac{\tau}{Y_i^2}\\
    & := C_{1,n}(Y_i) \cdot \frac{\tau}{Y_i^2}.
\end{align*}
Thus the product has a lower bound
\begin{equation}\label{eq:D.i.in.S}
    \prod_{i\in\mathcal S} \pi(Y_i\mid\tau) \geq \tau^{s_n} \prod_{i\in\mathcal S} \frac{C_{1,n}(Y_i)}{Y_i^2}.
\end{equation}

For $i\notin \mathcal S$, on the contrary, we can write
\[
\pi(Y_i \mid \tau) = \phi(Y_i) \int_0^\infty (e^{Y_i\mu_i} + e^{-Y_i\mu_i}) e^{-\frac{\mu_i^2}{2}}\pi(\mu_i\mid\tau) \diff\mu_i.
\]
Using $\cosh(Y_i\mu_i) \geq 1$ and Fubini's Theorem, we have the lower bound
\begin{align*}
    \pi(Y_i \mid \tau) &\geq \phi(Y_i) \int_{-\infty}^\infty e^{-\frac{1}{2}\mu_i^2} \pi(\mu_i\mid\tau) \diff\mu_i\\
    &= \phi(Y_i) \int_0^\infty \frac{2}{\pi}\frac{1}{1+\lambda_i^2} \int_{-\infty}^\infty \frac{1}{\sqrt{2\pi}\lambda_i\tau} \exp\left(-\frac{\mu_i^2}{2}\frac{1+\lambda_i^2\tau^2}{\lambda_i^2\tau^2}\right)  \diff\mu_i \diff\lambda_i\\
    &= \phi(Y_i) \int_0^\infty \frac{2}{\pi}\frac{1}{1+\lambda_i^2} \frac{1}{\sqrt{1+\lambda_i^2\tau^2}} \diff\lambda_i.
\end{align*}
Using $1/\sqrt{1+x} \geq (1-x/2)_+$ for $x>0$, where $(\cdot)_+$ denotes the maximum of the value and zero, as well as $\arctan x \geq \pi/2-1/x$ for $x\geq 1$, considering that we are currently discussing $\tau \in (s_n/n, 2s_n/n)$ which is asymptotically close to zero, we obtain
\begin{align*}
\pi(Y_i \mid \tau) &\geq \phi(Y_i) \left(\frac{2}{\pi}\arctan\frac{\sqrt{2}}{\tau}-\frac{\tau^2}{\pi}\int_0^{\sqrt{2}/\tau} \frac{\lambda_i^2}{1+\lambda_i^2}\diff\lambda_i\right)\\
&\geq \phi(Y_i)\left(1-\frac{2\sqrt{2}}{\pi}\tau\right)\\
&\geq \phi(Y_i) e^{-\tau}.
\end{align*}
Therefore we end up with
\begin{equation}\label{eq:D.i.notin.S}
    \prod_{i\notin \mathcal S} \pi(Y_i\mid\tau) \geq \prod_{i\notin S}\phi(Y_i) \cdot \exp\left[-(n-s_n)\tau\right].
\end{equation}
Combining Equations \eqref{eq:D.i.in.S} and \eqref{eq:D.i.notin.S}, we obtain a lower bound for $\mathcal D$:
\begin{align*}
    \mathcal D &\geq \prod_{i\in \mathcal S} \frac{C_{1,n}(Y_i)}{Y_i^2} \cdot \prod_{i\notin\mathcal S}\phi(Y_i) \cdot \int_{s_n/n}^{2s_n/n} \tau^{s_n} \exp\left[-(n-s_n)\tau\right] ne^{-n\tau} \diff\tau\\
    &\geq \prod_{i\in \mathcal S} \frac{C_{1,n}(Y_i)}{Y_i^2} \cdot \prod_{i\notin\mathcal S}\phi(Y_i) \cdot ne^{-4s_n} \int_{s_n/n}^{2s_n/n} \tau^{s_n}\diff\tau\\
    &\geq \prod_{i\in \mathcal S} \frac{C'_{1,n}(Y_i)}{Y_i^2} \cdot \prod_{i\notin\mathcal S}\phi(Y_i) \cdot\left(\frac{s_n}{n}\right)^{s_n},
\end{align*}
where $C'_{1,n}(Y_i) = 2e^{-4}\cdot C_{1,n}(Y_i)$.

Next, we shall consider the numerators.
The first numerator $\mathcal N_1$ is straight-forward, considering that
\[
\mathcal N_1 = \int_0^{K\tau_{n,0}} \tau\pi(\bm Y\mid \tau) \pi(\tau) \diff \tau \leq K\tau_{n,0} \mathcal D.
\]
It remains to provide an upper bound on the second numerator.
We have
\[
\frac{\mathcal N_2}{\mathcal D} \leq  \int_{K\tau_{n,0}}^\infty
\prod_{i\in\mathcal S} \frac{Y_i^2 \pi(Y_i \mid \tau)}{C'_{1,n}(Y_i)}\cdot
\prod_{i\notin\mathcal S} \frac{\pi(Y_i\mid\tau)}{\phi(Y_i)}\cdot
\left(\frac{n}{s_n}\right)^{s_n} \tau \pi(\tau) \diff \tau.
\]
Note that we are eventually interested in its expectation over the distribution of $\bm Y \mid \bm\theta$.
We may split that expectation as
\[
\E_{\bm Y\mid \bm\theta}\left[\frac{\mathcal N_2}{\mathcal D}\right]  = \E_{\bm Y\mid \bm\theta}\left[\frac{\mathcal N_2}{\mathcal D}\one{\mathcal B_n}\right] + \E_{\bm Y\mid \bm\theta}\left[\frac{\mathcal N_2}{\mathcal D}\one{\mathcal B_n^c}\right],
\]
where the event $\mathcal B_n = \{\min_{i\in\mathcal S}|Y_i| > (c-\sqrt{2})\sqrt{2\log n}\}$.
By Lemma 4 of \cite{castillo2015bayesian} and Assumption \ref{assumption:beta.min}, we know that this event occurs at probability no less than $1-2/n$.

We start from the first truncated expectation.
By applying Fubini's Theorem, we get
\begin{multline*}
\E_{\bm Y\mid \bm\theta}\left[\frac{\mathcal N_2}{\mathcal D}\one{\mathcal B_n}\right] \leq \\
\int_{K\tau_{n,0}}^\infty
\E_{\bm Y\mid\bm \theta}\left[\prod_{i\in\mathcal S} \frac{Y_i^2 \pi(Y_i \mid \tau)}{C'_{1,n}(Y_i)}\one{\mathcal B_n}\right]\cdot
\E_{\bm Y\mid\bm \theta}\left[ \prod_{i\notin\mathcal S} \frac{\pi(Y_i\mid\tau)}{\phi(Y_i)}\right] \cdot
\left(\frac{n}{s_n}\right)^{s_n} \tau \pi(\tau) \diff \tau.
\end{multline*}
First, note that as $i\notin \mathcal S$, we have $Y_i \sim N(0,1)$. Then the expectation
\[
\E_{Y_i\mid\theta_i}\left[ \frac{\pi(Y_i\mid\tau)}{\phi(Y_i)}\right] = \int \frac{\pi(y_i\mid\tau)}{\phi(y_i)} \phi(y_i) \diff y_i = 1.
\]
It remains the product of the other expectation to bound.
For $i\in\mathcal S$ and $\tau \geq K\tau_{n,0}$, we can split the integral by
\[
\pi(Y_i \mid \tau) = \int_{|\mu_i|\leq |Y_i|/2} \phi(Y_i-\mu_i) \pi(\mu_i\mid\tau) \diff\mu_i + \int_{|\mu_i|> |Y_i|/2} \phi(Y_i-\mu_i) \pi(\mu_i\mid\tau) \diff\mu_i.
\]
For the first integral, using $(Y_i - \mu_i)^2 \geq Y_i^2/4$, we have
\[
\int_{|\mu_i|\leq |Y_i|/2} \phi(Y_i-\mu_i) \pi(\mu_i\mid\tau) \diff\mu_i \leq \frac{1}{\sqrt{2\pi}}e^{-\frac{Y_i^2}{8}} \int \pi(\mu_i\mid\tau)\diff\mu_i = \frac{1}{\sqrt{2\pi}}e^{-\frac{Y_i^2}{8}}.
\]
The second integral can be upper bounded by Lemma \ref{lm:HS.pdf.UB.LB}.
Also considering $\log(1+x) < x$ for $x>0$ and $1/\mu_i^2 < 4/Y_i^2$, we have
\[
\int_{|\mu_i|> |Y_i|/2} \phi(Y_i-\mu_i) \pi(\mu_i\mid\tau) \diff\mu_i \leq \frac{8}{\sqrt{2\pi^3}}\frac{\tau}{Y_i^2} \int_{|\mu_i|>|Y_i|/2} \phi(Y_i-\mu_i) \diff\mu_i = \frac{8}{\sqrt{2\pi^3}}\frac{\tau}{Y_i^2}.
\]
We conclude that when $i\in\mathcal S$,
\begin{equation*}
    \pi(Y_i \mid \tau) \leq \frac{1}{\sqrt{2\pi}}e^{-\frac{Y_i^2}{8}} + \frac{8}{\sqrt{2\pi^3}}\frac{\tau}{Y_i^2},
\end{equation*}
and thus
\begin{equation}\label{eq:Y.pi.over.C}
\frac{Y_i^2 \pi(Y_i \mid \tau)}{C'_{1,n}(Y_i)} \leq
\frac{\sqrt{2\pi^3}e^4}{16} \frac{9Y_i^2 + 64s_n^2/n^2}{Y_i^2[1-2\Phi(-|Y_i|/2)]} \left(\frac{8}{\sqrt{2\pi^3}}\tau + \frac{1}{\sqrt{2\pi}}Y_i^2 e^{-\frac{Y_i^2}{8}}\right).
\end{equation}
On the event $\mathcal B_n$, we have $|Y_i| > (c-\sqrt{2})\sqrt{2\log n}$ for all $i\in\mathcal S$.
Also note that
\[
\E_{Y_i \mid \theta_i}\left[Y_i^2 e^{-\frac{Y_i^2}{8}}\right] = \left(\frac{4}{5}\right)^{3/2} \left(1 + \frac{4}{5}\theta_i^2\right) e^{-\theta_i^2/10}.
\]
Under the theta-min condition in Assumption \ref{assumption:beta.min}, all $|\theta_i| > c\sqrt{2\log n}$, thus it is further upper bounded by the order of $n^{-c^2/5}\log n$.
Therefore
\begin{equation}\label{eq:signal.product.Bn.bound}
\E_{\bm Y \mid \bm\theta}\left[\prod_{i\in\mathcal S}\frac{Y_i^2\pi(Y_i\mid\tau)}{C'_{1,n}} \one{\mathcal B_n}\right] \leq \left(\frac{9e^4}{2}\tau + \frac{36\pi e^4}{25\sqrt 5}c^2n^{-c^2/5}\log n\right)^{s_n}.
\end{equation}
Note that in this discussion, $\tau > K\tau_{n,0}$.
For any choice of $c>\sqrt{6}$, the second term in the parenthesis on the right hand side of \eqref{eq:signal.product.Bn.bound} is dominated by the first.
This leads to
\begin{align*}
    \E_{\bm Y\mid \bm\theta}\left[\frac{\mathcal N_2}{\mathcal D}\one{\mathcal B_n}\right]
    &\leq \left(\frac{9 e^{4}n}{2s_n}\right)^{s_n}\int_{K\tau_{n,0}}^\infty \tau^{s_n+1} ne^{-n\tau} \diff \tau\\
    &= \left(\frac{9 e^{4}}{2}\right)^{s_n} \cdot n \int_{K\tau_{n,0}}^\infty \tau\exp\left(s_n\log\frac{n\tau}{s_n} - n\tau \right) \diff \tau\\
    &:= \left(\frac{9 e^{4}}{2}\right)^{s_n} \cdot n \int_{K\tau_{n,0}}^\infty \tau \exp[g(\tau)] \diff\tau.
\end{align*}
Note that $g'(\tau) = s_n/\tau - n$, and $g''(\tau) = -s_n/\tau^2 <0$.
Thus for any $\tau > K\tau_{n,0}$, when $n$ is large enough, for any $K\geq 2$,
\[
g'(\tau) < g'(K\tau_{n,0}) = \frac{n}{K} - n \leq -\frac{n}{2},
\]
and $g(\tau) \leq g(K\tau_{n,0}) - n(\tau-K\tau_{n,0})/2$. Therefore
\begin{align*}
\int_{K\tau_{n,0}}^\infty \tau \exp[g(\tau)] \diff\tau 
&\leq \exp[g(K\tau_{n,0})] \int_{K\tau_{n,0}}^\infty \tau e^{-\frac{n}{2}(\tau - K\tau_{n,0})} \diff \tau \\
&= \exp[g(K\tau_{n,0})]\left(\frac{2K\tau_{n,0}}{n} + \frac{4}{n^2}\right).
\end{align*}
We conclude that when $n$ is sufficiently large,
\begin{align*}
    \E_{\bm Y\mid \bm\theta}\left[\frac{\mathcal N_2}{\mathcal D}\one{\mathcal B_n}\right]
    &\leq 3K\tau_{n,0} \exp\left[g(K\tau_{n,0}) + s_n \log(9 e^{4}/2)\right]\\
    &\leq 3K\tau_{n,0} \exp\left[-s_n \left(K - \log K - \log(9 e^{4}/2)\right)\right]\\
    &= o(\tau_{n,0}),
\end{align*}
as long as we choose the constant $K$ large enough such that $K - \log K - \log(9 e^{4}/2) >0$.
Such a $K$ is a universal constant.
For instance, $K=8$ is a valid choice.


On the event $\mathcal B_n^c$, on the contrary, the following observation is useful in the proof.
Consider any observed $Y_i \in \R$.
For any $\tau_1 \geq \tau_2 > 0$, we have
\begin{equation}\label{eq:marginal.monotone}
    \pi(Y_i \mid \tau_1) \leq \frac{\tau_1}{\tau_2} \pi(Y_i \mid \tau_2).
\end{equation}
The proof of \eqref{eq:marginal.monotone} is straight forward.
Note that the marginal density is
\begin{align*}
\pi(Y_i\mid\tau)
&=
\int_0^\infty 
\frac{1}{\sqrt{2\pi(1+\lambda_i^2\tau^2)}}
\exp\!\left(-\frac{y^2}{2(1+\lambda_i^2\tau^2)}\right)
\frac{2}{\pi}\frac{1}{1+\lambda_i^2}
\diff\lambda_i\\
&=
\int_0^\infty 
\frac{1}{\sqrt{2\pi(1+u^2)}}
\exp\!\left(-\frac{y^2}{2(1+u^2)}\right)
\frac{2}{\pi}\frac{\tau}{\tau^2+u^2}
\diff u.
\end{align*}
\eqref{eq:marginal.monotone} is then proved by noting that for any $u\geq 0$,
\[
\frac{\tau_1}{\tau_1^2+u^2}
\le
\frac{\tau_1}{\tau_2}\frac{\tau_2}{\tau_2^2+u^2}.
\]
By multiplying \eqref{eq:marginal.monotone} over $i=1,\cdots,n$, we get
\[
\pi(\bm Y \mid \tau_1) \le \left(\frac{\tau_1}{\tau_2}\right)^n \pi(\bm Y \mid \tau_2).
\]
Equivalently, the function $w_{\bm Y}(\tau) := \frac{\pi(\bm Y \mid \tau)}{\tau^n}$ is nonincreasing in $\tau$.
Moreover, under the exponential hyperprior of $\tau$, the posterior density of $\tau$ satisfies 
\[
\pi(\tau\mid\bm Y) \propto \pi(\bm Y \mid \tau) \pi(\tau) =  n e^{-n\tau}\tau^n w_{\bm Y}(\tau).
\]
Let $q$ denote the density of $\Gamma(n+1,n)$ distribution,
\[
q(\tau)=\frac{n^{n+1}}{\Gamma(n+1)}\tau^n e^{-n\tau}.
\]
Then the posterior mean of $\tau$ can be equivalently written as a mean over this Gamma distribution,
\[
\E[\tau\mid\bm Y]
=
\frac{\E_{\tau \sim q(\tau)}[\tau w_{\bm Y}(\tau)]}{\E_{\tau \sim q(\tau)}[w_{\bm Y}(\tau)]}.
\]

Because the map $\tau \mapsto w_{\bm Y}(\tau)$ is decreasing, we have
$\text{Cov}(\tau,w_{\bm Y}(\tau))\le0$.
Hence
\[
\E_{\tau \sim q(\tau)}[\tau w_{\bm Y}(\tau)]
\leq
\E_{\tau \sim q(\tau)}[\tau]\E_{\tau \sim q(\tau)}[w_{\bm Y}(\tau)].
\]
Since $\E_{\tau\sim q(\tau)}[\tau] = \frac{n+1}{n}$, we therefore arrive at the following bound on the posterior mean of $\tau$.
For any observation $\bm Y \in \R^n$,
\begin{equation}\label{eq:tau.posterior.mean.bound}
\E[\tau\mid\bm Y]
\leq
\frac{n+1}{n}.
\end{equation}
Because the ratio $\mathcal N_2 / \mathcal D$ is precisely the truncated posterior mean $\E[\tau \one{\tau > K\tau_{n,0}} \mid \bm Y]$, it is bounded by $\E[\tau \mid \bm Y]$.
Therefore we arrive at
\begin{align*}
\E_{\bm Y\mid \bm\theta}\left[\frac{\mathcal N_2}{\mathcal D}\one{\mathcal B_n^c}\right]
&\leq \E_{\bm Y\mid \bm\theta}\left[\E[\tau\mid \bm Y]\one{\mathcal B_n^c}\right]\\
&\leq \frac{n+1}{n} \P\{\mathcal B_n^c\} \leq \frac{4}{n} = o(\tau_{n,0}).
\end{align*}
Combining all above, we have
\[
\sup_{\bm\theta\in\Theta_n(s_n,c)}\E_{\bm Y\mid \bm\theta}\E[\tau\mid\bm Y] \leq K\tau_{n,0} (1+o(1)).
\]
This finishes the proof of Lemma \ref{lm:tau.not.overshooting}.

\subsection{Proof of Lemma \ref{lm:tau.not.undershooting}}\label{sec:pf.lm.tau.not.undershooting}

To prove this lemma, a number of results in the previous proof of Lemma \ref{lm:tau.not.overshooting} can be reused.
First, fix a quantity $\tau^\ast_n = C_\tau s_n/n$, where the universal constant $C_\tau$ will be chosen later in the proof.
Note that
\[
\E\left[\log\frac{1}{\tau} \mid \bm Y\right] = \E\left[\log\frac{\tau^\ast_n}{\tau} \mid \bm Y\right] + \log\frac{1}{\tau^\ast_n}.
\]
The main goal of this proof is to show that the first term on the right-hand side is negligible.
\[
\E\left[\log\frac{\tau^\ast_n}{\tau} \mid \bm Y\right] = \frac{\int \log(\tau^\ast_n/\tau)\pi(\bm Y \mid \tau) \pi(\tau) \diff\tau}{\int \pi(\bm Y \mid \tau)\pi(\tau)\diff\tau}.
\]
The denominator is exactly $\mathcal D$ that we defined in Appendix \ref{sec:pf.lm.tau.not.overshooting}.
We can similarly split the integral in the numerator into $\mathcal N'_1 + \mathcal N'_2$ by $\tau^\ast_n$.
Then we have
\begin{align*}
\mathcal N'_2
= \int_{\tau^\ast_n}^\infty \log\frac{\tau^\ast_n}{\tau} \pi(\bm Y \mid \tau) \pi(\tau) \diff\tau \leq \mathcal D.
\end{align*}
Rather more complexity arises from the other term
\[
\mathcal N'_1 = \int_0^{\tau^\ast_n} \log\frac{\tau^\ast_n}{\tau} \pi(\bm Y \mid \tau) \pi(\tau) \diff\tau.
\]
To tackle this integral, we need to split into the two events, $\mathcal C_n$ and $\mathcal C_n^c$, where $\mathcal C_n=\{\min_{i\in\mathcal S}|Y_i|\ge 1\}$.
Note that since $(c-\sqrt{2})\sqrt{2\log n}\rightarrow\infty$, when $n$ is large enough, we have $\mathcal B_n\subseteq \mathcal C_n$ and $\P\{\mathcal C_n^c\} \leq \P\{\mathcal B_n^c\} \leq 2/n$.
We decompose
\begin{equation}\label{eq:BnC.split}
\E_{\bm Y\mid\bm\theta}\left[\frac{\mathcal N'_1}{\mathcal D}\right]
=
\E_{\bm Y\mid\bm\theta}\left[\frac{\mathcal N'_1}{\mathcal D}\one{\mathcal C_n}\right]
+
\E_{\bm Y\mid\bm\theta}\left[\frac{\mathcal N'_1}{\mathcal D}\one{\mathcal C_n^c}\right].
\end{equation}
Note that
\begin{align*}
\frac{\mathcal N'_1}{\mathcal D} = \E\left[\log\frac{\tau^\ast_n}{\tau}\one{\tau < \tau^\ast_n} \mid \bm Y\right] = \int_0^\infty \Pi\left\{\tau < \tau^\ast_n e^{-t} \mid \bm Y\right\} \diff t.
\end{align*}
Here, the last step applied Fubini's Theorem and the identity that for any random variable $X$, $\log(1/X) \one{X<1} = \int_0^\infty \one{X < e^{-t}} \diff t$.
The problem then boils down to controlling
\begin{equation}\label{eq:Pi.tau}
    \Pi\{\tau < \tau^\ast_n e^{-t} \mid \bm Y\}
    = \frac{\int_0^{\tau^\ast_n e^{-t}}\pi(\bm Y \mid \tau) \pi(\tau) \diff\tau}{\int_0^\infty \pi(\bm Y \mid \tau) \pi(\tau) \diff\tau}.
\end{equation}
The control is slightly different on the events $\mathcal C_n$ and $\mathcal C_n^c$.
We first consider $\mathcal C_n$, where for all $i\in \mathcal S$, the scale of the observation $|Y_i|$ is larger than one.
In this case, we apply the lower bound on the denominator $\mathcal D$ used in \eqref{eq:D.i.in.S} and \eqref{eq:D.i.notin.S}.
Using the lower bound of $\mathcal D$, \eqref{eq:Pi.tau} can be upper bounded by
\[
\Pi\{\tau < \tau^\ast_n e^{-t} \mid \bm Y\}
\leq \int_0^{\tau^\ast_n e^{-t}} \prod_{i\in\mathcal S} \frac{Y_i^2 \pi(Y_i \mid \tau)}{C'_{1,n}(Y_i)} \cdot \prod_{i\notin\mathcal S}\frac{\pi(Y_i\mid\tau)}{\phi(Y_i)} \cdot \left(\frac{n}{s_n}\right)^{s_n} \pi(\tau) \diff \tau.
\]
We reuse the bound \eqref{eq:Y.pi.over.C} on the event $\mathcal C_n$.
On $\mathcal C_n$ we have $|Y_i|\ge 1$ for all $i\in\mathcal S$, and therefore
\[
1-2\Phi(-|Y_i|/2)\geq 1-2\Phi(-1/2),
\]
and thus
\[
\frac{9Y_i^2+64s_n^2/n^2}{Y_i^2\,[1-2\Phi(-|Y_i|/2)]}\leq \frac{10}{1-2\Phi(-1/2)} := C_\Phi,
\]
which is a universal constant.
Following the same derivation that led to \eqref{eq:signal.product.Bn.bound},
we have for $\tau\in(0,\tau_n^\ast)$, with the choice $c>\sqrt{6}$,
\[
\E_{\bm Y \mid \bm\theta}\!\left[\prod_{i\in\mathcal S}\frac{Y_i^2\pi(Y_i\mid\tau)}{C'_{1,n}(Y_i)} \one{\mathcal C_n}\right]
\le \left(\frac{C_\Phi e^4}{2}\tau+ C_\Phi \delta_n\right)^{s_n},
\]
where $\delta_n$ is the second term in the parenthesis on the right hand side of \eqref{eq:signal.product.Bn.bound}. As $c>\sqrt{6}$, we still have $\delta_n=o(\tau_n^\ast)$.
Together with the fact that $\E_{Y_i \mid \theta_i}[\pi(Y_i \mid \tau) / \phi(Y_i)] = 1$, and Fubini's Theorem, we arrive at
\begin{align*}
\E_{\bm Y \mid \bm\theta}[\Pi\{\tau < \tau^\ast_n e^{-t} \mid \bm Y\}\one{\mathcal C_n}]
&\leq \left(\frac{n}{s_n}\right)^{s_n} \int_0^{\tau^\ast_ne^{-t}} \left( \frac{C_\Phi e^4}{2}\tau + C_\Phi \delta_n \right)^{s_n} ne^{-n\tau} \diff \tau\\
&\leq \left(\frac{C_\Phi e^4n}{2s_n}\right)^{s_n} \frac{n}{s_n+1} (\tau^\ast_n)^{s_n+1} e^{-(s_n+1)t} (1+o(1)).
\end{align*}
For the last inequality, we used the trivial bound $e^{-n\tau}\leq 1$ and solved the remaining integral by a linear change of variable.
For any choice of constant $C_\tau > 2/(C_\Phi e^4)$, then the right hand side is further upper bounded by $C_\tau e^{-(s_n+1)t}$.
Therefore we have
\begin{align*}
\E_{\bm Y \mid \bm \theta}\left[\frac{\mathcal N'_1}{\mathcal D} \one{\mathcal C_n}\right] 
\leq  C_\tau \int_0^\infty e^{-(s_n+1)t}\diff t = \frac{C_\tau}{s_n+1}.
\end{align*}

Next, we shall control the behavior of $\mathcal N_1'/\mathcal D$ on the low-probability event $\mathcal C_n^c$ and bound the second term on the right-hand side of \eqref{eq:BnC.split}.
We may write $\mathcal N_1'/\mathcal D = \E[f(\tau) \mid \bm Y]$, where
\[
f(\tau) = \log\!\left(\frac{\tau_n^\ast}{\tau}\right)\one{\tau<\tau_n^\ast}
\]
for $\tau>0$.
By Lemma \ref{lm:post.tau} and the exponential hyperprior $\pi(\tau)=n e^{-n\tau}\one{\tau>0}$,
the posterior density can be written as
\[
\pi(\tau\mid \bm Y)
=
\frac{n e^{-n\tau}M_{\bm Y}(\tau)}{\int_0^\infty n e^{-nt}M_{\bm Y}(t)\diff t},
\quad \text{where }
M_{\bm Y}(\tau)=\prod_{i=1}^n H(Y_i,\tau).
\]
Here, the confluent hypergeometric function of two variables can be rewritten in the integral form,
\[
H(y,\tau)
=\tau \Phi_1\!\left(1,1,\frac{3}{2},-\frac{y^2}{2},1-\tau^2\right)
=
\frac{\tau}{2}\int_0^1 u^{-1/2} \frac{1}{1-(1-\tau^2)u}e^{-y^2u/2} \diff u.
\]

We shall now prove that for every $y\in\mathbb R$, the map $\tau\mapsto H(y,\tau)$ is strictly increasing on $(0,\infty)$.
To see this, for any fixed $y\in\mathbb R$ and $\tau>0$, we have the partial derivative
\[
\partial_\tau H(y,\tau)
=
\frac{1}{2} \int_0^1 e^{-y^2u/2} h_\tau(u)\diff u,
\qquad
h_\tau(u)
:=
\frac{1-u-\tau^2u}{u^{1/2}[1-(1-\tau^2)u]^2}.
\]
Let $u_0=1/(1+\tau^2)$. Then $h_\tau(u)\geq 0$ for $u\in(0,u_0]$ and $h_\tau(u)\leq 0$ for $u\in[u_0,1)$.
Since $u\mapsto e^{-y^2u/2}$ is decreasing, we have the pointwise bounds $e^{-y^2u/2}\geq e^{-y^2u_0/2}$ for $u\geq u_0$, and $e^{-y^2u/2}\leq e^{-y^2u_0/2}$ for $u\leq u_0$.
Multiplying by $h_\tau(u)$ and integrating yields
\begin{multline*}
\partial_\tau H(y,\tau)
=
\int_0^{u_0} e^{-y^2u/2}h_\tau(u)\diff u
+
\int_{u_0}^1 e^{-y^2u/2}h_\tau(u)\diff u\\
\geq
\int_0^{u_0} e^{-y^2u_0/2}h_\tau(u)\diff u
+
\int_{u_0}^1 e^{-y^2u_0/2}h_\tau(u)\diff u
=
e^{-y^2u_0/2} \partial_\tau H(0,\tau).
\end{multline*}
It remains to show $\partial_\tau H(0,\tau)>0$ for all $\tau>0$.
For $y=0$, substitute $u=t^2$ to obtain
\[
H(0,\tau)=\tau\int_0^1\frac{1}{1+(\tau^2-1)t^2}\,dt
=
\begin{cases}
\dfrac{\tau}{\sqrt{1-\tau^2}}\tanh^{-1}(\sqrt{1-\tau^2}), & 0<\tau<1,\\
1, & \tau=1,\\
\dfrac{\tau}{\sqrt{\tau^2-1}}\tan^{-1}(\sqrt{\tau^2-1}), & \tau>1.
\end{cases}
\]
Differentiating these expressions gives
\[
\partial_\tau H(0,\tau)
=
\begin{cases}
\dfrac{\tanh^{-1}(\sqrt{1-\tau^2})-\sqrt{1-\tau^2}}{(1-\tau^2)^{3/2}}, & 0<\tau<1,\\
\dfrac{\sqrt{\tau^2-1}-\tan^{-1}(\sqrt{\tau^2-1})}{(\tau^2-1)^{3/2}}, & \tau>1,
\end{cases}
\]
which is strictly positive since $\tanh^{-1}(a)>a$ for any $a\in(0,1)$ and $\tan^{-1}(a)<a$ for any $a>0$.
Hence $\partial_\tau H(y,\tau)>0$ for all $y$ and $\tau>0$.
Consequently $M_{\bm Y}(\tau)$ is strictly increasing in $\tau$ for every fixed $\bm Y$.
Note that the function $f(\tau)$ is nonincreasing with $\E_{\tau\sim\pi(\tau)}[|f(\tau)|]<\infty$.
Then we arrive at
\[
\E[f(\tau)\mid \bm Y]
=
\frac{\E_{\tau\sim\pi(\tau)}[f(\tau)M_{\bm Y}(\tau)]}{\E_{\tau\sim\pi(\tau)}[M_{\bm Y}(\tau)]}
\le
\E_{\tau\sim\pi(\tau)}[f(\tau)].
\]
Note that the left-hand side is exactly $\mathcal N'_1/\mathcal D$. Therefore we have
\begin{align*}
\frac{\mathcal N_1'}{\mathcal D}
&\le
\int_0^{\tau_n^\ast}\log\left(\frac{\tau_n^\ast}{\tau}\right)n e^{-n\tau}\diff\tau
\le
n\int_0^{\tau_n^\ast}\log\left(\frac{\tau_n^\ast}{\tau}\right)\diff\tau
=
n\tau_n^\ast
=
C_\tau s_n.
\end{align*}
Consequently, on the event $\mathcal C_n^c$,
\[
\E_{\bm Y\mid\bm\theta}\!\left[\frac{\mathcal N_1'}{\mathcal D}\one{\mathcal C_n^c}\right]
\le
C_\tau s_n\,\P\{\mathcal C_n^c\} \leq 2C_\tau\frac{s_n}{n}.
\]

We conclude that
\[
\E_{\bm Y \mid \bm \theta}\left[\frac{\mathcal N'_1}{\mathcal D} \right]  \leq \log \frac{C_\tau}{s_n+1} + 2C_\tau \frac{s_n}{n},
\]
and moreover
\[
\sup_{\bm\theta\in\Theta_n(s_n,c)}\E_{\bm Y \mid \bm\theta}\E\left[\log\frac{1}{\tau} \mid \bm Y\right] \leq \log\frac{n}{s_n} + 4 + \log\frac{9}{2} + \log \frac{2}{9e^4(s_n+1)} + \frac{C_\tau}{s_n+1} + 2C_\tau \frac{s_n}{n}.
\]
The only dominating term is $\log(n/s_n)$, and all others are of order $O(1)$ or less.
This finishes the proof of Lemma \ref{lm:tau.not.undershooting}.

\subsection{A Corollary of Lemma \ref{lm:tau.not.overshooting}}

The following lemma is useful in the proof of Theorem \ref{thm:adaptive.beta.min}.
When fixing an observation $Y_i$ from the noise group, that is, $i\notin\mathcal S$, it might be useful to specifically isolate the effect of $Y_i$ from the posterior $\pi(\tau \mid \bm Y)$.

\begin{corollary}\label{cor:loo.overshoot.fixedy}
Let $i\notin\mathcal S$, so $\theta_i=0$.
Fix any $y\in\R$ and define the posterior based on the sample $(\bm Y_{-i},Y_i=y)$ under the same
hyperprior $\pi(\tau)=ne^{-n\tau}\one{\tau>0}$.
Then under Assumption~\ref{assumption:beta.min} with $c>\sqrt6$,
\[
\sup_{\bm\theta\in\Theta_n(s_n,c)}
\E_{\bm Y_{-i}\mid\bm\theta}\E[\tau\mid \bm Y_{-i},Y_i=y]
\leq
K\tau_{n,0}(1+o(1)).
\]
Here, $K$ is a universal constant that satisfies $K - \log K - \log(9 e^{4}/2) > 0$.
\end{corollary}

\begin{proof}
The proof is very similar to that of Lemma \ref{lm:tau.not.overshooting} except for a few details.
Fix $y\in\R$. We can still split the integral by $K\tau_{n,0}$,
\[
\E(\tau\mid \bm Y_{-i},y)
=
\frac{\int_0^\infty \tau\,\pi(\bm Y_{-i},y\mid\tau)\pi(\tau)\,d\tau}{\int_0^\infty \pi(\bm Y_{-i},y\mid\tau)\pi(\tau)\,d\tau}
=
\frac{\mathcal N_1(y)+\mathcal N_2(y)}{\mathcal D(y)},
\]
where $\mathcal N_1(y)=\int_0^{K\tau_{n,0}}\tau\pi(\bm Y_{-i},y\mid\tau)\pi(\tau)\diff \tau$, 
$\mathcal N_2(y)=\int_{K\tau_{n,0}}^{\infty}\tau\pi(\bm Y_{-i},y\mid\tau)\pi(\tau)\diff \tau$,
and $\mathcal D(y)=\int_0^\infty \pi(\bm Y_{-i},y\mid\tau)\pi(\tau)\diff\tau$.
Trivially, we still have $\mathcal N_1(y)\le K\tau_{n,0} \mathcal D(y)$.
It remains to control $\mathcal N_2(y)/\mathcal D(y)$.
Recall the monotonicity inequality in \eqref{eq:marginal.monotone}.
For $\tau\in[s_n/n, 2s_n/n]$, we have
\[
\pi(y\mid\tau)\geq \frac{n\tau}{2s_n}\pi(y\mid \tau = 2s_n/n).
\]
Therefore, restricting the denominator integral to $\tau\in(s_n/n,2s_n,n)$ and using the lower bounds from the proof of Lemma \ref{lm:tau.not.overshooting} for the remaining coordinates, we get that when $Y_i=y$ is fixed, the denominator is lower bounded by
\begin{equation*}
\mathcal D(y)
\geq
\frac{\pi(y\mid \tau = 2s_n/n)}{2s_n/n}
\cdot \prod_{j\in\mathcal S}\frac{C'_{1,n}(Y_j)}{Y_j^2}
\cdot \prod_{\substack{j\notin\mathcal S\\ j\neq i}}\phi(Y_j)
\cdot \left(\frac{s_n}{n}\right)^{s_n+1},
\end{equation*}
where $C'_{1,n}(\cdot)$ is the same factor as in Appendix \ref{sec:pf.lm.tau.not.overshooting} and the $(s_n/n)^{s_n+1}$ power appears because of the extra multiplicative $\tau$ factor above.

For the numerator $\mathcal N_2(y)$, note that $\tau\geq K\tau_{n,0}\geq 2s_n/n$.
Applying the monotonicity inequality in \eqref{eq:marginal.monotone} yields
\[
\pi(y\mid\tau)\leq \frac{n\tau}{2s_n}\pi(y\mid\tau = 2s_n/n),
\]
and thus
\begin{equation*}
\mathcal N_2(y)
\leq
\frac{\pi(y\mid \tau = 2s_n/n)}{2s_n/n}
\int_{K\tau_{n,0}}^\infty \tau^2 \pi(\bm Y_{-i}\mid\tau) \pi(\tau)\diff \tau.
\end{equation*}
Therefore, the $\pi(y\mid\tau=2s_n/n)$ factors cancel and we obtain
\begin{equation*}
\frac{\mathcal N_2(y)}{\mathcal D(y)}
\leq
\left(\frac{n}{s_n}\right)^{s_n+1}
\int_{K\tau_{n,0}}^{\infty}
\prod_{j\in\mathcal S}\frac{Y_j^2\,\pi(Y_j\mid\tau)}{C'_{1,n}(Y_j)} \cdot
\prod_{\substack{j\notin\mathcal S\\ j\neq i}}\frac{\pi(Y_j\mid\tau)}{\phi(Y_j)}\cdot
\tau^2\pi(\tau)\diff \tau.
\end{equation*}

Same as in Appendix \ref{sec:pf.lm.tau.not.overshooting}, we may split according to the event $\mathcal B_n$.
Since $\mathcal B_n$ depends only on $\{Y_j:j\in\mathcal S\}\subset\bm Y_{-i}$, its probability bound is unaffected by fixing $Y_i=y$.
We still have $\P\{\mathcal B_n^c\} \leq 2/n$.
On $\mathcal B_n$, the signal product in \eqref{eq:signal.product.Bn.bound} still holds.
Moreover, for $j\notin\mathcal S$, $Y_j\sim N(0,1)$ and we still have
$\E[\pi(Y_j\mid\tau)/\phi(Y_j)]=1$.
By Fubini's Theorem, we get
\begin{align*}
\E_{\bm Y_{-i}\mid\bm\theta} \left[\frac{\mathcal N_2(y)}{\mathcal D(y)}\one{\mathcal B_n}\right]
&\leq
\left(\frac{9e^4}{2}\right)^{s_n} \left(\frac{n}{s_n}\right)^{s_n+1}
\int_{K\tau_{n,0}}^\infty \tau^{s_n+2} n e^{-n\tau}\diff\tau\\
&=
\left(\frac{9e^4}{2}\right)^{s_n}\frac{n^2}{s_n}\int_{K\tau_{n,0}}^\infty \tau^2\exp[g(\tau)]\diff \tau.
\end{align*}
Recall that $g(\tau) =s_n\log\frac{n\tau}{s_n}-n\tau$.
Using a similar manipulation of $g(\tau)$ as in Appendix \ref{sec:pf.lm.tau.not.overshooting}, we get
\[
\int_{K\tau_{n,0}}^\infty \tau^2 e^{g(\tau)}\diff \tau
\leq 
\exp[g(K\tau_{n,0})]\left(\frac{2K^2\tau_{n,0}^2}{n}+\frac{8K\tau_{n,0}}{n^2}+\frac{16}{n^3}\right).
\]
We may thus conclude that when $n$ is sufficiently large, for all $y\in \R$,
\begin{align*}
\E_{\bm Y_{-i}\mid\bm\theta}\left[\frac{\mathcal N_2(y)}{\mathcal D(y)}\one{\mathcal B_n}\right]
&\leq 2K^2\tau_{n,0}\exp\left[-s_n\left( K - \log K - \log(9e^4/2) \right)\right]
= o(\tau_{n,0})
\end{align*}
for our choice of universal constant $K$.
On $\mathcal B_n^c$, we use the similar bound as in Appendix \ref{sec:pf.lm.tau.not.overshooting},
\[
\E_{\bm Y_{-i}\mid\bm\theta} \left[\frac{\mathcal N_2(y)}{\mathcal D(y)}\one{\mathcal B_n^c}\right]
\leq
\E_{\bm Y_{-i}\mid\bm\theta} \left[\E(\tau\mid \bm Y_{-i},y)\one{\mathcal B_n^c}\right]
\leq
\frac{n+1}{n}\P(\mathcal B_n^c)
\leq \frac{4}{n}.
\]
The proof of Corollary \ref{cor:loo.overshoot.fixedy} is finished by combining the two bounds on $\mathcal B_n$ and $\mathcal B_n^c$.
\end{proof}

\section{Proof of Theorem \ref{thm:adaptive.beta.min}}\label{sec:pf.thm.adaptive.beta.min}
\subsection{Signal Case: $\theta_i \neq 0$}
Note that
\[
\E_{\bm Y \mid \bm \theta} \E_{\tau \mid \bm Y} [L(\theta_i, \hat p(\cdot \mid Y_i, \tau))] = \E_{\bm Y \mid \bm \theta} \E_{\tau \mid \bm Y} \E_{\tilde {\bm Y} \mid \bm \theta} [\tilde g(Y_i, \tilde Y_i, \theta_i, v)].
\]
Similar to the proof of Theorem \ref{thm:minimax.fixed.tau}, we provide an upper bound of this quantity under two separate cases, when the observation $|Y_i|$ is large or small.
We consider the assumption that the true underlying $\bm\theta \in \Theta_n(s_n,c)$ as defined in \eqref{eq:min.signal.set}.
One difference, however, is that since the parameter $\tau$ is no longer fixed, the threshold for $|Y_i|$ is now set to $\zeta_{n,v} = \sqrt{2v\log(n / s_n)}$.
Although $\zeta_{n,v}$ contains the true underlying sparsity level $s_n$ that is unknown, it is only used as an intermediate step in the proof.
The full-Bayes guarantee holds without the knowledge of $s_n$.

\subsubsection{Large observation: $|Y_i| > \sqrt{2v \log(n/s_n)}$}
By the decomposition of $\tilde g(Y_i, \tilde Y_i, \theta_i, v)$ in Lemma \ref{lm:g.tilde.decomp}, using the expressions of $N_v(Y_i, \tilde Y_i)$ and $D(Y_i)$ in \eqref{eq:Nv.D.form.2}, we get
\begin{multline}\label{eq:g.tilde.decomp.signal.large.obs}
\tilde g(Y_i, \tilde Y_i, \theta_i, v) = \frac{\tilde Y_i \theta_i}{r} - \frac{\theta_i^2}{2r} + \log(\sqrt{v}) - \frac{v}{2}(Y_i + \tilde Y_i / r)^2 + \frac{1}{2}Y_i^2\\
+ \log \frac{\int_0^1 (1-u)^{-1/2}\frac{1}{1-(1-\tau^2)u}e^{-\frac{Y_i^2}{2}u} \diff u}{\int_0^1 (1-u)^{-1/2}\frac{1}{1-(1-\tau^2/v)u}e^{-\frac{v}{2}(Y_i + \tilde Y_i / r)^2 u} \diff u}.
\end{multline}
Note that $\tau$ is not involved in any of the first five terms on the right hand side of \eqref{eq:g.tilde.decomp.signal.large.obs}.
Let $Z_i$ be a standard Gaussian random variable, and write $Y_i = \theta_i + Z_i$ and $\tilde Y_i = \theta_i + \sqrt{r}Z_i$. Therefore
\begin{multline*}
\E_{\bm Y \mid \bm \theta} \E_{\tau \mid \bm Y} \E_{\tilde {\bm Y} \mid \bm \theta}\left[\left(\frac{\tilde Y_i \theta_i}{r} - \frac{\theta_i^2}{2r} + \log(\sqrt{v}) - \frac{v}{2}(Y_i + \tilde Y_i / r)^2 + \frac{1}{2}Y_i^2\right) \one{|Y_i| > \zeta_{n,v}}\right]\\
= \log(\sqrt{v}) - (v^{-1/2} - 1) \E[\theta_i Z_i\one{|Z_i+\theta_i| > \zeta_{n,v}}].
\end{multline*}
Using $\log(\sqrt{v})< 0$ and \eqref{eq:cross.term.exp.neg}, the right hand side is upper bounded by zero.
This leaves us with only the last term in \eqref{eq:g.tilde.decomp.signal.large.obs}.
Recall that $g_v(u) = (1-u)^{-1/2} \frac{1}{1-(1-\tau^2/v)u}$.
By the mean value theorem, for any $u\in [0,1]$, there exists $\bar u_v \in [0,u]$ such that
\[
g_v(u) = g_v(0) + ug'_v(\bar u_v).
\]
Following the similar procedure in Appendix \ref{sec:known.tau.signal.large.obs} that led to \eqref{eq:risk.frac.ub}, for any choice of $1\leq s\leq 1$, we have
\begin{multline}\label{eq:adaptive.risk.frac.ub}
    \frac{\int_0^1 (1-u)^{-1/2}\frac{1}{1-(1-\tau^2)u}e^{-\frac{Y_i^2}{2}u} \diff u}{\int_0^1 (1-u)^{-1/2}\frac{1}{1-(1-\tau^2/v)u}e^{-\frac{v}{2}(Y_i + \tilde Y_i / r)^2 u} \diff u}\\
    \leq \frac{v(Y_i + \tilde Y_i / r)^2}{Y_i^2} \frac{h^\ast_1(Y_i,s) + \frac{g'_1(\bar u_1)}{Y_i^2} h^\ast_2(Y_i,s) + 2Y_i^2 e^{-\frac{Y_i^2}{2}s} h_3^1(\tau)}{\tilde h_1^v(Y_i, \tilde Y_i, s) + \frac{g'_v(\bar u_v)}{v(Y_i+\tilde Y_i/r)^2} \tilde h_2^v(Y_i,\tilde Y_i, s)},
\end{multline}
where
\begin{align*}
    h^\ast_1(Y_i,s) &= 2-2e^{-\frac{Y_i^2}{2}s},\\
    h^\ast_2(Y_i,s) &= 4-4\left(1+\frac{Y_i^2}{2}s\right)e^{-\frac{Y_i^2}{2}s},\\
    \tilde h_1^v(Y_i, \tilde Y_i, s) &= 2-2e^{-\frac{v}{2}(Y_i + \tilde Y_i/r)^2 s},\\
    \tilde h_2^v(Y_i, \tilde Y_i, s) &= 4-4\left(1+\frac{v}{2}(Y_i + \tilde Y_i /r)^2 s \right)e^{-\frac{v}{2}(Y_i + \tilde Y_i / r)^2 s},
\end{align*}
and recall that $h_3^1(\tau) = \tau^{-1}(1-\tau^2)^{-1/2} \tan^{-1}(\sqrt{1-\tau^2} / \tau)$.
The task of bounding the expectation over the logarithm of this right hand side is twofold.

Since the first factor $v(Y_i + \tilde Y_i/r)^2 / Y_i^2$ does not involve $\tau$, using the similar procedure as in \eqref{eq:theta.Z.ratio.bound}, we can upper bound 
\begin{align*}
&\E_{\bm Y\mid \bm\theta}\E_{\tau\mid\bm Y} \E_{\tilde{\bm Y}\mid \bm\theta} \left[\log \frac{v(Y_i + \tilde Y_i/r)^2}{Y_i^2} \one{|Y_i| > \zeta_{n,v}}\right]\\
&= 2\E\left[\log\left(1+\frac{(v^{-1/2}-1)|\theta_i|}{|Z_i + \theta_i|} \right) \one{|Z_i + \theta_i| > \zeta_{n,v}} \right]\\
&\leq (v^{-1/2}-1) [6+4(s_n/n)^{v/4}].
\end{align*}

Meanwhile, similar to \eqref{eq:risk.frac.ub.bound}, the second factor on the right hand side of \eqref{eq:adaptive.risk.frac.ub} can be upper bounded by
\begin{equation}\label{eq:adaptive.signal.large.obs.key.ratio}
\frac{2 + \frac{4g'_1(\bar u_1)}{Y_i^2} + 2Y_i^2 e^{-\frac{Y_i^2}{2}s} h_3^1(\tau)}{2(1-e^{-\frac{v}{2}(Y_i+\tilde Y_i/r)^2 s}) + \frac{4g'_v(\bar u_v)}{v(Y_i+\tilde Y_i/r)^2 }(1-e^{-\frac{\zeta_{n,v}^2}{2}s})}.
\end{equation}

We fix $s=1/2$, and for brevity define $W_i = (Y_i+\tilde Y_i/r)^2$.
Then the ratio in \eqref{eq:adaptive.signal.large.obs.key.ratio} can be written as
\[
R_i(\tau)
:=
\frac{2 + \frac{4g'_1(\bar u_1)}{Y_i^2} + 2Y_i^2 e^{-\frac{Y_i^2}{4}} h_3^1(\tau)}
{2(1-e^{-vW_i/4}) + \frac{4g'_v(\bar u_v)}{vW_i }(1-e^{-\frac{\zeta_{n,v}^2}{4}}) }.
\]

By elementary algebra, we can see that for all $\tau\in(0,1)$ and $u\in(0,1/2)$, the derivative is bounded, i.e., $|g'_v(u)| \leq C_v$ for some finite constant $C_v$ depending only on $v$.
We may thus bound the denominator of $R_i(\tau)$ from below on the event $\mathcal W_i:=\{W_i\ge w_0\}$ as follows:
\begin{equation}\label{eq:den.lower.bound.onFi}
2(1-e^{-vW_i/4})+\frac{4g'_v(\bar u_v)}{vW_i}(1-e^{-\zeta_{n,v}^2 s/2})
\ge 2(1-e^{-vW_i/4})-\frac{4C_v}{vW_i}.
\end{equation}
We may choose a fixed constant $w_0\ge 1$ related to $v$ that is so large such that $\frac{4C_v}{vw_0}\le 1-e^{-\frac v4 w_0}$.
Then on $\mathcal W_i$, \eqref{eq:den.lower.bound.onFi} can be further lower bounded by $1-e^{-vw_0/4}$, a constant that is bounded away from zero.

Meanwhile, on the event $\{|Y_i|>\zeta_{n,v}\}$, the $4g'_1(\bar u_1) / Y_i^2$ term is of order $o(1)$ and thus negligible.
Also, for $\tau\in(0,1)$ we have the standard bound
\[
h_3^1(\tau)
=
\tau^{-1}(1-\tau^2)^{-1/2}\tan^{-1}\left(\frac{\sqrt{1-\tau^2}}{\tau}\right)
\le \frac{\pi}{2}\cdot \frac{1}{\tau}.
\]
Then on the intersection of the events $\mathcal W_i\cap\{|Y_i|>\zeta_{n,v}\}$ and for $\tau\in(0,1)$,
\begin{equation*}
R_i(\tau)
\le
\frac{2+o(1)+2Y_i^2 e^{-Y_i^2/4} h_3^1(\tau)}{1-e^{-vw_0/4}}
\leq
\frac{2}{1-e^{-vw_0/4}}\left(1+Y_i^2 e^{-Y_i^2/4} \frac{1}{\tau}\right).
\end{equation*}
Now consider a fixed observation $\bm Y$.
We have shown in Appendix \ref{sec:pf.lm.tau.not.undershooting} that the posterior density has the form $\pi(\tau\mid\bm Y)\propto ne^{-n\tau}M_{\bm Y}(\tau)$ with $M_{\bm Y}(\tau)$ increasing in $\tau$.
On the contrary, note that the map $\tau\mapsto \log(1+a_i/\tau)$ is nonincreasing.
Using $\log(1+x)\leq \log 2\cdot \one{x\leq 1}+\log(2x)\cdot \one{x>1}$, we obtain
\begin{align*}
\E\left[\log R_i(\tau) \one{\tau<1} \mid \bm Y\right]
&\leq
\E_{\tau\sim\pi(\tau)}\left[\log\left(1+\frac{Y_i^2 e^{-Y_i^2/4}}{\tau}\right)\right] + \log\frac{2}{1-e^{-vw_0/4}}\\
&\leq
\log 2 + \int_0^{Y_i^2 e^{-Y_i^2/4}} \log\left(\frac{2Y_i^2 e^{-Y_i^2/4}}{\tau}\right)ne^{-n\tau}\diff\tau + \log\frac{2}{1-e^{-vw_0/4}}\\
&\leq n(1+\log 2) Y_i^2 e^{-Y_i^2/4} + \log\frac{4}{1-e^{-vw_0/4}}
\end{align*}
Here, we used the trivial bound $e^{-n\tau}\leq 1$ in the last inequality.

Taking expectation over $Y_i\sim N(\theta_i,1)$ and using the theta-min condition $|\theta_i|\ge c\sqrt{2\log n}$,
a direct computation yields
\[
\E_{Y_i\mid\theta_i}\!\left[Y_i^2e^{-Y_i^2/4}\right]
=
\Big(\frac{2}{3}\Big)^{5/2}\Big(\theta_i^2+\frac{3}{2}\Big)e^{-\theta_i^2/6}
\lesssim \log n \cdot n^{-c^2/3},
\]
which is of order $o(n^{-1})$ as long as $c>2$.
Therefore,
\begin{equation}\label{eq:Fi.exp.conclusion}
\sup_{\bm\theta\in\Theta_n(s_n,c)}
\E_{\bm Y\mid\bm\theta}\E_{\tau\mid\bm Y}\E_{\tilde{\bm Y}\mid\bm\theta}[\tilde g(Y_i, \tilde Y_i, \theta_i,v)\one{\tau<1}\one{\mathcal W_i}\one{|Y_i|>\zeta_{n,v}}]
= O(1).
\end{equation}

Next we shall consider the complement event $\mathcal W_i^c$.
Note that $v^{1/2}(Y_i + \tilde Y_i/r) \sim N(\theta_i/v^{1/2}, 1)$.
Then since $|\theta_i| > c\sqrt{2\log n}$,
\begin{align*}
\P\{\mathcal W_n^c\}
&=
\P\left(\left|v^{1/2}(Y_i + \tilde Y_i/r)\right|<\sqrt{vw_0}\right)\\
&\leq 2\Phi\left(\sqrt{vw_0}-\frac{|\theta_i|}{\sqrt v}\right)
\leq \exp\left\{-\frac{1}{2}\left(\frac{|\theta_i|}{\sqrt v}-\sqrt{vw_0}\right)^2\right\}
\lesssim n^{-c^2/v+o(1)},
\end{align*}
which is polynomially small with our choice of $c>\sqrt{6}$.
Recall Lemma \ref{lm:g.tilde.decomp} that
\[
\tilde g(Y_i,\tilde Y_i,\theta_i,v)
=
\frac{\tilde Y_i\theta_i}{r}-\frac{\theta_i^2}{2r}-\log N_v(Y_i,\tilde Y_i)+\log D(Y_i).
\]
For $\tau\in(0,1)$, we may bound $D(Y_i)$ from above and $N_v(Y_i,\tilde Y_i)$ from below.
Using the representation \eqref{eq:Nv.D.form.1}, we have $\tau^2+(1-\tau^2)u\geq \tau^2$ for $\tau\in(0,1)$, and hence
\[
\log D(Y_i)
\leq
\log\left(\frac{\tau}{\pi}\cdot \frac{e^{Y_i^2/2}}{\tau^2}\int_0^1 u^{-1/2}\diff u\right)
=
\frac{Y_i^2}{2}+\log\frac{1}{\tau}+\log\frac{2}{\pi}.
\]
Meanwhile, note that $\tau^2/v + (1-\tau^2/v)u \leq 1/v$ for $\tau\in(0,1)$, hence using $e^x\geq 1$,
\[
\log N_v(Y_i,\tilde Y_i)\geq \log\left(\frac{\tau}{\pi\sqrt v}\cdot v \int_0^1 u^{-1/2}\diff u\right)
=
-\log\frac{1}{\tau}+\log\frac{\pi}{2\sqrt v}.
\]
Also note that
\[
\frac{\tilde Y_i\theta_i}{r} - \frac{\theta_i^2}{2r}  = \frac{1}{2r}\left(\tilde Y_i^2 - (\tilde Y_i-\theta_i)^2\right) \leq \frac{\tilde Y_i^2}{2r}.
\]
Hence by Lemma \ref{lm:g.tilde.decomp}, we conclude that for $\tau\in(0,1)$,
\[
\tilde g(Y_i,\tilde Y_i,\theta_i,v) \leq \frac{\tilde Y_i^2}{2r} + \frac{Y_i^2}{2} - \log\sqrt{v} + 2\log\frac{1}{\tau}.
\]
Therefore, by Lemma \ref{lm:tau.not.undershooting}, we get
\begin{align*}
&\sup_{\bm\theta\in\Theta_n(s_n,c)}
\E_{\bm Y\mid\bm\theta}\E_{\tau\mid\bm Y}\E_{\tilde{\bm Y}\mid\bm\theta}
[\tilde g (Y_i,\tilde Y_i,\theta_i,v) \one{\tau<1} \one{\mathcal W_n^c}]\\
&\leq 2\sup_{\bm\theta\in\Theta_n(s_n,c)}\E_{\bm Y\mid\bm\theta} \E\left[\log\frac{1}{\tau}\mid\bm Y\right] \cdot \P\{\mathcal W_n^c\} + 
2-\log\sqrt{v}\\
&\lesssim n^{-c^2/v + o(1)}\log\frac{n}{s_n} + O(1),
\end{align*}
which is of order $O(1)$.
Combining with \eqref{eq:Fi.exp.conclusion}, we arrive at
\begin{equation}\label{eq:adaptive.signal.bound.tau.small}
    \sup_{\bm\theta\in\Theta_n(s_n,c)}
\E_{\bm Y\mid\bm\theta}\E_{\tau\mid\bm Y}\E_{\tilde{\bm Y}\mid\bm\theta}
[\tilde g (Y_i,\tilde Y_i,\theta_i,v) \one{\tau<1} \one{|Y_i|>\zeta_{n,v}}] = O(1).
\end{equation}
Finally, when $\tau\geq 1$, we may have a crude bound on $\tilde g$ as well.
To do so, we can provide an upper bound for $D(Y_i)$ and a lower bound for $N_v(Y_i,\tilde Y_i)$.
By \eqref{eq:Nv.D.form.1},
first note that $\tau^2 + (1-\tau^2)u \geq 1$. Then
\[
\log D(Y_i) \leq \log\left(\frac{\tau}{\pi}e^{Y_i^2/2}\int_0^1 u^{-1/2} \diff u\right) = \frac{Y_i^2}{2}+\log\tau + \log\frac{2}{\pi}.
\]
Meanwhile, note that $\tau^2/v + (1-\tau^2/v)u \leq \tau^2/v$,
\[
\log N_v(Y_i, \tilde Y_i) \geq \log\left(\frac{\sqrt{v}}{\pi\tau}\int_0^1 u^{-1/2}\diff u\right) \geq -\log \tau + \log\frac{2\sqrt{v}}{\pi}.
\]
Also note that
\[
\frac{\tilde Y_i\theta_i}{r} - \frac{\theta_i^2}{2r}  = \frac{1}{2r}\left(\tilde Y_i^2 - (\tilde Y_i-\theta_i)^2\right) \leq \frac{\tilde Y_i^2}{2r},
\]
hence by Lemma \ref{lm:g.tilde.decomp}, we conclude that for $\tau>1$,
\[
\E_{\tilde{\bm Y}\mid\bm\theta}[\tilde g(Y_i,\tilde Y_i,\theta_i,v)] \leq 1 + 2\log\tau + \frac{Y_i^2}{2} - \log\sqrt{v}.
\]
Also note that $\log\tau < \tau$ for $\tau\geq 1$. Thus by Lemma \ref{lm:tau.not.overshooting} and Markov's Inequality, we get
\begin{align*}
&\sup_{\bm\theta\in\Theta_n(s_n,c)}
\E_{\bm Y\mid\bm\theta}\E_{\tau\mid\bm Y}\E_{\tilde{\bm Y}\mid\bm\theta}
[\tilde g (Y_i,\tilde Y_i,\theta_i,v) \one{\tau\ge 1}]\\
&\leq 2\sup_{\bm\theta\in\Theta_n(s_n,c)}\E_{\bm Y\mid\bm\theta} \E[\tau\mid\bm Y] + 
\left(2-\log\sqrt{v}\right)\sup_{\bm\theta\in\Theta_n(s_n,c)}\E_{\bm Y\mid\bm\theta} \Pi\{\tau\geq 1 \mid \bm Y\}\\
&\leq \left(4-\log\sqrt{v}\right) \sup_{\bm\theta\in\Theta_n(s_n,c)}\E_{\bm Y\mid\bm\theta} \E[\tau\mid\bm Y]\\
& \leq \left(4-\log\sqrt{v}\right) K\frac{s_n}{n} (1+o(1)),
\end{align*}
which is of order $o(1)$.
Combining with \eqref{eq:adaptive.signal.bound.tau.small}, we conclude that
\[
\sup_{\bm\theta\in\Theta_n(s_n,c)}
\E_{\bm Y \mid \bm\theta} \E_{\tau \mid \bm Y} \E_{\tilde {\bm Y} \mid \bm\theta}
[\tilde g(Y_i,\tilde Y_i,\theta_i,v)\one{|Y_i|>\zeta_{n,v}}]
=O(1).
\]

\subsubsection{Small observation: $|Y_i| \leq \sqrt{2v\log(n/s_n)}$}
In this case, we use the trivial bound $\log N_v(Y_i, \tilde Y_i) \leq 0$. Thus
\[
\tilde g(Y_i, \tilde Y_i, \theta_i, v) \leq \frac{\theta_i^2}{2r} + \log D(Y_i).
\]
The first term can be upper bounded using the inequality $(a+b)^2 \leq 2a^2 + 2b^2$ and $v/r = 1-v$,
\[
\frac{\theta_i^2}{2r} \leq \frac{(Y_i-\theta_i)^2 + Y_i^2}{r} \leq \frac{(Y_i-\theta_i)^2}{r} + 2(1-v) \log\frac{n}{s_n}.
\]
For the $\log D(Y_i)$ term, following the same proof in Appendix \ref{sec:known.tau.signal.small.obs}, we split into two cases $0<r\leq 1$ and $r>1$ again.

\noindent{\bf The case when $0<r\leq 1$}

In this case, $0 < v \leq 1/2$, and $v\leq 1-v$.
Using the expression of $D(Y_i)$ in \eqref{eq:Nv.D.form.1},
\begin{align*}
    D(Y_i) &\leq \frac{\tau}{\pi} e^{Y_i^2/2} \int_0^1 u^{-1/2} \frac{1}{\tau^2 + (1-\tau^2)u} \diff u \leq e^{Y_i^2/2}.
\end{align*}
Since $|Y_i| \leq \zeta_{n,v}$, we have $e^{Y_i^2/2} \leq (n/s_n)^v \leq (n/s_n)^{1-v}$.
Therefore,
\[
\log D(Y_i) \leq (1-v)\log\frac{n}{s_n} \quad\text{ when } r\in(0,1].
\]
We may conclude that as $r\in(0,1]$,
\[
\tilde g(Y_i, \tilde Y_i, \theta_i, v) \leq \frac{(Y_i - \theta_i)^2}{r} + 3(1-v) \log\frac{n}{s_n},
\]
and therefore
\begin{multline*}
\E_{\bm Y \mid \bm\theta} \E_{\tau\mid\bm Y} \E_{\tilde{\bm Y} \mid \bm\theta}[\tilde g(Y_i, \tilde Y_i, \theta_i, v) \one{|Y_i|\leq \zeta_{n,v}}]\\
\leq \E_{\bm Y \mid \bm\theta}\left[\frac{(Y_i - \theta_i)^2}{r}\right] + 3(1-v)\log\frac{n}{s_n}\cdot \P\{|Y_i| \leq \zeta_{n,v}\}.
\end{multline*}
Here, the first term on the right hand side is $1/r$.
For the second term, note that under Assumption \ref{assumption:beta.min}, using Mill's ratio bound,
\[
\P\{|Y_i| \leq \zeta_{n,v}\} \leq 2\Phi\left(-(c-\sqrt{v})\sqrt{2\log n}\right) \leq \frac{1}{(c-\sqrt{v})\sqrt{\pi \log n}} n^{-(c-\sqrt{v})^2}.
\]
Therefore, for any choice of $c>\sqrt{v}$ and $r\in (0,1]$,
\[
\sup_{\bm\theta\in\Theta_n(s_n,c)}\E_{\bm Y \mid \bm\theta} \E_{\tau\mid\bm Y} \E_{\tilde{\bm Y} \mid \bm\theta}[\tilde g(Y_i, \tilde Y_i, \theta_i, v) \one{|Y_i|\leq \zeta_{n,v}}] \leq \frac{1}{r} + o(1).
\]

\noindent{\bf The case when $r > 1$}

In this case, we split the integral in the expression of $D(Y_i)$ in \eqref{eq:Nv.D.form.1} by $\frac{1-v}{v}<1$.
For the first part, recall the bound $\frac{\tau}{\pi} \int_0^1 u^{-1/2} \frac{1}{\tau^2+(1-\tau^2)u} \diff u \leq 1$, we have
\[
\frac{\tau}{\pi}\int_0^{\frac{1-v}{v}} u^{-1/2} \frac{1}{\tau^2+(1-\tau^2)u} e^{\frac{Y_i^2}{2}u} \diff u \leq e^{\frac{Y_i^2}{2}\frac{1-v}{v}} \leq (n/s_n)^{1-v}.
\]
For the second part, note that $\tau^2 + (1-\tau^2)u \geq u$ for all $\tau>0$ and $0<u<1$. Thus
\begin{multline*}
\frac{\tau}{\pi}\int_{\frac{1-v}{v}}^1 u^{-1/2} \frac{1}{\tau^2+(1-\tau^2)u} e^{\frac{Y_i^2}{2}u} \diff u\\
\leq \frac{2\tau}{\pi}e^{Y_i^2/2}\left(\sqrt{\frac{v}{1-v}} - 1\right) \leq \frac{2}{\pi} \left(\sqrt{\frac{v}{1-v}} - 1\right) \left(\frac{n}{s_n}\right)^v \tau.
\end{multline*}
This is, again, different from the fixed-$\tau$ scenario where $(n/s_n)^v\tau$ converges to zero.
Using $\log(1+x) \leq x$ for $x\geq 0$, we arrive at
\[
\log D(Y_i) \leq \frac{2}{\pi}\left(\sqrt{\frac{v}{1-v}} -1\right) \left(\frac{n}{s_n}\right)^{2v-1} \tau, \quad \text{when }r\in(1,\infty).
\]
By Lemma \ref{lm:tau.not.overshooting}, we have for $r>1$,
\[
\sup_{\bm\theta\in\Theta_n(s_n,c)}\E_{\bm Y \mid \bm\theta} \E_{\tau\mid\bm Y} \E_{\tilde{\bm Y} \mid \bm\theta}[\log D(Y_i)] \leq \frac{2K}{\pi}\left(\sqrt{\frac{v}{1-v}} -1\right) \left(\frac{n}{s_n}\right)^{2v-2}.
\]
Note that this bound is of order $o(1)$, since $v<1$.
For the remaining terms in $\tilde g(Y_i, \tilde Y_i, \theta_i, v)$, using the same method as the $r\in(0,1]$ case, we have for any choice of $c>\sqrt{v}$ and $r>1$,
\[
\sup_{\bm\theta\in\Theta_n(s_n,c)}\E_{\bm Y \mid \bm\theta} \E_{\tau\mid\bm Y} \E_{\tilde{\bm Y} \mid \bm\theta}[\tilde g(Y_i, \tilde Y_i, \theta_i, v) \one{|Y_i|\leq \zeta_{n,v}}] \leq \frac{1}{r} + o(1).
\]
Combining the $0<r\leq 1$ case and the $r>1$ case, we conclude that this same bound holds for any $r>0$ and $c>\sqrt{v}$.

As a conclusion to the signal case, by combining the two cases for the scale of $|Y_i|$, we arrive at
\begin{equation}\label{eq:adaptive.signal.case}
\sup_{\bm\theta \in \Theta_n(s_n,c)} \E_{\bm Y\mid \bm\theta} \E_{\tau \mid \bm Y} [L(\theta_i, \hat p(\cdot \mid Y_i, \tau))] = O(1).
\end{equation}
Eventually, all the signal terms in the predictive risk contributes the order of $O(s_n)$, which is slightly smaller than the minimax rate.

\subsection{Noise Case: $\theta_i = 0$}

\subsubsection{Large observation: $|Y_i|>\sqrt{2\log (n/s_n)}$}

Fix an index $i$ such that $\theta_i=0$, so that $Y_i\sim N(0,1)$.
Recall $\zeta_{n,1}=\sqrt{2\log(n/s_n)}$.
We bound
\[
\E_{\bm Y\mid\bm\theta}\E_{\tau\mid\bm Y} [L(0,\hat p(\cdot\mid Y_i,\tau))\one{|Y_i|>\zeta_{n,1}}].
\]

By the spectroscopy result \eqref{eq:loss.spectroscopy}, for each fixed $\tau>0$ and $y\in\R$,
\[
L(0,\hat p(\cdot\mid y,\tau))
\leq
\int_0^\infty L(0,\hat p_{\lambda_i}(\cdot\mid y,\tau))\,\pi(\lambda_i\mid y,\tau) \diff \lambda_i,
\]
where $\hat p_\lambda(\cdot\mid y,\tau)$ is the predictive density under the Gaussian prior with fixed local scale $\lambda$.

For fixed $\tau$ and $\lambda$, recall that the form of the predictive density $\hat p_{\lambda_i}(\tilde Y_i \mid Y_i)$ given in \eqref{eq:predictive.density.fixed.lambda} and its corresponding tight upper bound in KL loss, that is, with $\theta_i=0$,
\[
L(0, \hat p_{\lambda_i}(\cdot \mid Y_i,\tau)) \leq \frac{1-v}{2}\frac{1+\lambda_i^2\tau^2}{v+\lambda_i^2\tau^2} \left(\frac{\lambda_i^2\tau^2}{1+\lambda_i^2\tau^2}Y_i\right)^2 \leq \frac{1-v}{2}Y_i^2.
\]
Since this bound does not directly involve $\lambda_i$, for any $\tau>0$ and any observation $Y_i \in \R$, \eqref{eq:loss.spectroscopy} implies
\[
L(0,\hat p(\cdot\mid y,\tau))
\leq
\frac{1-v}{2}Y_i^2.
\]
Therefore
\[
\sup_{\bm\theta\in\Theta_n(s_n)}\E_{\bm Y\mid\bm\theta}\E_{\tau\mid\bm Y}[L(0,\hat p(\cdot\mid Y_i,\tau))\one{|Y_i|>\zeta_{n,1}}]
\leq
\frac{1-v}{2}\E[Z^2\one{|Z|>\zeta_{n,1}}],
\]
where $Z\sim N(0,1)$.
Using integration by parts,
\begin{align*}
\E[Z^2\one{|Z|>\zeta_{n,1}}]
&=2\int_{\zeta_{n,1}}^\infty z^2\phi(z)\diff z\\
&=2\zeta_{n,1}\phi(\zeta_{n,1})+2\Phi(-\zeta_{n,1})\\
&\leq 2\phi(\zeta_{n,1})\left(\zeta_{n,1} + \frac{1}{\zeta_{n,1}}\right) \\
&= \frac{2}{\sqrt{\pi}}\frac{s_n}{n}\sqrt{\log \frac{n}{s_n}}.
\end{align*}
Therefore,
\begin{equation}\label{eq:noise.large.obs.final}
\sup_{\bm\theta\in\Theta_n(s_n)}
\E_{\bm Y\mid\bm\theta}\E_{\tau\mid\bm Y}[L(0,\hat p(\cdot\mid Y_i,\tau))\one{|Y_i|>\zeta_{n,1}}]
\leq
\frac{1}{\sqrt{\pi}}(1-v)\frac{s_n}{n}\sqrt{\log\frac{n}{s_n}}.
\end{equation}

\subsubsection{Small observation: $|Y_i| \leq \sqrt{2\log(n/s_n)}$}
In this case, we apply the spectroscopy result in Lemma \ref{lm:risk.spectroscopy.adaptive}.
By \eqref{eq:risk.spectroscopy.adaptive}, it suffices to consider the upper bound of
\[
\frac{1-v}{6}\frac{\Phi_1(1,1,\frac{5}{2}, -\frac{Y_i^2}{2}, 1-\tau^2)}{\Phi_1(1,1,\frac{3}{2}, -\frac{Y_i^2}{2}, 1-\tau^2)}Y_i^2,
\]
which is equivalent to
\[
\frac{1-v}{2} \frac{\int_0^1 u^{1/2} \frac{1}{\tau^2 + (1-\tau^2)u} e^{\frac{Y_i^2}{2}} \diff u}{\int_0^1 u^{-1/2} \frac{1}{\tau^2 + (1-\tau^2)u} e^{\frac{Y_i^2}{2}} \diff u} Y_i^2.
\]
Following the same proof technique as Appendix \ref{sec:known.tau.noise.small.obs}, for any choice of constant $1<a<\tau^{-2}$, this quantity is upper bounded by
\[
\frac{1-v}{3} \tau^2 Y_i^2 e^{\frac{Y_i^2}{2}\tau^2} + (1-v) \tau Y_i^2 e^{\frac{Y_i^2}{2}\frac{1}{a}}\left(\frac{1}{\sqrt{a}} - \tau\right) + (1-v)\sqrt{a}\tau e^{\frac{Y_i^2}{2}}.
\]
We focus on the third term in the following analysis, and the first two terms can be bounded with a similar method.
Recall that $i\notin\mathcal S$ so that $\theta_i=0$.  Conditioning on $Y_i=y$ and using
$e^{y^2/2}\phi(y) =  (2\pi)^{-1/2}$, by Corollary \ref{cor:loo.overshoot.fixedy},
\begin{align*}
\E_{\bm Y\mid\bm\theta}\E_{\tau\mid\bm Y}\!\left[\tau e^{Y_i^2/2}\one{|Y_i|\leq \zeta_{n,1}}\right]
&=
\E_{\bm Y_{-i}\mid\bm\theta} \left[\int_{-\zeta_{n,1}}^{\zeta_{n,1}}
e^{y^2/2}\phi(y) \E(\tau\mid \bm Y_{-i},y)\diff y\right]\\
&=
\frac{1}{\sqrt{2\pi}}\int_{-\zeta_{n,1}}^{\zeta_{n,1}}
\E_{\bm Y_{-i}\mid\bm\theta} \left[\E(\tau\mid \bm Y_{-i},y)\right]\diff y\\
&\leq
\frac{2}{\sqrt{2\pi}}\zeta_{n,1} K\tau_{n,0}\\
&= \frac{2}{\sqrt{\pi}}\frac{s_n}{n} K \sqrt{\log\frac{n}{s_n}}.
\end{align*}
Therefore
\[
\sup_{\bm\theta\in\Theta_n(s_n,c)}
\E_{\bm Y\mid\bm\theta}\E_{\tau\mid\bm Y}\left[ (1-v)\sqrt{a} \tau e^{Y_i^2/2}\one{|Y_i|\leq \zeta_{n,1}}\right]
\leq
\frac{2K}{\sqrt{\pi}}\sqrt{a}(1-v)\frac{s_n}{n}\sqrt{\log\frac{n}{s_n}}.
\]
This gives the desired per-coordinate bound for the dominant term in the noise--small-observation analysis.

Also by Corollary \ref{cor:loo.overshoot.fixedy}, using the same proof technique, we can show that
\begin{multline*}
\sup_{\bm\theta\in \Theta_n(s_n,c)}\E_{\bm Y\mid \bm\theta} \E_{\tau\mid \bm Y}\left[(1-v) \tau Y_i^2 e^{\frac{Y_i^2}{2}\frac{1}{a}}\left(\frac{1}{\sqrt{a}} - \tau\right) \one{|Y_i| \leq \zeta_{n,1}}\right]\\
\lesssim (1-v) a^{-1/2}(1-a^{-1})^{-3/2}\frac{s_n}{n},
\end{multline*}
and using additionally $e^{\frac{Y_i^2}{2}\tau^2} \leq e^{\frac{Y_i^2}{2}}$, we have
\begin{equation*}
\sup_{\bm\theta\in \Theta_n(s_n,c)}\E_{\bm Y\mid \bm\theta} \E_{\tau\mid \bm Y}\left[\frac{1-v}{3} \tau^2 Y_i^2 e^{\frac{Y_i^2}{2}\tau^2} \one{|Y_i| \leq \zeta_{n,1}}\right]
\lesssim (1-v)\frac{s_n^2}{n^2}\log^{3/2}\frac{n}{s_n}.
\end{equation*}
Adding up the three terms, with the choice of $a$ as a constant close to one, we conclude that
\begin{equation}\label{eq:noise.small.obs.final}
\sup_{\bm\theta \in \Theta_n(s_n,c)}\E_{\bm Y \mid \bm \theta} \E_{\tau \mid \bm Y} [L(0, \hat p(\cdot \mid Y_i, \tau)) \one{|Y_i| \leq \zeta_{n,1}}]
\leq
\frac{2K}{\sqrt{\pi}} (1-v)\frac{s_n}{n}\sqrt{\log\frac{n}{s_n}} (1+ o(1)).
\end{equation}

By combining \eqref{eq:noise.large.obs.final} and \eqref{eq:noise.small.obs.final}, we arrive at a conclusion to the noise case,
\begin{equation}\label{eq:adaptive.noise.case}
    \sup_{\bm\theta \in \Theta_n(s_n,c)}\E_{\bm Y \mid \bm \theta} \E_{\tau \mid \bm Y} [L(0, \hat p(\cdot \mid Y_i, \tau)) ]
    \leq
    \frac{2}{\sqrt{\pi}}(K+1) (1-v)\frac{s_n}{n}\sqrt{\log\frac{n}{s_n}} (1+o(1)).
\end{equation}

\subsection{Conclusion}

When we impose a hyperprior $\pi(\tau)$ on random hyperparameter $\tau$, by Jensen's inequality, the Kullback-Leibler loss under the hierarchical prior is upper bounded by a scale mixture of losses under fixed $\tau$,
\begin{equation}
    L(\bm \theta, \hat p(\cdot \mid \bm y)) \leq \int L(\bm \theta, \hat p(\cdot \mid \bm y, \tau)) \pi(\tau\mid \bm y) \diff \tau.
\end{equation}
This leads to the following upper bound on the risk,
\begin{multline}\label{eq:risk.decomp.adaptive}
    \rho_n (\bm \theta, \hat p) \leq \E_{\bm Y \mid \bm \theta} \E_{\tau \mid \bm Y} [L(\bm \theta, \hat p(\cdot \mid \bm Y, \tau))]\\
    = \sum_{i\in\mathcal S} \E_{\bm Y \mid \bm \theta} \E_{\tau \mid \bm Y} [L(\theta_i, \hat p (\cdot \mid Y_i, \tau))] + (n-s_n) \E_{\bm Y \mid \bm \theta} \E_{\tau \mid \bm Y} [L(0, \hat p(\cdot \mid Y_i,\tau))].
\end{multline}
By combining \eqref{eq:adaptive.signal.case} with \eqref{eq:adaptive.noise.case} and plugging into \eqref{eq:risk.decomp.adaptive}, we conclude that for any $c>\sqrt{6}$,
\begin{multline*}
\sup_{\bm\theta\in\Theta_n(s_n,c)}\rho_n(\bm\theta,\hat p)
\leq s_n\cdot O(1) + \frac{n-s_n}{n} \frac{2}{\sqrt{\pi}}(K+1) (1-v) s_n \sqrt{\log\frac{n}{s_n}} (1+ o(1))\\
= \frac{2}{\sqrt{\pi}}(K+1) (1-v) s_n\sqrt{\log\frac{n}{s_n}} (1+ o(1)).
\end{multline*}
This finishes the proof of Theorem \ref{thm:adaptive.beta.min}.

\section{Additional Simulations}

\subsection{KL Risk under Fixed Global Shrinkage Parameter}\label{sec:max.KL}

We start from the oracle case when the sparsity level $s_n$ is known.
In this case, both the spike-and-slab family and the Horseshoe prior of  $\bm\theta$ are separable, either given a fixed proportion of slab, $\eta$, or a fixed global shrinkage parameter, $\tau$.
The maximum predictive risk can thus be written as
\begin{equation}\label{eq:max.risk.fixed}
\sup_{\bm\theta \in \Theta_n(s_n)}\rho_n(\bm\theta, \hat p) = (n-s_n)\rho(0,\hat p) + s_n \sup_{\theta\in \R} \rho(\theta, \hat p).
\end{equation}
A good shrinkage prior should find a balance between predictive risks contributed by signal and noise, reaching an optimal multivariate predictive risk.
The following quantitative experiment compares this maximum KL risk with the minimax risk in \eqref{eq:minimax.rate}.

\begin{figure}[h]
    \centering
    \includegraphics[width=12cm]{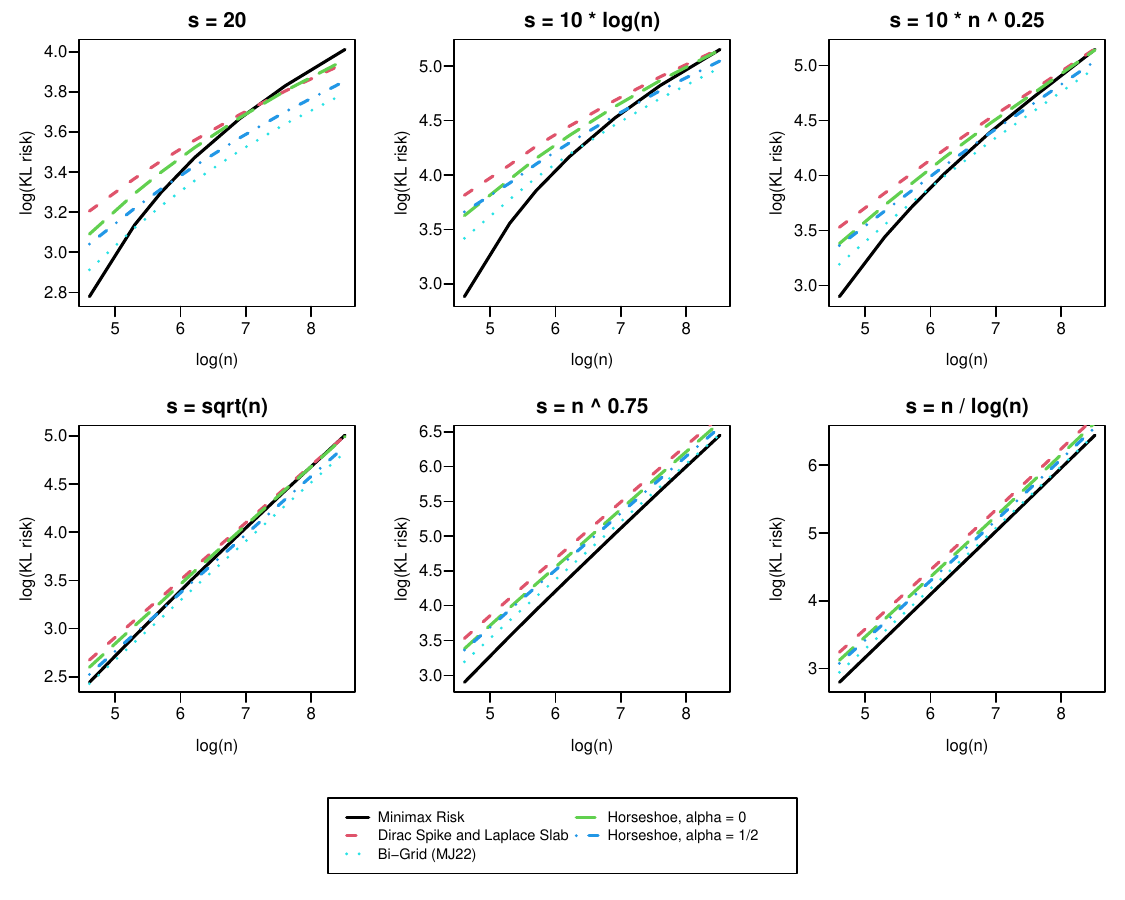}
    \caption{A comparison of maximum risk of various methods to the minimax KL risk under increasing $n$. Six schemes are considered, with the true underlying sparsity level $s_n$ increasing with $n$ at different rates.}
    \label{fig:KL.fixed.tau}
\end{figure}

We consider a series of increasing $n$, and series of $s_n$ that depends on $n$ in six different schemes, from slow to fast:
(1) $s_n$ fixed to $20$,
(2) $s_n = 10\log n$,
(3) $s_n = 10n^{1/4}$,
(4) $s_n = n^{1/2}$,
(5) $s_n = n^{3/4}$,
and (6) $s_n = n/\log n$,
Except for the first one, all other schemes satisfy the assumption that $s_n \rightarrow \infty$, and all six schemes satisfy the nearly-black assumption of $s_n/n\rightarrow 0$.

We consider the following two calibrations of the Horseshoe prior for each scheme, with $\alpha = 0$, i.e. $\tau = s_n/n$,
and with $\alpha = 1/2$, i.e. $\tau = s_n \sqrt{\log(n/s_n)} / n$.
Under the setting where $s_n$ is known, we compute their respective univariate risk curves, find the value at zero and the maximum, then add them up for the maximum predictive risk by \eqref{eq:max.risk.fixed}.
The results are shown in Figure \ref{fig:KL.fixed.tau}, where we plot the logarithm of $\sup_{\bm\theta \in \Theta_n(s_n)} \rho_n(\bm\theta, \hat p)$ versus the logarithm of $n$.
As a comparison, the minimax risk in \eqref{eq:minimax.rate} is also plotted,
as well as the maximum predictive risk achieved by Dirac spike-and-slab prior \citep{rockova2023adaptive}, and the bi-grid prior by \cite{mukherjee2022minimax}.
We see that the maximum risk curves of
the Horseshoe prior are relatively close to the minimax risk, especially for larger $n$.
This phenomenon persists for all six growth scheme of $s_n$ with respect to $n$, indicating that the Horseshoe prior with fixed $\tau$ is capable of balancing the signal and the noise.
Despite being a continuous mixture of Gaussian, the Horseshoe prior features a spike near zero and a long-tailed slab that imitates the performance of a spike-and-slab prior.

\begin{figure}[t]
    \centering
    \includegraphics[width=0.3\linewidth]{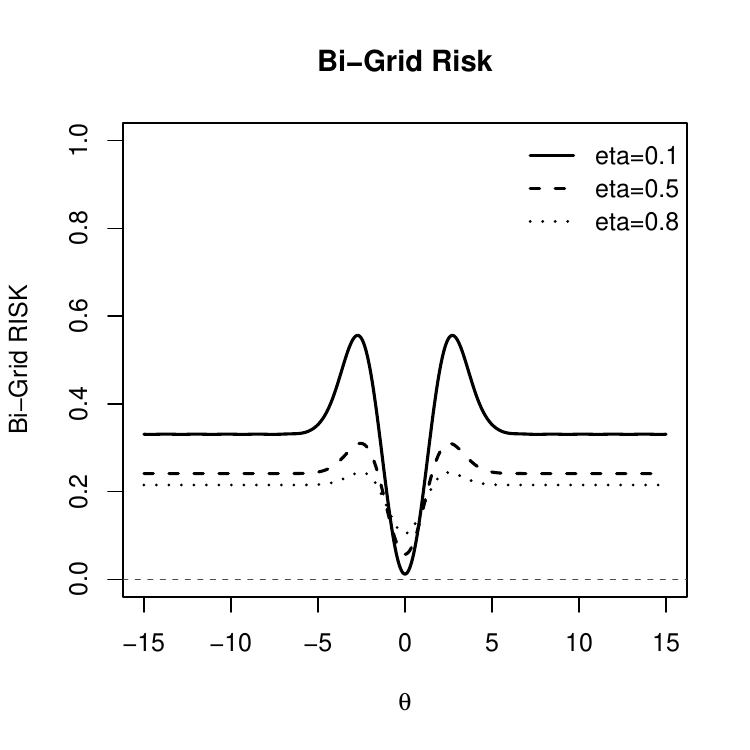}
    \includegraphics[width=0.3\linewidth]{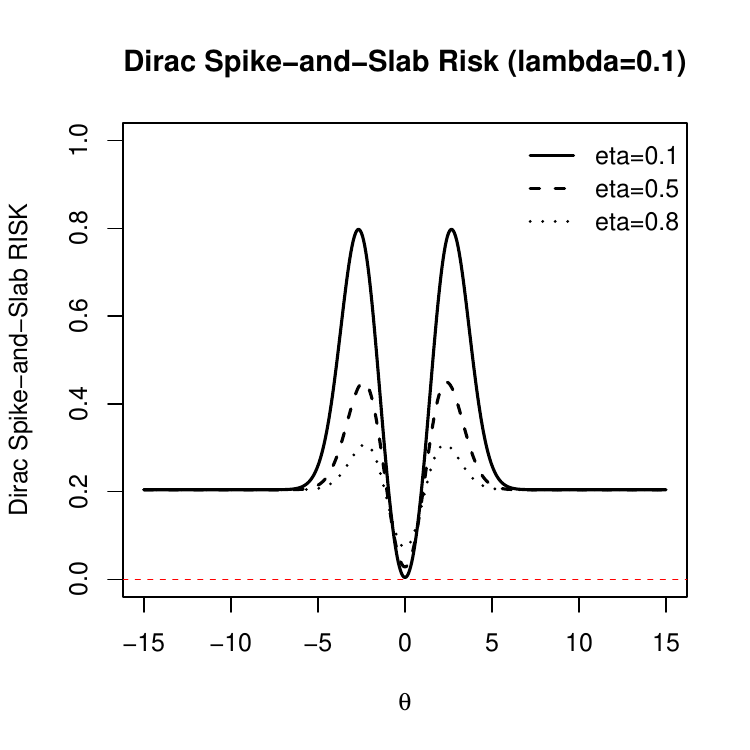}
    \includegraphics[width=0.3\linewidth]{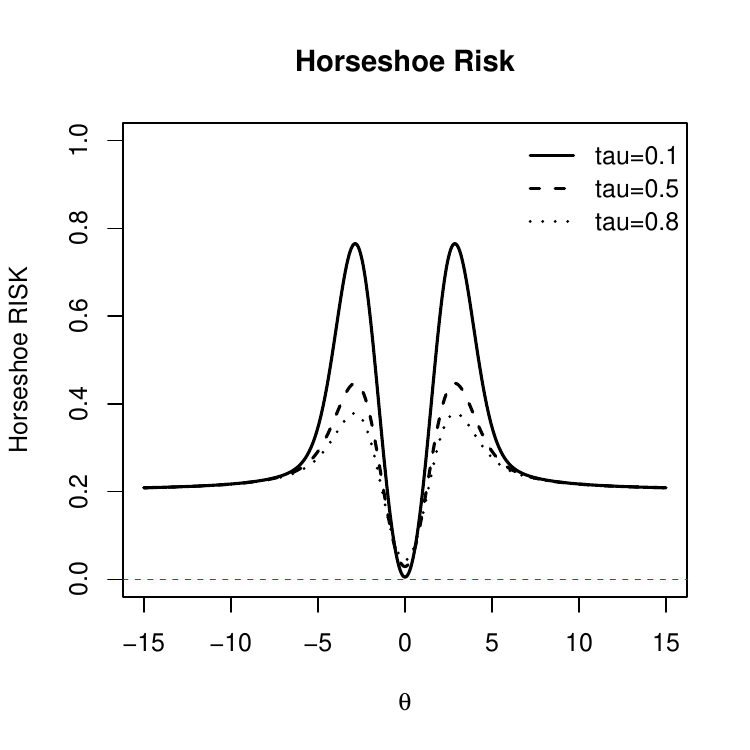}
    \caption{Univariate risk plots for the Bi-Grid, the Dirac spike-and-slab, and the Horseshoe priors.}
    \label{fig:uni.risk.plots}
\end{figure}

We also plot the univariate KL risk of the three priors with respect to $\theta$ in Figure \ref{fig:uni.risk.plots}.
While tuning the main hyperparameter ($\eta$ for the spike-and-slab family and $\tau$ for the Horseshoe), all three priors exhibit a tradeoff between the risk at zero and the maximum risk, a sign of balance between signal and noise.
The bi-grid prior \citep{mukherjee2022minimax} attains a maximum univariate risk that is lower than the other two priors, which explains our observation in Figure \ref{fig:KL.fixed.tau} that it generally attains the lowest maximum multivariate KL risk.
Beyond this worst-case scenario, the Dirac spike-and-slab prior \citep{rockova2023adaptive} and the Horseshoe prior generally attain a lower univariate risk for large $|\theta|$.
While all priors have theoretical guarantees for the maximum predictive KL risk, practitioners may still want to select the prior according to the problem setup.
When a large number of strong signals are present, and we are more concerned about the case-specific risk instead of the worst-case scenario risk, the Dirac spike-and-slab and the Horseshoe may achieve lower risks.
See, for instance, the simulation results in Tables \ref{tab:KL.risk.setup.B} and \ref{tab:KL.risk.setup.A}.

\subsection{Details on Predictive Risk Simulation}\label{sec:risk.sim.detail}

When the hyperparameters like $\tau$ for the Horseshoe and $\eta$ for the spike-and-slab priors are fixed, the predictive risk can be directly computed using the risk decomposition results, e.g., Lemma \ref{lm:risk.decomp}.
But for the full-Bayes approach, the multivariate prior $\pi(\bm\theta)$ and the predictive density $\hat p(\tilde{\bm y} \mid \bm y)$ are no longer separable.
Instead, we take the following Algorithm \ref{algo:random.tau} to evaluate the predictive risk under random $\tau$ for the Horseshoe prior.
For other priors, we may evaluate the risk with a similar approach.

\begin{algorithm}
\caption{Estimation of Adaptive Predictive Risk $\rho_n(\bm\theta, \hat p_\pi)$}
\label{algo:random.tau}
\begin{algorithmic}[1]
    \State Sample $\bm y^{(1)}, \cdots, \bm y^{(B)} \sim N(\bm\theta, I_n)$ and $\tilde{\bm y}^{(1)}, \cdots, \tilde{\bm y}^{(B)} \sim N(\bm\theta, rI_n)$.
    \For{$1 \leq b \leq B$} (Estimate $\hat p_\pi(\bm{\tilde y}^{(b)} \mid \bm y^{(b)})$)
        \State Sample $\tau^{(1)}, \cdots, \tau^{(Q)}$ from the posterior $\pi(\tau\mid \bm y^{(b)})$.
        \For{$1 \leq q \leq Q$ and $1 \leq i \leq n$} (Estimate $\E_{\lambda_i \mid y_i^{(b)},\tau^{(q)}} \hat p_{\lambda_i}(\tilde y_i^{(b)} \mid y_i^{(b)}, \tau^{(q)})$)
            \State  Sample $\lambda_i^{(1)},\cdots, \lambda_i^{(L)}$ from the posterior $\pi(\lambda_i \mid y_i^{(b)},\tau^{(q)})$.
            \State  Compute average value of $\hat p_{\lambda_i^{(l)}}(\tilde y_i^{(b)} \mid y_i^{(b)}, \tau^{(q)})$ over $1\leq l \leq L$.
        \EndFor
        \State Compute average value of $\hat p_\pi(\tilde{\bm y}^{(b)} \mid \bm y^{(b)}, \tau^{(q)})$ over $1\leq q\leq Q$.
    \EndFor
    \State Compute average value of $\log\{\pi(\tilde{\bm y}^{(b)}\mid \bm\theta) / \hat p_\pi(\tilde{\bm y}^{(b)} \mid \bm y^{(b)}) \}$ over $1\leq b\leq B$.
\end{algorithmic}
\end{algorithm}

In this experiment, our choice of the technical parameters is $B=1000$, $Q=200$, and $L=300$.

For the Dirac spike-and-slab prior \citep{rockova2023adaptive}, specifically, a tuning parameter is $\lambda$, the rate of the Laplace slab.
In each of our fixed-hyperparameter simulation settings, we consider a grid $\lambda \in \{0.05, 0.1, 0.2, 0.5, 1, 1.5, 2, 2.5, 3\}$.
For each fixed-$\eta$ simulation setup for the DSnS prior, we compute a predictive KL risk for each $\lambda$ value, and choose the minimal risk.
In the full-Bayes approach, we consider the same grid of $\lambda$.
For each simulation setup, we generate 1,200 samples, $\{\bm Y^{(i)}\}_{i=1}^{N}$ and $\{\tilde {\bm Y}^{(i)}\}_{i=1}^{N}$, use $N_H=200$ of them as a held-out set, select a value of $\lambda$ by
\[
\hat\lambda = \argmax_\lambda \frac{1}{N_H} \sum_{i=1}^{N_H} \log \hat p_{\lambda}(\tilde {\bm Y}^{(i)} \mid \bm Y^{(i)}),
\]
and then use the value $\hat\lambda$ on the remaining 1,000 samples to evaluate the predictive KL risk under the Dirac spike-and-slab prior.

\subsection{Additional Predictive Risk Simulation}\label{sec:risk.sim.setup.A}

In the main text, we have explored a simulation setup where we have strong signals, weak signals and pure noise.
We provide a supplementary experiment where we only have either strong signals or pure noise.
For a fixed $n$, assign multiple sparsity levels $s_n$.
For each sparsity level, we let $s_n$ many entries of $\bm\theta$ be $c\sqrt{2\log n}$ and the rest zero.
Here, the choice of factor $c$ is still 2, 3, or 4.
We still let $n=500$, and $s_n \in \{25, 50, 100\}$.

\begin{table}[t]
\centering
\caption{Kullback-Leibler risks for parameter vector $\bm\theta$ under the alternative setup with no weak signal.
When adaptive to $s_n$, the risks are computed on 1000 data vectors $\bm Y$ and $\tilde{\bm Y}$ generated from $\bm\theta$.
}
\footnotesize
\label{tab:KL.risk.setup.A}
\renewcommand{\arraystretch}{1.15}
\begin{tabular}{l *{9}{r}}
\toprule
 &
 \multicolumn{3}{c}{$s_n=25$} &
 \multicolumn{3}{c}{$s_n=50$} &
 \multicolumn{3}{c}{$s_n=100$} \\
\cmidrule(lr){2-4}\cmidrule(lr){5-7}\cmidrule(lr){8-10}
Signal Level $c$ & 2 & 3 & 4 & 2 & 3 & 4 & 2 & 3 & 4 \\
\midrule
\multicolumn{10}{c}{\textit{As if $s_n$ were known}}\\

Bi-Grid
& 18.16 & 18.14 & 18.13
& 33.19 & 33.18 & 33.15
& 59.48 & 59.47 & 59.47\\

DSnS
& 9.72 & 9.70 & 9.70
& 19.20 & 19.18 & 19.18
& 37.76 & 37.74 & 37.74 \\
HS, $\alpha=1/2$  
& 13.82 & 13.06 & 12.84
& 26.89 & 25.39 & 24.95
& 50.62 & 47.66 & 46.78 \\
HS, $\alpha=0$  
& 11.96 & 11.20 & 10.98
& 24.20 & 22.69 & 22.24
& 48.14 & 45.16 & 44.28 \\
\addlinespace[2pt]
\multicolumn{10}{c}{\textit{$s_n$ unknown; calibration fixed as if $s_n=1$}}\\

Bi-Grid
& 21.70 & 17.52 & 21.56
& 43.13 & 34.77 & 42.85
& 86.00 & 69.27 & 85.44\\

DSnS
& 8.93 & 8.73 & 8.73
& 17.80 & 17.41 & 17.41
& 35.55 & 34.76 & 34.76 \\
HS, $\alpha=1/2$  
& 10.15 & 9.29 & 9.07
& 20.16 & 18.43 & 17.99
& 40.16 & 36.71 & 35.83 \\
HS, $\alpha=0$  
& 10.06 & 9.21 & 8.99
& 20.05 & 18.36 & 17.91
& 40.03 & 36.64 & 35.76 \\

\addlinespace[2pt]
\multicolumn{10}{c}{\textit{Adapts to $s_n$ --- full-Bayes approach}}\\

DSnS-Beta
& 9.49 & 9.31 & 9.27
& 18.66 & 18.51 & 18.41
& 36.79 & 36.34 & 37.04 \\
HS-Exp
& 11.37 & 10.52 & 10.58
& 23.50 & 21.74 & 21.14
& 46.76 & 43.50 & 42.46 \\

\bottomrule
\end{tabular}
\end{table}

The experiment results under this setup is given in Table \ref{tab:KL.risk.setup.A}.
This setup represents the best case scenario when the signal and the noise are clearly and perfectly separated.
In this case, the Dirac spike-and-slab prior leads to the optimal predictive risk under multiple settings.
The structure of $\bm\theta$ features multiple entries at exactly zero, and some more entries further away.
A separate Dirac spike is more suitable for this setup compared to the Horseshoe's continuous spike near zero.

A comparison between the oracle and adaptive cases in Table \ref{tab:KL.risk.setup.A} shows that the KL risks are not exceedingly sensitive to the choice of fixed hyperparameters (e.g. $\eta$ for DSnS and $\tau$ for HS), especially when the signal levels are high.
For instance, the choice of $\tau = 1/n$ for the Horseshoe, which is vastly underestimated, is not considerably detrimental to the predictive KL risk compared to the theoretically optimal choice of $\tau = s_n/n$.

Regarding the adaptiveness of the full-Bayes hierarchical models, both DSnS-Beta and HS-Exp are capable of recovering predictive KL risks similar to those in the oracle cases where $s_n$ are known.
These priors are also theoretically proven to be able to recover unknown sparsity level.
Note that this setup satisfies the theta-min condition (Assumption \ref{assumption:beta.min}), and therefore the theoretical guarantee of rate minimaxity should hold.

\section{Supplementary Real Data Experiments}

\subsection{JAFFE Details and Extra Experiments}\label{sec:apdx.JAFFE}

\subsubsection{Daubechies-4 wavelet decomposition}\label{sec:daubechies}

Each image is represented as a two-dimensional array $X = \{x_{u,v}\}_{1\leq u,v\leq m}$, where $m$ denotes the spatial resolution.
In our example, $m=256=2^8$.
We can decompose $X$ via discrete wavelet transform.
Denote our choice of wavelet basis by $\{\psi_{jk}\}$, where $0\leq j\leq 8$ is the level of frequency resolution (the smaller the coarser), and $k = (k_1,k_2), 1\leq k_1,k_2 < 2^j$ is the position index for each level.
Specifically, we take Daubechies-4 wavelet basis throughout this data experiment.
We can then obtain wavelet coefficients $\{y_{jk}\}$ by projecting the 2-dimensional image signal onto the basis functions,
\[
y_{jk} = \sum_{u=1}^m \sum_{v=1}^m x_{u,v} \psi_{jk}(u,v).
\]
Since discrete wavelet transform corresponds to an orthogonal matrix operator, it can be seen as an orthogonal transformation from the pixel domain to the wavelet domain.
If we assume i.i.d.~Gaussian noise in the pixel domain, this property can be inherited by the wavelet coefficients.
Moreover, natural images like human faces are usually characterized by piecewise smooth regions separated by edges.
Therefore, the wavelet coefficient can be considered as a Gaussian random vector around a sparse mean.
Here, the mean is zero for the smooth regions on the image, and is nonzero over the edges.
We expect more zero means in higher frequency resolution levels (larger $j$).

For the task of facial recognition, specifically, we are more concerned about the coarser levels of the wavelet coefficients.
They tend to encode the invariant biological properties of the subject, like the bone structure, distance between eyes, and the contours of the nose and jawline.
These features remain relatively stable despite the different facial expressions.
Meanwhile, changes in facial expression often involve local elastic deformation in soft tissue, which is usually filtered out.
Therefore, we may test whether the wavelet coefficients of two images share the same underlying mean vector, and therefore effectively test for identity.

In our example, we take $0\leq j\leq J$, with $J=3$.
By selecting these coarser levels, the high-frequency noise (e.g., skin texture, wrinkles caused by facial expressions) is filtered out, and the nonzero signals are prominent.
Our choice of Daubechies-4 wavelet is also efficient at compressing structural information, like major facial features, into a few high-magnitude coefficients at the coarser levels.
Therefore, the technical condition on the minimum signal strength is more likely to hold under our settings.

\subsubsection{Rank-based predictive score}

Apart from the energy score \eqref{eq:JAFFE.energy.score}, this section provides two additional metrics for whether the new observation (i.e., image $i_2$) fits in the predictive distribution from image $i_1$.
We shall repeat the JAFFE experiments in Section \ref{sec:JAFFE} with these two metrics.

We introduce a rank-based score based on predictive samples that measures the centrality of the new observation $\bm y_{i_2}$ with respect to the predictive distribution based on $\bm y_{i_1}$.
\begin{equation}\label{eq:JAFFE.P.score}
P_{i_1,i_2} = \frac{1}{N} \sum_{l=1}^N  \one{\|\bm y_{i_2} - \bar{\bm y}_{i_1} \|_2 < \|\hat{\bm y}_{i_1}^{(l)} - \bar{\bm y}_{i_1}\|_2},
\end{equation}
where $\bar{\bm y}_{i_1} = N^{-1}\sum_{l=1}^N \hat{\bm y}_{i_1}^{(l)}$.
A rank-based predictive score close to one means that image $i_2$ is close to the geometric center of the Bayesian predictive set, which suggests that image $i_2$ may be based on the same sparse mean vector as image $i_1$.
On the contrary, a score close to zero represents a large discrepancy detected.
Despite the notation, this is not a p-value in the frequentist sense, since its distribution is not uniform but will rather be skewed to the right under the null distribution.

For all 213 images, we can therefore run pairwise tests and obtain a square matrix $\bm P$.
The diagonal entries of $\bm P$ are set to be one.
Similar to the energy score matrix $\bm E$, $\bm P$ is also generally asymmetric.
Figure \ref{fig:JAFFE.P.matrix} shows a visualization.
To test any pairs of images $i_1$ and $i_2$, we can let $\tilde P_{i_1,i_2} = (P_{i_1,i_2} + P_{i_2,i_1}) / 2$, and accept that the pair are from the same subject if $\tilde P_{i_1,i_2}$ is larger than a threshold $\bar P$.
Again, this cutoff value $\bar P$ can be treated as a tuning parameter.
By varying $\bar P$, we obtain a ROC curve in Figure \ref{fig:JAFFE.ROC.P}.
In this case, the Horseshoe predictive inference method achieves an AUC of 0.923.
Figure \ref{fig:JAFFE.threshold.P} provides a plot of precision, recall, and F1 scores with respect to the changing threshold $\bar P$.
The global F1 score is maximized at $\bar P = 0.35$.
We can also use the three data-driven ways mentioned in Section \ref{sec:JAFFE} to select $\bar P$ according to the observed data.
The held-out method leads to $\bar P = 0.23$.
By setting 10 target clusters, we have $\bar P = 0.38$.
Note that the distribution of the rank-based predictive scores are concentrated around zero and one.
So fully unsupervised method finds the valley at $\bar P=0.655$, which is relatively far from the oracle value.
The accuracy of the four choices of $\bar P$ is demonstrated in Figure \ref{fig:JAFFE.P.results}.

A main advantage of this rank-based predictive score is robustness to the choice of cutoff $\bar P$.
The variation of global F1 score \ref{fig:JAFFE.threshold.P} is much smaller compared to that in \ref{fig:JAFFE.threshold.E}.
Likewise, even if the unsupervised choice $\bar P=0.655$ is far from the oracle value, it is still capable of recovering a reasonable accuracy.

\begin{figure}[t]
    \centering
    \begin{subfigure}[b]{0.3\textwidth}
        \centering
        \includegraphics[width=\textwidth]{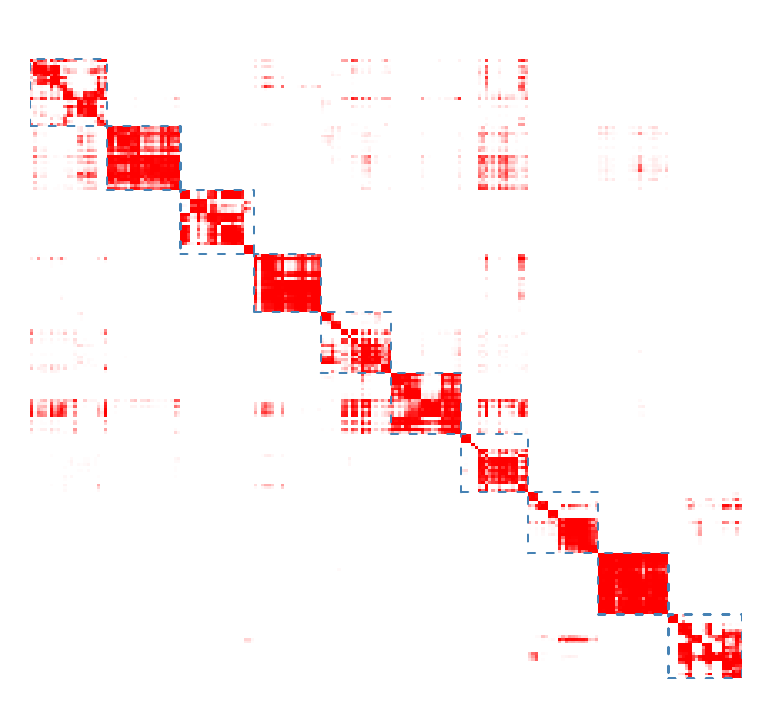}
        \caption{Score matrix $\bm P$}
        \label{fig:JAFFE.P.matrix}
    \end{subfigure}
    \hfill 
    \begin{subfigure}[b]{0.3\textwidth}
        \centering
        \includegraphics[width=\textwidth]{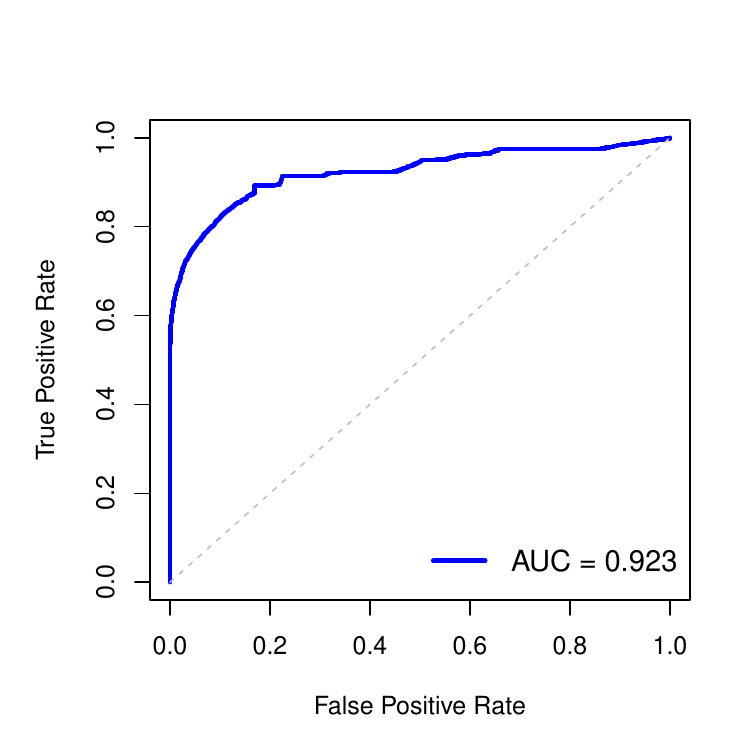}
        \caption{ROC Curve}
        \label{fig:JAFFE.ROC.P}
    \end{subfigure}
    \hfill 
    \begin{subfigure}[b]{0.3\textwidth}
        \centering
        \includegraphics[width=\textwidth]{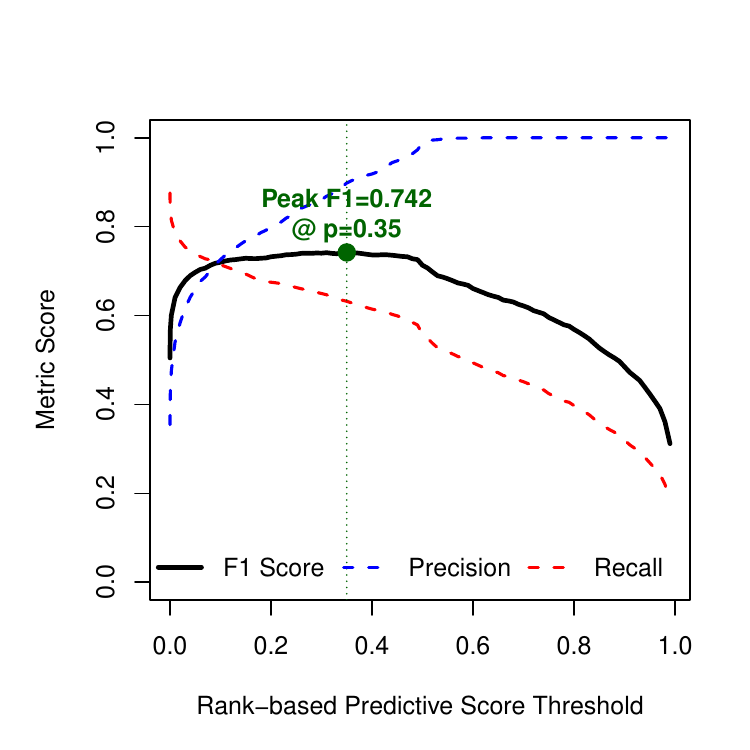}
        \caption{Accuracy of varying $\bar P$}
        \label{fig:JAFFE.threshold.P}
    \end{subfigure}
    
    \caption{Results of the Horseshoe predictive inference method on the JAFFE dataset, using rank-based predictive score as the metric.
    (a) Heatmap of rank-based predictive score matrix $\bm P$. 
    Each row and column represents an image.
    Larger rank-based predictive scores are represented by darker red.
    The blue dashed boxes represent images of one subject.
    (b) The ROC curve obtained by varying the pairwise testing cutoff value $\bar P$.
    (c) In-sample selection of optimal threshold $\bar P$, with precision, recall, and F1 score plotted against varying $\bar P$.}
    \label{fig:JAFFE.P.varying}
\end{figure}

\begin{figure}[t]
    \centering
    \begin{subfigure}[b]{0.24\textwidth}
        \centering
        \includegraphics[width=\textwidth]{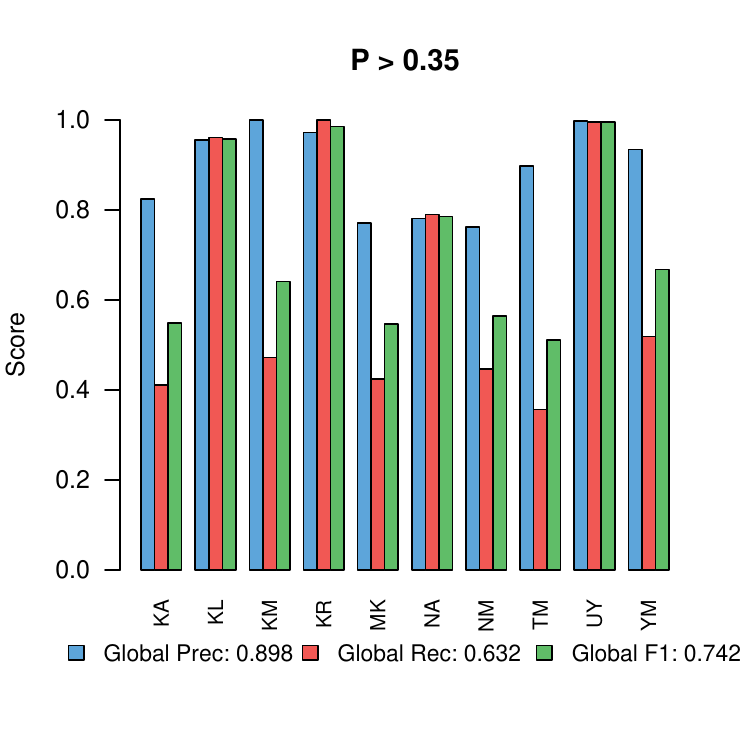}
        \caption{Oracle $\bar P$}
        \label{fig:P.MEAN.oracle}
    \end{subfigure}
    \begin{subfigure}[b]{0.24\textwidth}
        \centering
        \includegraphics[width=\textwidth]{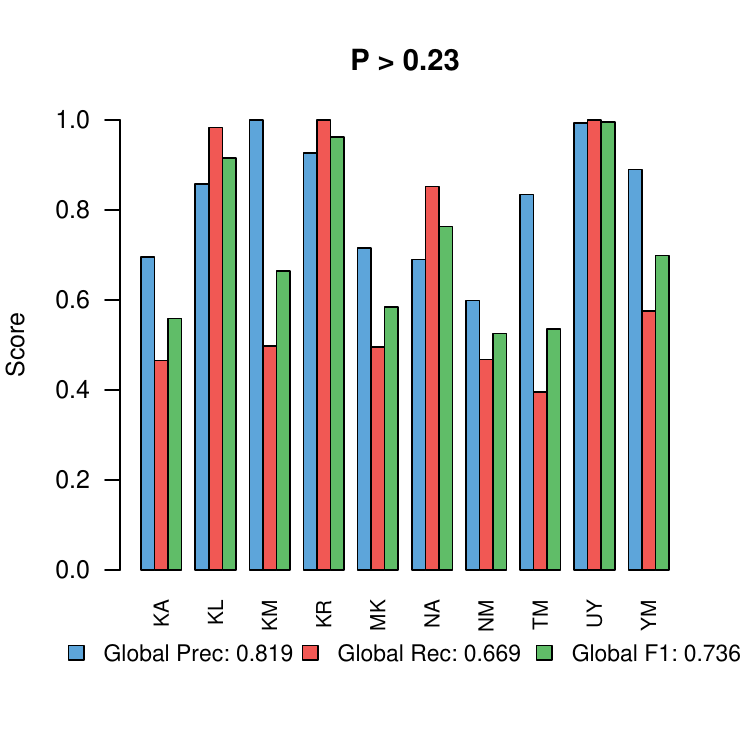}
        \caption{$\bar P$ from held-out set}
        \label{fig:P.MEAN.heldout}
    \end{subfigure}
    \begin{subfigure}[b]{0.24\textwidth}
        \centering
        \includegraphics[width=\textwidth]{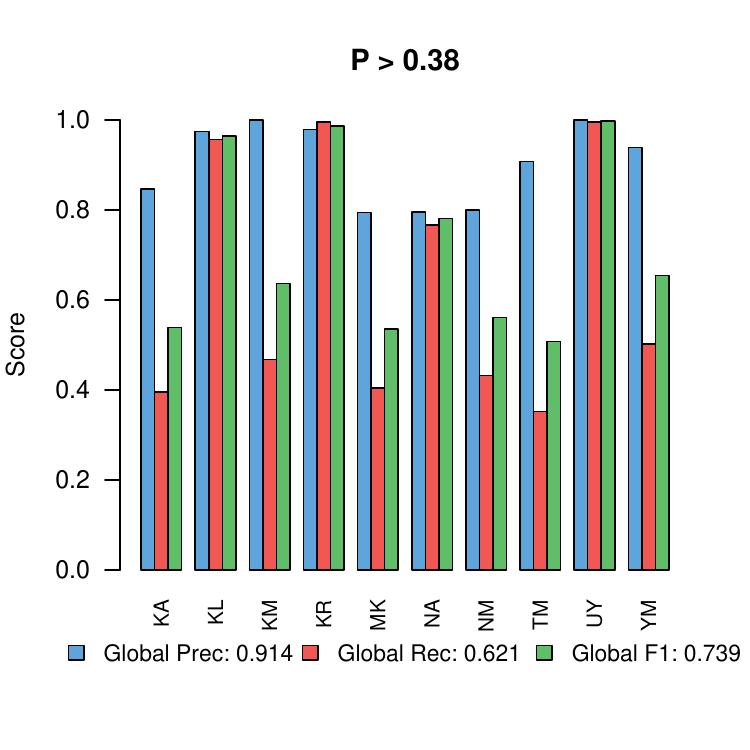}
        \caption{$\bar P$ for 10 clusters}
        \label{fig:P.MEAN.cluster}
    \end{subfigure}
    \begin{subfigure}[b]{0.24\textwidth}
        \centering
        \includegraphics[width=\textwidth]{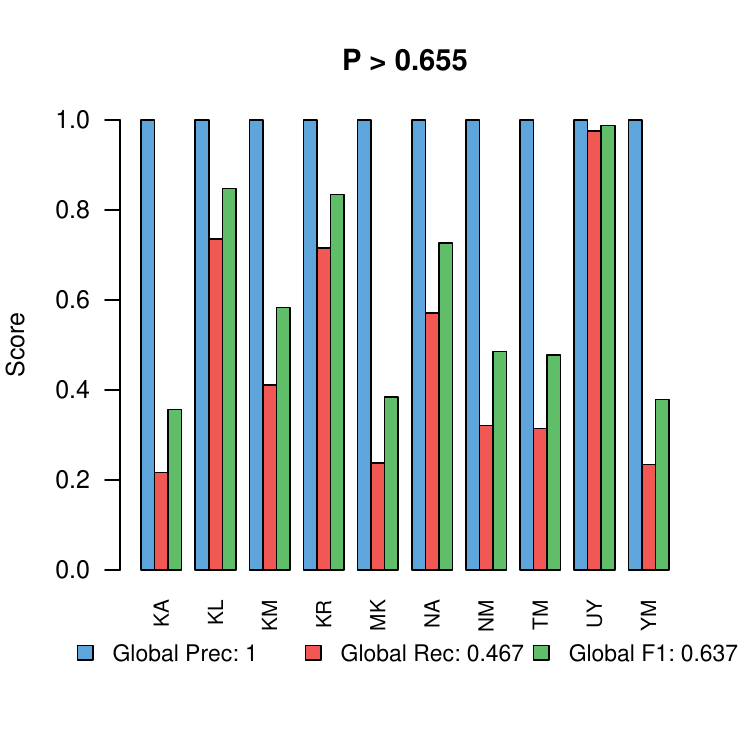}
        \caption{Valley $\bar P$}
        \label{fig:P.MEAN.valley}
    \end{subfigure}
    
    \caption{Accuracy of JAFFE facial recognition under different choices of cutoff $\bar P$. The predictive inference is done using the Horseshoe prior.}
    \label{fig:JAFFE.P.results}
\end{figure}

\subsubsection{Predictive coverage rate}

A perhaps more intuitive score is the entrywise coverage rate in the predictive inference.
This may be useful considering that the Daubechies wavelet decomposition essentially divides the image into adjacent tiles of different sizes, and that the ratio of tiles of image $i_2$ that fall into the predictive interval can also be used to test the null hypothesis that the two images are based on the same mean.
The coverage rate is defined as
\begin{equation}\label{eq:JAFFE.C.score}
    C_{i_1, i_2} = \frac{1}{n}\sum_{j=1}^n\one{q_{\alpha/2}(\hat y_{i_1,j}) < y_{i_2, j} < q_{1-\alpha/2}(\hat y_{i_1,j})}
\end{equation}
Here, $\alpha = 0.1$ is the significance level, $q_{\alpha/2}(\hat y_{i_1,j})$ and $q_{1-\alpha/2}(\hat y_{i_1,j})$ are the lower and upper quantiles of the marginal predictive distributions based on $\bm y_{i_1}$.
Recall that $n$ is the number of tiles, instead of the number of samples from the predictive distribution.
For any two images based on the same mean, $C_{i_1,i_2}$ should be very close to one.
A score far from one indicates that there may be discrepancy in mean vectors.

For all 213 images, we can therefore run pairwise tests and obtain a square matrix $\bm C$.
The diagonal entries of $\bm C$ are set to be one.
Note that $\bm C$ is also generally asymmetric.
Figure \ref{fig:JAFFE.C.matrix} shows a visualization.
To test any pairs of images $i_1$ and $i_2$, we can let $\tilde C_{i_1,i_2} = (C_{i_1,i_2} + C_{i_2,i_1}) / 2$, and accept that the pair are from the same subject if $\tilde C_{i_1,i_2}$ is larger than a threshold $\bar C$ -- in a similar logic to the $\bm P$ metric.
Again, this cutoff value $\bar C$ can be treated as a tuning parameter.
By varying $\bar C$, we obtain a ROC curve in Figure \ref{fig:JAFFE.ROC.C}.
In this case, the Horseshoe predictive inference method achieves an AUC of 0.907.
Figure \ref{fig:JAFFE.threshold.C} provides a plot of precision, recall, and F1 scores with respect to the changing threshold $\bar C$.
The global F1 score is maximized at $\bar C = 0.905$.
We can also use the three data-driven ways mentioned in Section \ref{sec:JAFFE} to select $\bar C$ according to the observed data.
The held-out method leads to $\bar C = 0.903$.
By setting 10 target clusters, we have $\bar C = 0.912$.
The valley method leads to $\bar C=0.883$, which is relatively far from the oracle value.
The accuracy of the four choices of $\bar C$ is demonstrated in Figure \ref{fig:JAFFE.C.results}.

\begin{figure}[t]
    \centering
    \begin{subfigure}[b]{0.3\textwidth}
        \centering
        \includegraphics[width=\textwidth]{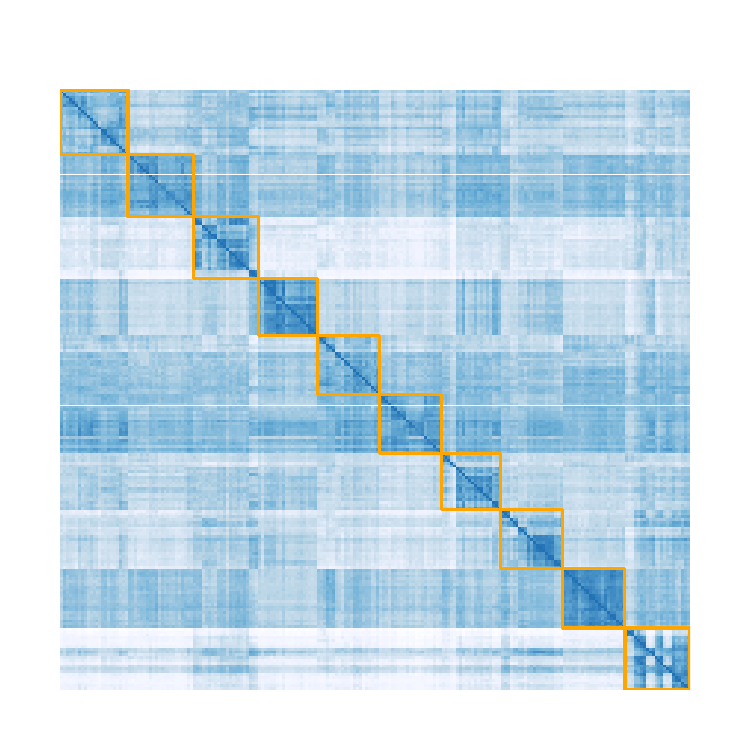}
        \caption{Score matrix $\bm C$}
        \label{fig:JAFFE.C.matrix}
    \end{subfigure}
    \hfill 
    \begin{subfigure}[b]{0.3\textwidth}
        \centering
        \includegraphics[width=\textwidth]{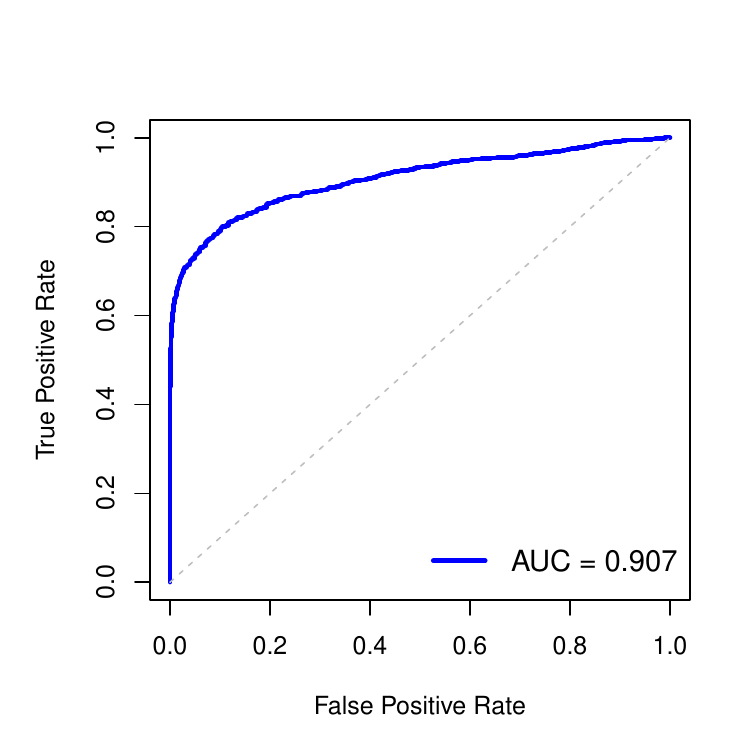}
        \caption{ROC Curve}
        \label{fig:JAFFE.ROC.C}
    \end{subfigure}
    \hfill 
    \begin{subfigure}[b]{0.3\textwidth}
        \centering
        \includegraphics[width=\textwidth]{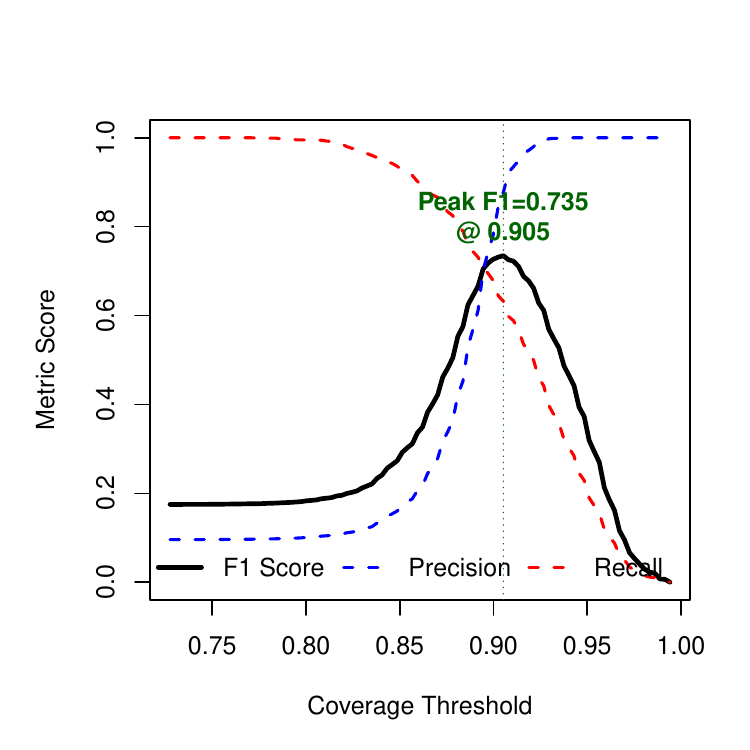}
        \caption{Accuracy of varying $\bar C$}
        \label{fig:JAFFE.threshold.C}
    \end{subfigure}
    
    \caption{Results of the Horseshoe predictive inference method on the JAFFE dataset, using coverage rate as the metric.
    (a) Heatmap of coverage rate matrix $\bm C$. 
    Each row and column represents an image.
    Larger coverage rates are represented by darker blue.
    The blue dashed boxes represent images of one subject.
    (b) The ROC curve obtained by varying the pairwise testing cutoff value $\bar C$.
    (c) In-sample selection of optimal threshold $\bar C$, with precision, recall, and F1 score plotted against varying $\bar C$.}
    \label{fig:JAFFE.C.varying}
\end{figure}

\begin{figure}[t]
    \centering
    \begin{subfigure}[b]{0.24\textwidth}
        \centering
        \includegraphics[width=\textwidth]{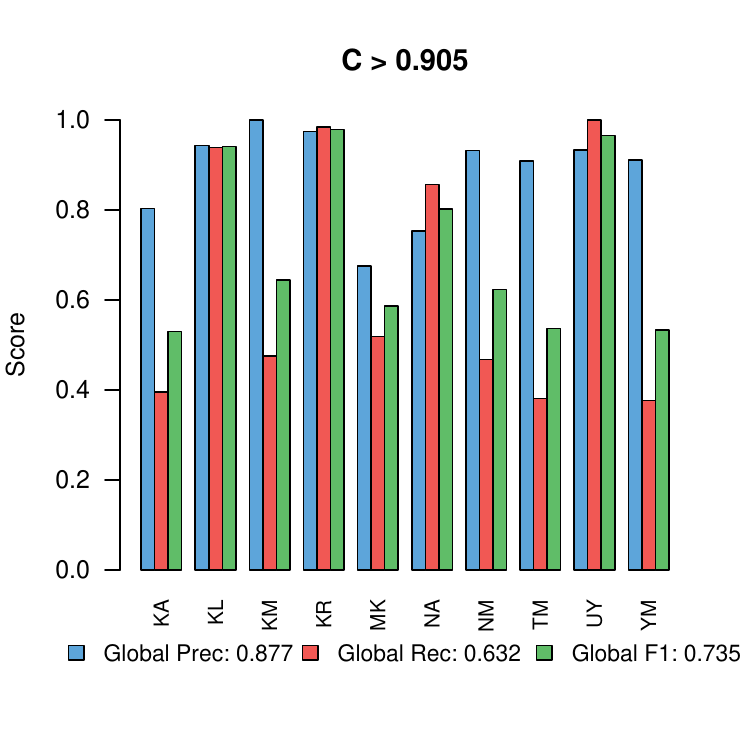}
        \caption{Oracle $\bar C$}
        \label{fig:C.MEAN.oracle}
    \end{subfigure}
    \begin{subfigure}[b]{0.24\textwidth}
        \centering
        \includegraphics[width=\textwidth]{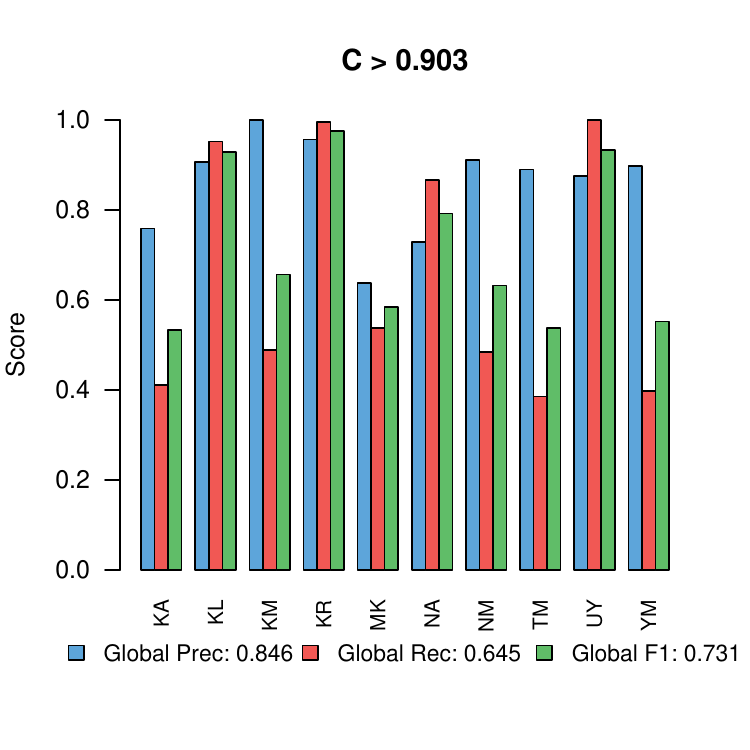}
        \caption{$\bar C$ from held-out set}
        \label{fig:C.MEAN.heldout}
    \end{subfigure}
    \begin{subfigure}[b]{0.24\textwidth}
        \centering
        \includegraphics[width=\textwidth]{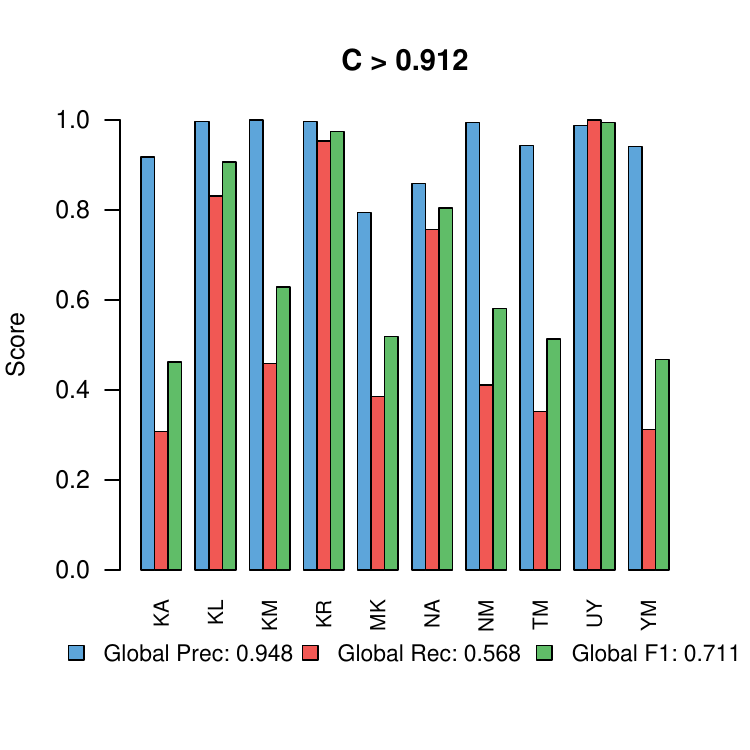}
        \caption{$\bar C$ for 10 clusters}
        \label{fig:C.MEAN.cluster}
    \end{subfigure}
    \begin{subfigure}[b]{0.24\textwidth}
        \centering
        \includegraphics[width=\textwidth]{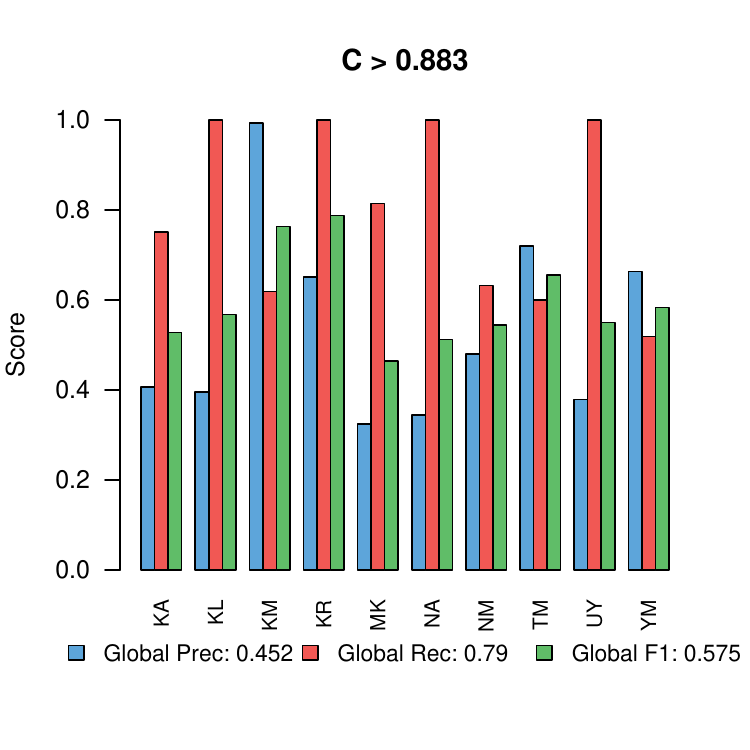}
        \caption{Valley $\bar C$}
        \label{fig:C.MEAN.valley}
    \end{subfigure}
    
    \caption{Accuracy of JAFFE facial recognition under different choices of cutoff $\bar C$. The predictive inference is done using the Horseshoe prior.}
    \label{fig:JAFFE.C.results}
\end{figure}

\subsection{ABIDE Details and Extra Experiments}

\subsubsection{Functional principal component analysis (FPCA)}\label{sec:apdx.FPCA}

Given an autocovariance kernel $K_{jj}(t,t') = \text{Cov}(g_{ij}(t), g_{ij}(t'))$, there exist eigenpairs $\{\lambda_{jm}, \phi_{jm}(\cdot)\}_{m\in\mathbb{N}}$ (Theorem 7.2.6 of \cite{hsing2015theoretical}), where $\{\lambda_{jm}\}_{m\in\mathbb{N}}$ are the eigenvalues, and $\{\phi_{jm}(\cdot)\}_{m\in\mathbb N}$ are orthonormal eigenfunctions.
Since the covariance kernel is symmetric and positive semidefinite, we assume $\lambda_{j1} \geq \lambda_{j2} \geq \cdots \geq 0$ without loss of generality.
According to the Karhunen-Lo\`eve Theorem (Theorem 1.5 of \cite{bosq2000linear}), we have the following infinite-dimensional representation,
\[
g_{ij}(\cdot) = \sum_{m=1}^\infty a_{ijm} \phi_{jm}(\cdot),
\]
where the coefficients (i.e. FPC scores) are
\[
a_{ijm} = \langle g_{ij}, \phi_{jm}\rangle = \int g_{ij}(t) \phi_{jm}(t) \diff t,
\]
which follows a Gaussian distribution with a variance of $\lambda_{jm}$,
with $a_{ijm}$ independent of $a_{ijm'}$ for any $m\neq m'$.
The eigenfunctions $\{\phi_{jm}(\cdot)\}_{m\in\mathbb N}$ form a basis that we refer to as the FPCA basis.

In practice, since the true underlying covariance kernel $K_{jj}$ is unknown, we may first estimate the empirical version $\hat K_{jj}(t,t') = n^{-1}\sum_{i=1}^n g_{ij}(t) g_{ij}(t')$ over $n$ patients, and subsequently obtain the estimated eigenpairs $\{\hat\lambda_{jm}, \hat\phi_{jm}(\cdot)\}_{m\in\mathbb{N}}$.
We may choose only the first $M$ basis functions and estimate FPC scores $\hat a_{ijm} = \langle g_{ij}, \hat\phi_{jm}\rangle$ for $1\leq m\leq M$.
The FPC scores, essentially projection scores from a dimension reduction procedure, will then inherit the correlation between brain areas.
What is important to our analysis, though, is that the FPC score vectors within the same brain area can be considered mutually independent, satisfying Gaussian sequence model.

\subsubsection{Benjamini--Yekutieli (BY) correction}\label{sec:apdx.BY}

To account for multiple testing issue across 54 pairs of ROIs, it is necessary to control false discovery rate (FDR).
Considering the dense multicollinearity and complex dependency structure for time series signals across the brain, we apply the Benjamini--Yekutieli (BY) procedure \citep{benjamini2001control} for FDR control.
The BY procedure introduces a penalty factor defined by the harmonic number $c_K = \sum_{k=1}^K \frac{1}{k}$.
For $K=54$ tests, this penalty factor is $c_{54} \approx 4.56$.
Let $p_{(1)} \leq p_{(2)} \leq \dots \leq p_{(K)}$ denote the ordered raw $p$-values obtained from the Wilcoxon rank-sum tests, and let $H_{(1)}, H_{(2)}, \dots, H_{(K)}$ be the corresponding null hypotheses.
The BY step-up procedure rejects all hypotheses $H_{(i)}$ for $i = 1, \dots, k$, where $k$ is the largest integer such that $p_{(k)} \leq \frac{k}{m \cdot c_m} \alpha$.
To present these results continuously rather than relying on a binary decision threshold at a fixed $\alpha$, we compute the BY-adjusted $p$-values. The adjusted $p$-value represents the smallest FDR level $\alpha$ at which a given hypothesis would be rejected.
For the $i$-th ordered test, the nominal adjusted value is $p_{(i)} \cdot \frac{K \cdot c_K}{i}$.

To ensure the adjusted $p$-values remain monotonically increasing (i.e., if a hypothesis with a larger raw $p$-value is rejected, all hypotheses with smaller raw $p$-values must also be rejected), we apply a step-down minimum constraint.
The final BY-adjusted $p$-value for the $i$-th ordered test is calculated as:
\begin{equation}
p_{\text{BY}}^{(i)} = \min_{j \ge i} \left\{ \min\left(1, p_{(j)} \cdot \frac{K \cdot c_K}{j}\right) \right\}.
\end{equation}
These $p_{\text{BY}}$ values are reported in our results, and regional symmetry changes are deemed significant if $p_{\text{BY}} < 0.05$.

\begin{table}[htbp]
\centering
\begin{singlespacing}
\renewcommand{\arraystretch}{1.2}
\caption{The 54 symmetric pairs of brain regions and Wilcoxon rank-sum test results based on energy scores.
The alternative hypothesis is that the distribution of energy scores differs between the ASD and Control groups.
Bold values indicate a significant change in symmetry pattern after Benjamini--Yekutieli procedure.
\emph{Hyper} indicates increased symmetry and strengthened interconnectivity.
\emph{Hypo} indicates decreased symmetry and weakened interconnectivity.
}
\label{tab:ABIDE.wilcoxon}
\small
\setlength{\tabcolsep}{2pt}
\begin{tabular}{clccc|clccc}
\toprule
\textbf{$k$} & \textbf{Region Pair} & {$p_{\text{raw}}$} & {$p_{\text{BY}}$} & \textbf{Status} & \textbf{$k$} & \textbf{Region Pair} & {$p_{\text{raw}}$} & {$p_{\text{BY}}$} & \textbf{Status} \\
\midrule
1 & Precentral & 0.290 & 1.000 & - & 28 & Fusiform & 0.818 & 1.000 & - \\
2 & Frontal\_Sup & 0.797 & 1.000 & - & 29 & Postcentral & 0.374 & 1.000 & - \\
3 & Frontal\_Sup\_Orb & 0.012 & 0.190 & - & 30 & Parietal\_Sup & 0.003 & 0.066 & - \\
4 & Frontal\_Mid & 0.029 & 0.343 & - & 31 & Parietal\_Inf & 0.046 & 0.515 & - \\
5 & Frontal\_Mid\_Orb & 0.066 & 0.653 & - & 32 & SupraMarginal & \textbf{0.001} & \textbf{0.027} & Hyper \\
6 & Frontal\_Inf\_Oper & 0.060 & 0.617 & - & 33 & Angular & 0.015 & 0.210 & - \\
7 & Frontal\_Inf\_Tri & \textbf{0.001} & \textbf{0.027} & Hypo & 34 & Precuneus & 0.845 & 1.000 & - \\
8 & Frontal\_Inf\_Orb & 0.329 & 1.000 & - & 35 & Paracentral\_Lob & 0.013 & 0.194 & - \\
9 & Rolandic\_Oper & \textbf{0.000} & \textbf{0.004} & Hypo & 36 & Caudate & 0.354 & 1.000 & - \\
10 & Supp\_Motor & 0.057 & 0.612 & - & 37 & Putamen & 0.941 & 1.000 & - \\
11 & Olfactory & 0.974 & 1.000 & - & 38 & Pallidum & 0.006 & 0.101 & - \\
12 & Frontal\_Sup\_Med & \textbf{0.000} & \textbf{0.000} & Hypo & 39 & Thalamus & 0.120 & 1.000 & - \\
13 & Frontal\_Med\_Orb & 0.974 & 1.000 & - & 40 & Heschl & \textbf{0.001} & \textbf{0.027} & Hyper \\
14 & Rectus & 0.004 & 0.079 & - & 41 & Temporal\_Sup & 0.855 & 1.000 & - \\
15 & Insula & 0.625 & 1.000 & - & 42 & Temporal\_Pole\_Sup & 0.811 & 1.000 & - \\
16 & Cingulum\_Ant & 0.436 & 1.000 & - & 43 & Temporal\_Mid & 0.174 & 1.000 & - \\
17 & Cingulum\_Mid & 0.264 & 1.000 & - & 44 & Temporal\_Pole\_Mid & 0.951 & 1.000 & - \\
18 & Cingulum\_Post & 0.597 & 1.000 & - & 45 & Temporal\_Inf & 0.441 & 1.000 & - \\
19 & Hippocampus & 0.212 & 1.000 & - & 46 & Cerebellum\_Crus1 & 0.005 & 0.097 & - \\
20 & ParaHippocampal & 0.766 & 1.000 & - & 47 & Cerebellum\_Crus2 & 0.934 & 1.000 & - \\
21 & Amygdala & \textbf{0.000} & \textbf{0.004} & Hypo & 48 & Cerebellum\_3 & 0.005 & 0.097 & - \\
22 & Calcarine & \textbf{0.000} & \textbf{0.000} & Hyper & 49 & Cerebellum\_4\_5 & \textbf{0.000} & \textbf{0.000} & Hyper \\
23 & Cuneus & 0.020 & 0.255 & - & 50 & Cerebellum\_6 & 0.024 & 0.290 & - \\
24 & Lingual & \textbf{0.000} & \textbf{0.002} & Hypo & 51 & Cerebellum\_7b & 0.120 & 1.000 & - \\
25 & Occipital\_Sup & 0.491 & 1.000 & - & 52 & Cerebellum\_8 & 0.539 & 1.000 & - \\
26 & Occipital\_Mid & 0.843 & 1.000 & - & 53 & Cerebellum\_9 & \textbf{0.001} & \textbf{0.027} & Hypo \\
27 & Occipital\_Inf & 0.348 & 1.000 & - & 54 & Cerebellum\_10 & 0.350 & 1.000 & - \\
\bottomrule
\end{tabular}
\end{singlespacing}
\end{table}

\subsubsection{Rank-based predictive score}

Apart from the energy score in \eqref{eq:ABIDE.energy.score},
the proximity between two ROIs can also be measured by the rank-based predictive score,
\begin{equation}\label{eq:ABIDE.P.score}
P^i_{j_1,j_2} = \frac{1}{N} \sum_{l=1}^N  \one{\|\bm y_{i,j_2} - \bar{\bm y}_{i,j_1} \|_2 < \|\hat{\bm y}_{i,j_1}^{(l)} - \bar{\bm y}_{i,j_1}\|_2},
\end{equation}
where $\bar{\bm y}_{i,j_1} = N^{-1}\sum_{l=1}^N \hat{\bm y}_{i,j_1}^{(l)}$.
Again, $P^i_{j_1,j_2} \in [0,1]$.
For a test for ROI pair $k$ that contains regions $j_1$ and $j_2$, we consider the aggregated rank-based predictive score $\tilde P^i_{j_1,j_2} = (P^i_{j_1, j_2} + P^i_{j_2, j_1})/2$.
We compare the distribution of $\{\tilde P^i_{j_1,j_2}\}_{i\in \mathcal I_{\text{ASD}}}$ versus that of $\{\tilde P^i_{j_1,j_2}\}_{i\in \mathcal I_{\text{Ctrl}}}$ with Wilcoxon rank-sum test.
A boxplot is given in Figure \ref{fig:ABIDE.P.boxplot}.
A visualized 3D brain image of significantly hypo- or hyper-symmetric areas is given in Figure \ref{fig:ABIDE.P.brain.image}.
Table \ref{tab:ABIDE.wilcoxon.P} provides the p-values in the two-sided Wilcoxon rank-sum tests.

\begin{figure}[p]
    \centering
    
    \includegraphics[width=0.95\linewidth]{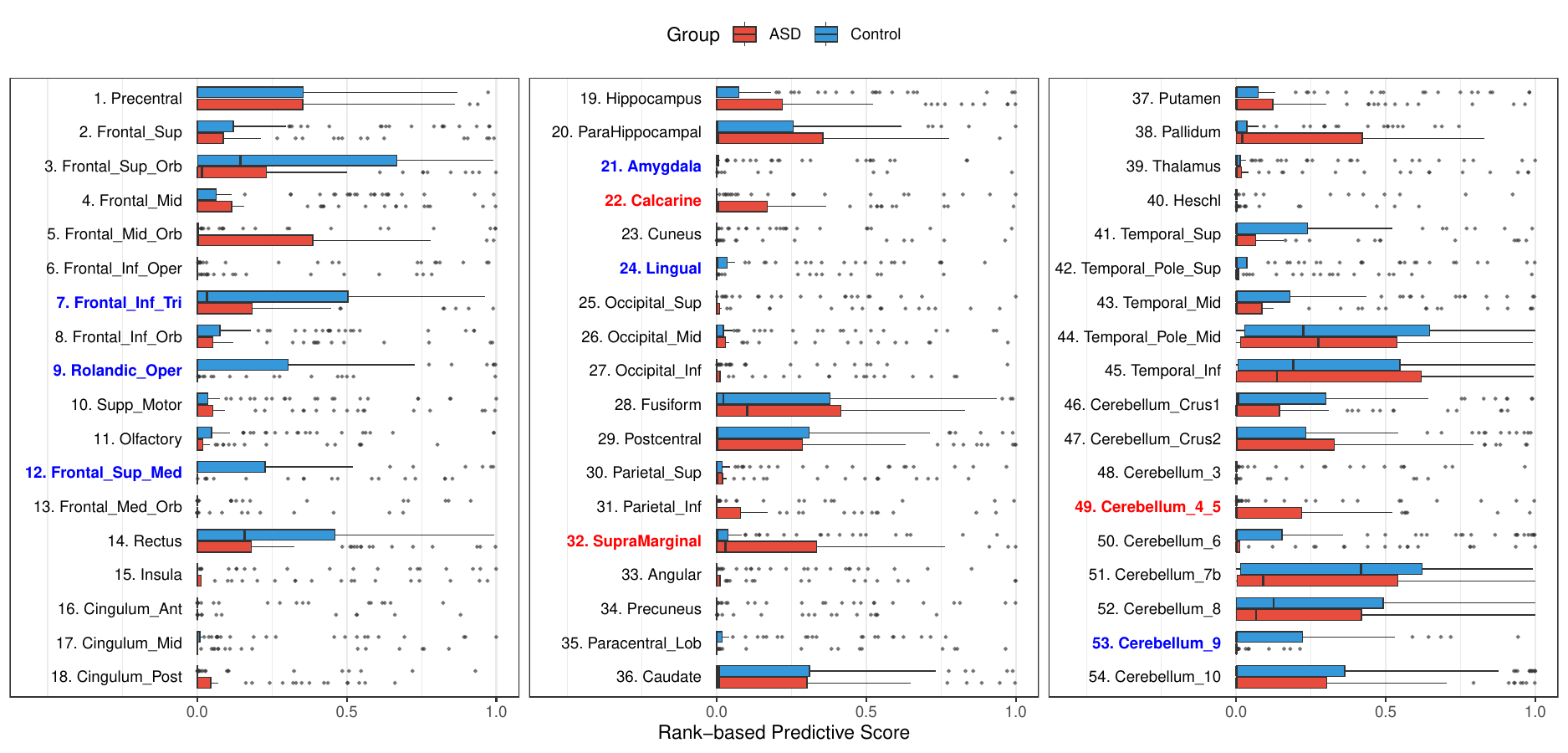}
    \caption{Distribution of rank-based predictive scores across 54 brain regions.
    Boxplots display the distribution of these scores for the ASD group (red) and control group (blue).
    Region labels are color-coded based on significant group differences in Wilcoxon rank-sum test, subject to Benjamini--Yekutieli correction.
    Hypo-symmetric (i.e., ASD less symmetric than control) regions are marked in blue, while hyper-symmetric (i.e., ASD more symmetric than control) regions are in red.}
    \label{fig:ABIDE.P.boxplot}
    
    \vspace{1cm} 
    
    \includegraphics[width=0.45\linewidth]{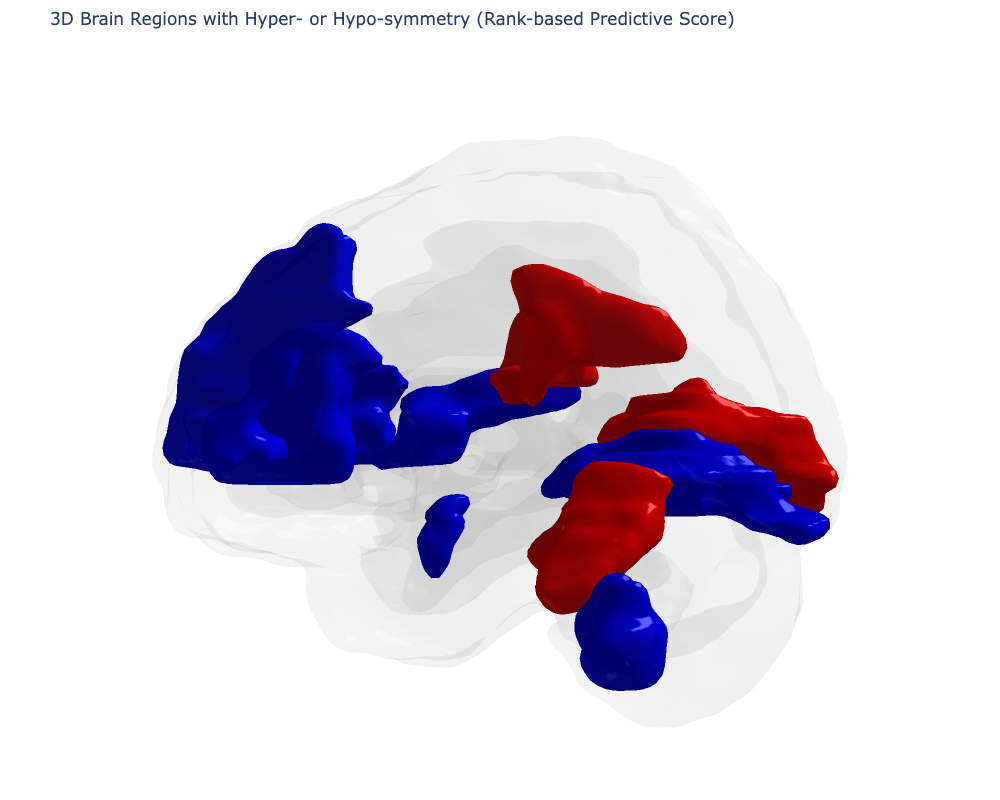}
    \includegraphics[width=0.45\linewidth]{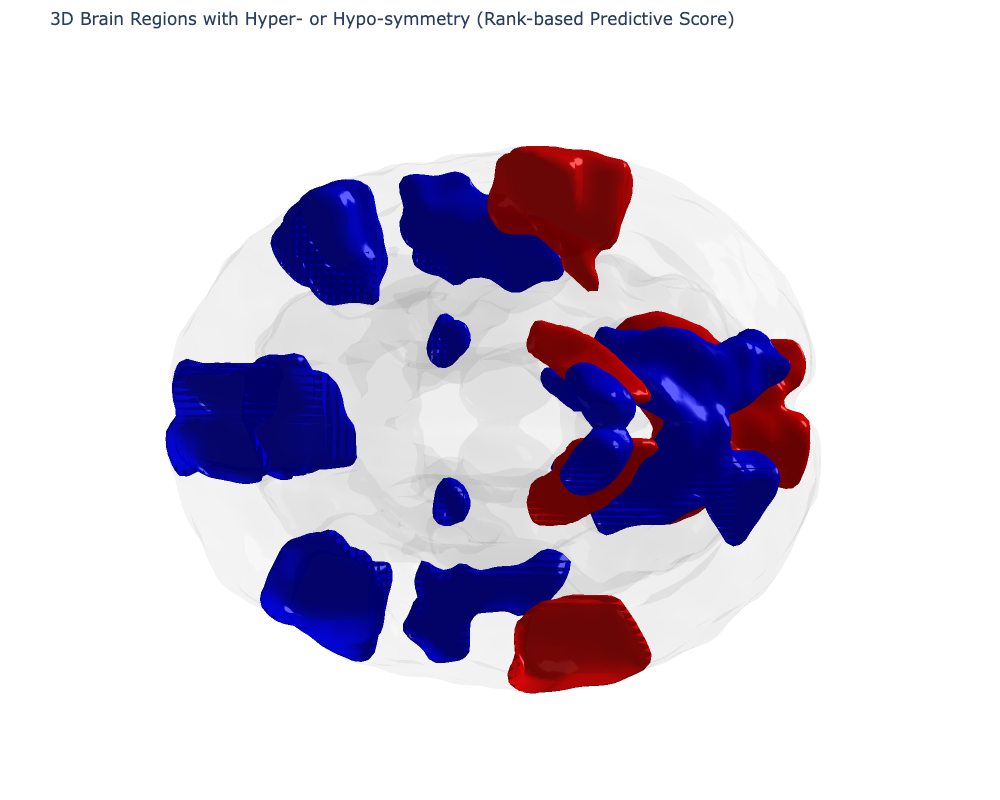}
    \caption{Anatomical topography of brain regions with significant discrepancy between ASD and control groups, based on rank-based predictive score.
    Hypo-symmetric regions are marked in blue, while hyper-symmetric regions are in red.
    The first image shows the view from the left, while the second image shows the view from above.
    In both images, the front of the brain is on the left, and the rear is on the right.}
    \label{fig:ABIDE.P.brain.image}
    
\end{figure}

\begin{table}[htbp]
\centering
\begin{singlespacing}
\renewcommand{\arraystretch}{1.2}
\caption{The 54 symmetric pairs of brain regions and Wilcoxon rank-sum test results based on rank-based predictive scores.
The alternative hypothesis is that the distribution of rank-based predictive scores differs between the ASD and Control groups.
Bold values indicate a significant change in symmetry pattern after Benjamini--Yekutieli procedure.
\emph{Hyper} indicates increased symmetry and strengthened interconnectivity.
\emph{Hypo} indicates decreased symmetry and weakened interconnectivity.
}
\label{tab:ABIDE.wilcoxon.P}
\small
\setlength{\tabcolsep}{2pt}
\begin{tabular}{clccc|clccc}
\toprule
\textbf{$k$} & \textbf{Region Pair} & {$p_{\text{raw}}$} & {$p_{\text{BY}}$} & \textbf{Status} & \textbf{$k$} & \textbf{Region Pair} & {$p_{\text{raw}}$} & {$p_{\text{BY}}$} & \textbf{Status} \\
\midrule
1 & Precentral & 0.520 & 1.000 & - & 28 & Fusiform & 0.787 & 1.000 & - \\
2 & Frontal\_Sup & 0.734 & 1.000 & - & 29 & Postcentral & 0.534 & 1.000 & - \\
3 & Frontal\_Sup\_Orb & 0.010 & 0.187 & - & 30 & Parietal\_Sup & 0.007 & 0.153 & - \\
4 & Frontal\_Mid & 0.042 & 0.611 & - & 31 & Parietal\_Inf & 0.050 & 0.649 & - \\
5 & Frontal\_Mid\_Orb & 0.045 & 0.611 & - & 32 & SupraMarginal & \textbf{0.002} & \textbf{0.047} & Hyper \\
6 & Frontal\_Inf\_Oper & 0.130 & 1.000 & - & 33 & Angular & 0.014 & 0.244 & - \\
7 & Frontal\_Inf\_Tri & \textbf{0.000} & \textbf{0.013} & Hypo & 34 & Precuneus & 0.516 & 1.000 & - \\
8 & Frontal\_Inf\_Orb & 0.624 & 1.000 & - & 35 & Paracentral\_Lob & 0.097 & 1.000 & - \\
9 & Rolandic\_Oper & \textbf{0.000} & \textbf{0.002} & Hypo & 36 & Caudate & 0.312 & 1.000 & - \\
10 & Supp\_Motor & 0.355 & 1.000 & - & 37 & Putamen & 0.711 & 1.000 & - \\
11 & Olfactory & 0.804 & 1.000 & - & 38 & Pallidum & 0.004 & 0.102 & - \\
12 & Frontal\_Sup\_Med & \textbf{0.001} & \textbf{0.039} & Hypo & 39 & Thalamus & 0.645 & 1.000 & - \\
13 & Frontal\_Med\_Orb & 0.807 & 1.000 & - & 40 & Heschl & 0.062 & 0.740 & - \\
14 & Rectus & 0.009 & 0.179 & - & 41 & Temporal\_Sup & 0.656 & 1.000 & - \\
15 & Insula & 0.705 & 1.000 & - & 42 & Temporal\_Pole\_Sup & 0.926 & 1.000 & - \\
16 & Cingulum\_Ant & 0.760 & 1.000 & - & 43 & Temporal\_Mid & 0.305 & 1.000 & - \\
17 & Cingulum\_Mid & 0.555 & 1.000 & - & 44 & Temporal\_Pole\_Mid & 0.738 & 1.000 & - \\
18 & Cingulum\_Post & 0.383 & 1.000 & - & 45 & Temporal\_Inf & 0.434 & 1.000 & - \\
19 & Hippocampus & 0.666 & 1.000 & - & 46 & Cerebellum\_Crus1 & 0.024 & 0.402 & - \\
20 & ParaHippocampal & 0.640 & 1.000 & - & 47 & Cerebellum\_Crus2 & 0.892 & 1.000 & - \\
21 & Amygdala & \textbf{0.000} & \textbf{0.005} & Hypo & 48 & Cerebellum\_3 & 0.042 & 0.611 & - \\
22 & Calcarine & \textbf{0.000} & \textbf{0.001} & Hyper & 49 & Cerebellum\_4\_5 & \textbf{0.000} & \textbf{0.018} & Hyper \\
23 & Cuneus & 0.085 & 0.952 & - & 50 & Cerebellum\_6 & 0.268 & 1.000 & - \\
24 & Lingual & \textbf{0.000} & \textbf{0.003} & Hypo & 51 & Cerebellum\_7b & 0.063 & 0.740 & - \\
25 & Occipital\_Sup & 0.201 & 1.000 & - & 52 & Cerebellum\_8 & 0.534 & 1.000 & - \\
26 & Occipital\_Mid & 0.937 & 1.000 & - & 53 & Cerebellum\_9 & \textbf{0.001} & \textbf{0.021} & Hypo \\
27 & Occipital\_Inf & 0.809 & 1.000 & - & 54 & Cerebellum\_10 & 0.581 & 1.000 & - \\
\bottomrule
\end{tabular}
\end{singlespacing}
\end{table}

\subsubsection{Predictive coverage rate}

We can also use the entrywise coverage rate in the predictive inference as the metric. 
\begin{equation}\label{eq:ABIDE.C.score}
    C^i_{j_1, j_2} = \frac{1}{n}\sum_{v=1}^n\one{q_{\alpha/2}(\hat y_{i,j_1,v}) < y_{i,j_2, v} < q_{1-\alpha/2}(\hat y_{i,j_1,v})}
\end{equation}
Here, $\alpha = 0.1$ is the significance level, $q_{\alpha/2}(\hat y_{i,j_1,v})$ and $q_{1-\alpha/2}(\hat y_{i,j_1,v})$ are the lower and upper quantiles of the marginal predictive distributions based on $\bm y_{i,j_1}$.
Again, $C^i_{j_1,j_2} \in [0,1]$.
For a test for ROI pair $k$ that contains regions $j_1$ and $j_2$, we consider the aggregated convergence rate $\tilde C^i_{j_1,j_2} = (C^i_{j_1, j_2} + C^i_{j_2, j_1})/2$.
We compare the distribution of $\{\tilde C^i_{j_1,j_2}\}_{i\in \mathcal I_{\text{ASD}}}$ versus that of $\{\tilde C^i_{j_1,j_2}\}_{i\in \mathcal I_{\text{Ctrl}}}$ with Wilcoxon rank-sum test.
A boxplot is given in Figure \ref{fig:ABIDE.C.boxplot}.
A visualized 3D brain image of significantly hypo- or hyper-symmetric areas is given in Figure \ref{fig:ABIDE.C.brain.image}.
Table \ref{tab:ABIDE.wilcoxon.C} provides the p-values in the two-sided Wilcoxon rank-sum tests.

\begin{figure}[p]
    \centering
    
    \includegraphics[width=0.95\linewidth]{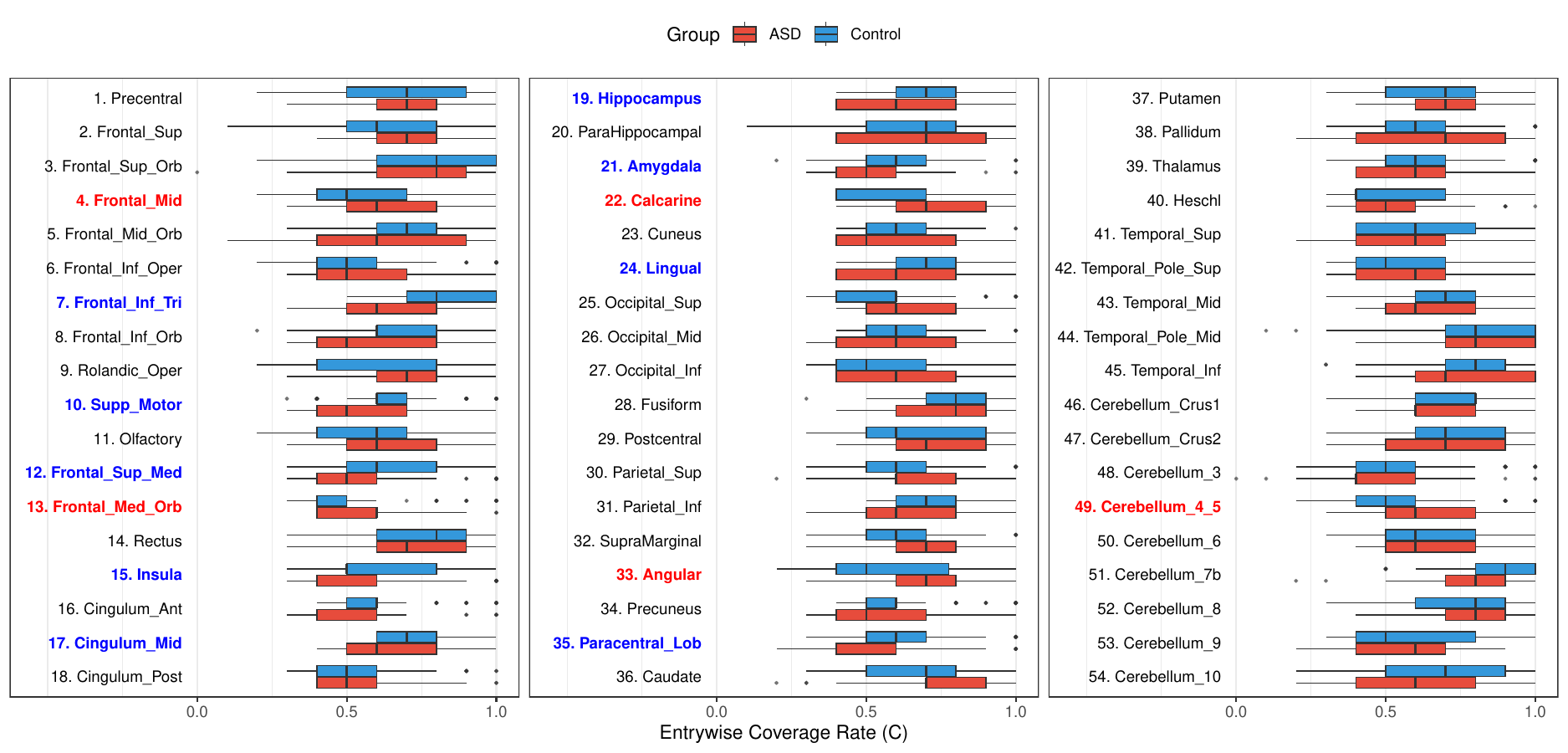}
    \caption{Distribution of coverage rates across 54 brain regions.
    Boxplots display the distribution of these scores for the ASD group (red) and control group (blue).
    Region labels are color-coded based on significant group differences in Wilcoxon rank-sum test, subject to Benjamini--Yekutieli correction.
    Hypo-symmetric (i.e., ASD less symmetric than control) regions are marked in blue, while hyper-symmetric (i.e., ASD more symmetric than control) regions are in red.}
    \label{fig:ABIDE.C.boxplot}
    
    \vspace{1cm} 
    
    \includegraphics[width=0.45\linewidth]{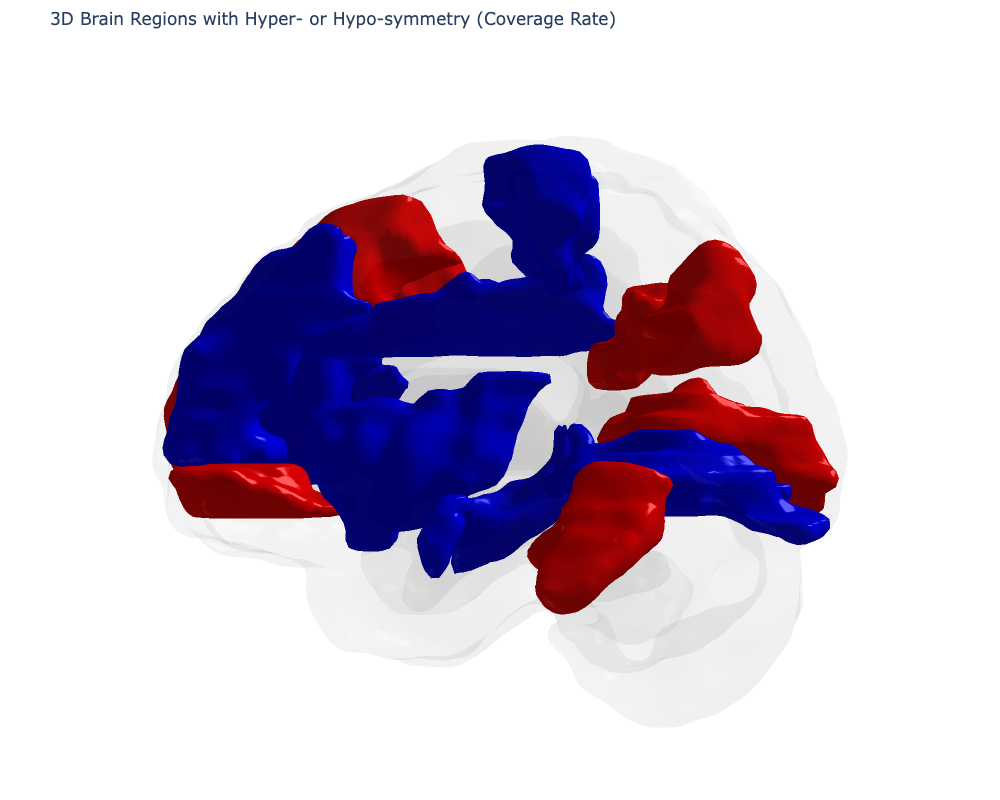}
    \includegraphics[width=0.45\linewidth]{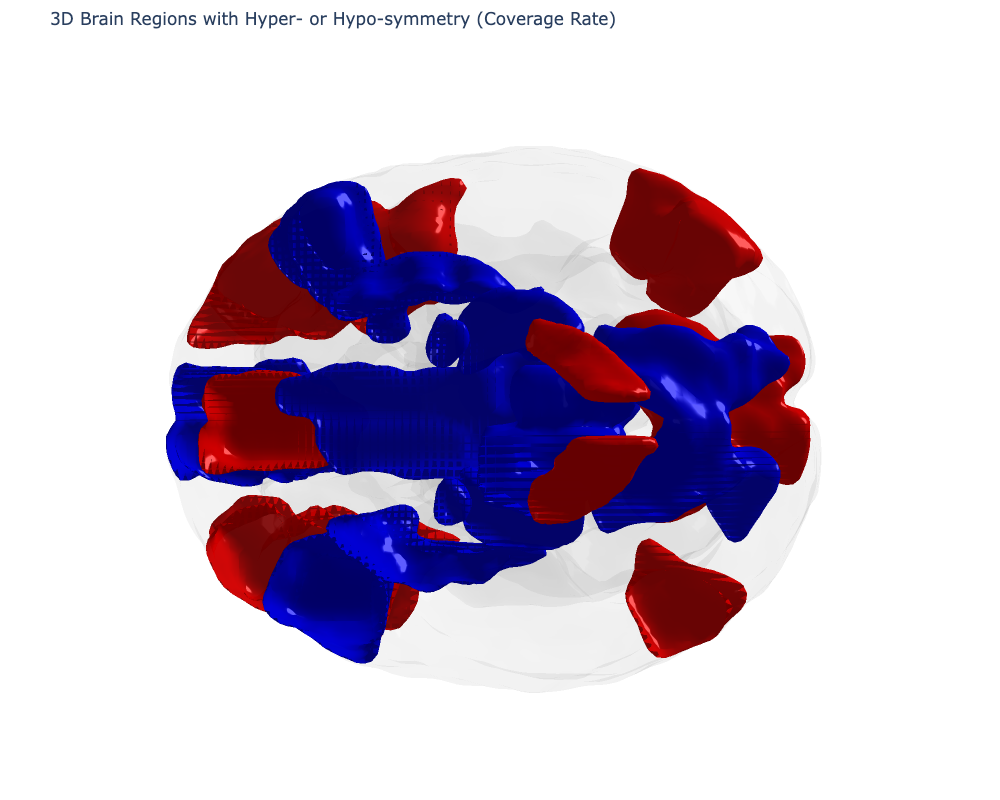}
    \caption{Anatomical topography of brain regions with significant discrepancy between ASD and control groups, based on coverage rate.
    Hypo-symmetric regions are marked in blue, while hyper-symmetric regions are in red.
    The first image shows the view from the left, while the second image shows the view from above.
    In both images, the front of the brain is on the left, and the rear is on the right.}
    \label{fig:ABIDE.C.brain.image}
    
\end{figure}

\begin{table}[htbp]
\centering
\begin{singlespacing}
\renewcommand{\arraystretch}{1.2}
\caption{The 54 symmetric pairs of brain regions and Wilcoxon rank-sum test results based on coverage rates.
The alternative hypothesis is that the distribution of coverage rates differs between the ASD and Control groups.
Bold values indicate a significant change in symmetry pattern after Benjamini--Yekutieli procedure.
\emph{Hyper} indicates increased symmetry and strengthened interconnectivity.
\emph{Hypo} indicates decreased symmetry and weakened interconnectivity.
}
\label{tab:ABIDE.wilcoxon.C}
\small
\setlength{\tabcolsep}{2pt}
\begin{tabular}{clccc|clccc}
\toprule
\textbf{$k$} & \textbf{Region Pair} & {$p_{\text{raw}}$} & {$p_{\text{BY}}$} & \textbf{Status} & \textbf{$k$} & \textbf{Region Pair} & {$p_{\text{raw}}$} & {$p_{\text{BY}}$} & \textbf{Status} \\
\midrule
1 & Precentral & 0.473 & 1.000 & - & 28 & Fusiform & 0.234 & 1.000 & - \\
2 & Frontal\_Sup & 0.156 & 1.000 & - & 29 & Postcentral & 0.014 & 0.205 & - \\
3 & Frontal\_Sup\_Orb & 0.203 & 1.000 & - & 30 & Parietal\_Sup & 0.029 & 0.330 & - \\
4 & Frontal\_Mid & \textbf{0.002} & \textbf{0.039} & Hyper & 31 & Parietal\_Inf & 0.075 & 0.712 & - \\
5 & Frontal\_Mid\_Orb & 0.417 & 1.000 & - & 32 & SupraMarginal & 0.025 & 0.289 & - \\
6 & Frontal\_Inf\_Oper & 0.237 & 1.000 & - & 33 & Angular & \textbf{0.001} & \textbf{0.016} & Hyper \\
7 & Frontal\_Inf\_Tri & \textbf{0.000} & \textbf{0.000} & Hypo & 34 & Precuneus & 0.047 & 0.484 & - \\
8 & Frontal\_Inf\_Orb & 0.020 & 0.264 & - & 35 & Paracentral\_Lob & \textbf{0.000} & \textbf{0.001} & Hypo \\
9 & Rolandic\_Oper & 0.849 & 1.000 & - & 36 & Caudate & 0.079 & 0.720 & - \\
10 & Supp\_Motor & \textbf{0.000} & \textbf{0.008} & Hypo & 37 & Putamen & 0.830 & 1.000 & - \\
11 & Olfactory & 0.017 & 0.228 & - & 38 & Pallidum & 0.088 & 0.780 & - \\
12 & Frontal\_Sup\_Med & \textbf{0.000} & \textbf{0.009} & Hypo & 39 & Thalamus & 0.714 & 1.000 & - \\
13 & Frontal\_Med\_Orb & \textbf{0.000} & \textbf{0.001} & Hyper & 40 & Heschl & 0.548 & 1.000 & - \\
14 & Rectus & 0.052 & 0.509 & - & 41 & Temporal\_Sup & 0.768 & 1.000 & - \\
15 & Insula & \textbf{0.000} & \textbf{0.001} & Hypo & 42 & Temporal\_Pole\_Sup & 0.503 & 1.000 & - \\
16 & Cingulum\_Ant & 0.004 & 0.060 & - & 43 & Temporal\_Mid & 0.205 & 1.000 & - \\
17 & Cingulum\_Mid & \textbf{0.001} & \textbf{0.015} & Hypo & 44 & Temporal\_Pole\_Mid & 0.751 & 1.000 & - \\
18 & Cingulum\_Post & 0.749 & 1.000 & - & 45 & Temporal\_Inf & 0.281 & 1.000 & - \\
19 & Hippocampus & \textbf{0.001} & \textbf{0.012} & Hypo & 46 & Cerebellum\_Crus1 & 0.023 & 0.278 & - \\
20 & ParaHippocampal & 1.000 & 1.000 & - & 47 & Cerebellum\_Crus2 & 0.259 & 1.000 & - \\
21 & Amygdala & \textbf{0.000} & \textbf{0.008} & Hypo & 48 & Cerebellum\_3 & 0.206 & 1.000 & - \\
22 & Calcarine & \textbf{0.000} & \textbf{0.001} & Hyper & 49 & Cerebellum\_4\_5 & \textbf{0.000} & \textbf{0.008} & Hyper \\
23 & Cuneus & 0.178 & 1.000 & - & 50 & Cerebellum\_6 & 0.762 & 1.000 & - \\
24 & Lingual & \textbf{0.000} & \textbf{0.009} & Hypo & 51 & Cerebellum\_7b & 0.034 & 0.361 & - \\
25 & Occipital\_Sup & 0.006 & 0.095 & - & 52 & Cerebellum\_8 & 0.232 & 1.000 & - \\
26 & Occipital\_Mid & 0.816 & 1.000 & - & 53 & Cerebellum\_9 & 0.749 & 1.000 & - \\
27 & Occipital\_Inf & 0.323 & 1.000 & - & 54 & Cerebellum\_10 & 0.320 & 1.000 & - \\
\bottomrule
\end{tabular}
\end{singlespacing}
\end{table}

\end{appendix}

\end{document}